\UseRawInputEncoding
\documentclass[11pt]{imsart}
\RequirePackage{amsthm,amsmath,amsfonts,amssymb}
\RequirePackage[numbers]{natbib}
\usepackage{enumerate}
\usepackage{color}

\usepackage{sectsty}

\usepackage{apptools}

\RequirePackage{graphicx}

\startlocaldefs

\AtAppendix{\counterwithin{theorem}{subsection}}

\usepackage{xcolor}

\RequirePackage[colorlinks,citecolor=blue,urlcolor=blue,backref=page]{hyperref}

\usepackage[bottom=0.4cm]{geometry}
\newtheorem{theorem}{Theorem}[section]
\newtheorem{lemma}[theorem]{Lemma}
\newtheorem{proposition}[theorem]{Proposition}
\newtheorem{corollary}[theorem]{Corollary}
\theoremstyle{remark}
\newtheorem{definition}[theorem]{Definition}
\newtheorem{remark}[theorem]{Remark}
\newtheorem{example}{Example}

\newcommand{\nc}{\newcommand}
\nc{\I}{{\mathbf 1}}
\newcommand{\remove}[1]{}

\nc{\bN}{{\mathbf N}}
\nc{\cB}{{\mathcal B}}
\nc{\cF}{{\mathcal F}}
\nc{\cX}{{\mathcal X}}
\nc{\cS}{{\mathcal S}}
\nc{\cC}{{\mathcal C}}
\nc{\cN}{{\mathcal N}}
\nc{\cK}{{\mathcal K}}
\nc{\R}{{\mathbb R}}
\nc{\N}{{\mathbb N}}
\nc{\Z}{{\mathbb Z}}
\nc{\BX}{{\mathbb X}}
\nc{\BM}{{\mathbb M}}
\nc{\BY}{{\mathbb Y}}
\nc{\bx}{\mathbf{x}}
\nc{\Binf}{B^{\infty}}
\nc{\teta}{\tilde{\eta}}
\nc{\cO}{\mathcal{O}}
\nc{\cV}{\mathcal{V}}
\nc{\tx}{\tilde{x}}

\def\1{\mathbf{1}}
\newcommand{\md}{\mathrm{d}}

\DeclareMathOperator{\supp}{supp}
\nc{\BP}{\mathbb{P}}
\nc{\BE}{\mathbb{E}}
\nc{\BQ}{\mathbb{Q}}
\DeclareMathOperator{\BV}{{\mathbb Var}}
\DeclareMathOperator{\BC}{{\mathbb Cov}}

\nc{\dy}[1]{\textcolor{magenta}{#1}}
\nc{\rev}[1]{\textcolor{blue}{#1}}
\nc{\todo}[1]{\textcolor{red}{#1}}

\numberwithin{equation}{section}

\endlocaldefs 

\usepackage{fancyhdr}
\begin{document} 

\begin{frontmatter}

\title{Phase transitions and noise sensitivity\\ on the Poisson space \\ via stopping sets and decision trees}

\begin{aug}
\author[A]{\fnms{G\"unter} \snm{Last}\ead[label=e1]{guenter.last@kit.edu}},
\author[B]{\fnms{Giovanni} \snm{Peccati}\ead[label=e2,mark]{giovanni.peccati@uni.lu}}
\and
\author[C]{\fnms{D.} \snm{Yogeshwaran}\ead[label=e3,mark]{d.yogesh@isibang.ac.in}}
\address[A]{Institute for Stochastics,  Karlsruhe Institute of Technology,  Karlsruhe. 
\printead{e1}}

\address[B]{Department of Mathematics (DMATH), Luxembourg University, Luxembourg. 
\printead{e2}}

\address[C]{Theoretical Statistics and Mathematics Unit, Indian Statistical Institute, Bangalore. 
\printead{e3}}
\end{aug}


\date{\today}
\maketitle
%



\begin{abstract}\noindent Proofs of sharp phase transition and noise sensitivity in percolation 
have been significantly simplified by the use of randomized algorithms, via the
  {\em OSSS inequality}
  (proved by O'Donnell, Saks, Schramm and Servedio in \cite{OSSS05}) and the {\em Schramm-Steif inequality} for the Fourier-Walsh
  coefficients of functions defined on the Boolean hypercube (see
  \cite{SS10}). In this article, we prove intrinsic versions of the OSSS and Schramm-Steif inequalities
  for functionals of a general Poisson process, and apply these new estimates to deduce sufficient conditions --- expressed in terms of
  randomized stopping sets --- yielding sharp phase transitions, quantitative noise sensitivity,
  exceptional times and bounds on critical windows for monotonic Boolean Poisson functions. 
Our analysis is based on a new general definition of `stopping set', not requiring any topological property for the underlying measurable space, as well as on the new concept of a `continuous-time decision tree', for which we establish several fundamental properties. We apply our findings to the {\em $k$-percolation of the Poisson Boolean model} and to the Poisson-based {\em confetti
  percolation} with bounded random grains. In these two models, we
  reduce the proof of sharp phase transitions for percolation, and of noise sensitivity for crossing events, to the construction of suitable
  randomized stopping sets and the computation of one-arm probabilities at
  criticality. This enables us to settle an open problem suggested by Ahlberg, Tassion and Texeira \cite{Ahlbergsharp18} (a special case of which was conjectured earlier by Ahlberg, Broman, Griffiths and Morris \cite{ABGM14}) on noise sensitivity of crossing events for the planar Poisson Boolean model with random
balls whose radius distribution has finite $(2+\alpha)$-moments and also show the same for planar confetti percolation model with bounded
  random balls. We also prove that critical probability is $1/2$ for the
  planar confetti percolation model with bounded, $\pi/2$-rotation
  invariant and reflection invariant random grains. Such a result was conjectured by Benjamini and
  Schramm \cite{Benjamini98} in the case of fixed balls and proved
  by M\"{u}ller \cite{Muller17}, Hirsch \cite{Hirsch15} and Ghosh and
  Roy \cite{Ghosh18} in the case of balls, boxes and random boxes,
  respectively; our results contain all previous findings as special cases.
  \end{abstract}
  
  \smallskip



\end{frontmatter}  
  
  {
  \hypersetup{linkcolor=black}
  \tableofcontents
}

\remove{  \noindent{\bf Keywords:} Poisson process ; functional inequalities ; stopping sets ; noise sensitivity ; dynamical percolation ; sharp phase-transition. 
  
  \smallskip

}

\section{Introduction}
\label{s:intro}

\subsection{Overview and motivation}\label{ss:introoverview}

Functional inequalities for mappings on the discrete hypercubes $\{0,1\}^n$ and $\{-1,1\}^n$,  $n \in \{1,\ldots, \infty\}$, (typically endowed with some product
measure) play a pivotal role in many applications, ranging from the mathematical theory of voting and social choice, to computational complexity and percolation on lattices; see e.g. \cite{GarbanSteif14, ODonnell14, ODICM, vHandel2016}. These estimates -- examples of which are the discrete Poincar\'e and logarithmic Sobolev inequalities (see e.g. \cite[Theorem 1.13]{GarbanSteif14} and \cite[Lemma 8.15]{vHandel2016}) and Talagrand's $L^1$-$L^2$ bound \cite{Cordero2012chapter, Talagrand94} --  are analytic in nature and can be canonically framed in the language of Markov
semigroups, see e.g. \cite{Cordero2012chapter, Nourdin2020}. 

More than a decade ago, two striking estimates --- one for the variance of $f : \{0,1\}^n \to \{0,1\}$ and the other for the Fourier-Walsh coefficients of $f : \{-1,1\}^n \to \{-1,1\}$
--- were proved by O'Donnell, Saks, Schramm and Servedio \cite[Theorem 3.1]{OSSS05}
and Schramm and Steif \cite[Theorem 1.8]{SS10}, respectively. The former inequality
is known as the {\em OSSS inequality} and we shall refer to the latter as the
{\em Schramm-Steif inequality}. Both the OSSS and Schramm-Steif estimates require the existence of a randomized
algorithm determining the function $f$ (although this assumption can be relaxed in the Schramm-Steif case, see Section \ref{ss:consid} below), in such a way that the upper bounds are expressed in
terms of {\em revealment probabilities}, i.e., of the probability that a bit/coordinate of the input
is revealed by the algorithm.  While the Schramm-Steif inequality was
motivated by quantitative noise sensitivity for discrete percolation
models, the OSSS inequality was established in the context of decision tree
complexity. More recently, references \cite{DCRT19,DCRT18,DCRT19b} have pioneered the use of the OSSS inequality to
derive simple proofs of sharp phase transitions in percolation. 

The aim of this paper is to establish bounds analogous to the OSSS and Schramm-Steif inequalities for random variables that depend on Poisson processes defined on abstract measurable spaces, see Theorem \ref{t:POSSS}, Theorem \ref{t:condwitobd} and Corollary \ref{c:sspoisson} below. As demonstrated by the applications discussed in the forthcoming Section \ref{ss:introapp}, our estimates are perfectly tailored for studying phase transitions and noise sensitivity in models of {\em continuum percolation} (see e.g. \cite{Bollobas06, Hall85, MeesterRoy96}), and will allow us to settle some open problems from \cite{ABGM14,Ahlbergsharp18,Benjamini98}. 

One new theoretical insight developed in our work is that, in the context of functionals of random point measures, the role of randomized algorithms for Boolean inputs is naturally played by {\em randomized stopping sets} and {\em continuous-time decision trees} (CTDTs). Although the notion of `stopping set' is classical in stochastic geometry (see e.g. \cite{BaumstarkLast09, Molchanov05, Privault15, Zuyev99, Zuyev06}), in the sections to follow we choose to adopt a more general definition of such an object, that will allow us to directly work with Borel spaces not necessarily verifying specific topological requirements. To the best of our knowledge, such a general theory of stopping sets (which is presented in a mostly self-contained way in Appendix \ref{appendix1} and has a clear independent interest) is developed here for the first time. The notion of CTDT is also new, and will be studied from scratch in Section \ref{secCTDT} and Appendix \ref{appendix2}. 

It is important to observe that the proofs of Theorem \ref{t:POSSS}
and Theorem \ref{t:condwitobd} below {\em do not} rely on semigroup
techniques or on discretization schemes, but are rather based on the
use of the Markov property of stopping sets -- as proved in Theorem
\ref{tAstopping2} -- combined with the classical characterization of
Poisson processes as completely independent random measures. In
particular, for now it does not seem possible to deduce our main
findings from classical functional inequalities on the Poisson space,
such as those proved in \cite{APS20, LPS, LastPenrose17, Peccati2016,
  Nourdin2020, W}. More to the point, when compared with the tools
available for mappings on Boolean hypercubes, the relevance of our
estimates is amplified by the fact that $L^1$-$L^2$
inequalities on the Poisson space only hold under very restrictive
assumptions, see \cite{Nourdin2020}. Also, even if $L^1$-$L^2$
inequalities on the Poisson space were available in full generality,
it would not be straightforward to deduce from them sharp
phase-transitions (see the discussion below (2.12) in
\cite{DCIsingnotes}). Remarkably, our analysis will also show that
{\em multiple Wiener-It\^o integrals} and associated chaotic
decompositions (see \cite{Last16, LastPenrose17, PTbook}) play
  a role similar to Fourier-Walsh expansions (associated with mappings
  on hypercubes) for establishing quantitative noise sensitivity in
  various contexts, see Section \ref{s:quantns}. Our method of proof
should eventually be contrasted with the original proof of the
discrete OSSS inequality given in \cite{OSSS05}, which is rather based
on martingale estimates reminiscent of the arguments leading to the
discrete Poincar\'e inequality.

\subsection{A closer look at our main bounds}\label{ss:consid}
 In order to motivate the reader, we now provide a more formal discussion of some of the new estimates proved in this paper. We refer to Section
  \ref{ss:intropoisson} for a rigorous definition of the objects
 appearing below.

Let $\eta$ be a {\em Poisson process} on a Borel space
$\BX$ (i.e.,  $\BX$ is a measurable space with a Borel-measurable bijection to a Borel subset of $[0,1]$ with a measurable inverse),  with a locally finite intensity measure
$\lambda$, and consider a random variable of the form $f(\eta)$, where $f$ is some measurable mapping and $\BE[f(\eta)^2]<\infty$. In this case, the variance of $f(\eta)$ can be bounded by using the classical Poincar\'e inequality stated in  \cite[Theorem
18.7]{LastPenrose17}, according to which
\begin{align}\label{poincare}
\BV[f(\eta)]\le \int \BE[|D_xf(\eta)|^2]\,\lambda(dx),
\end{align}
where, for $x\in\BX$, the mapping $D_xf$ is the {{\it add-one cost operator}} given by 
 $D_xf(\eta):=f(\eta+\delta_x)-f(\eta)$, with $\delta_x$ the Dirac mass at $x$, and $\BV$ indicates the variance. The estimate \eqref{poincare} is an infinite-dimensional counterpart to the well-known {\it Efron-Stein inequality} (see e.g. \cite[Chapter 3]{BLM}), stating that, if $F(X)$ is a square-integrable functional of a finite vector of independent random variables $X = (X_1,...,X_n)$, then
 \begin{equation}\label{e:es}
\BV(F(X)) \leq \frac12 \sum_{i=1}^n \BE[ (F(X) - F(X^i))^2],
\end{equation}
where $X^i$ is the vector obtained from 
$X$ by re-randomizing the $i$th coordinate $X_i$. Note that, if $f,F$ in the above discussion take values in $\{-1,1\}$, then $\BV(f(\eta)) = \BE[|f(\eta) - f(\eta')]]$ and $\BV(F(X)) = \BE[|F(X) - F(X')|]$, where $\eta'$ and $X'$ are independent copies of $\eta$ and $X$, respectively, and also $\BE[|D_xf(\eta)|^2] = 4 \BP(f(\eta) \neq f(\eta+\delta_x))$ and $\BE[ (F(X) - F(X^i))^2] = 4\BP( F(X)\neq F(X^i))$. It is easily checked that, if $X$ takes vaules in 
$\{-1,1\}^n$, then the latter probability coincides (up to a factor) with the usual notion of {\em influence} of the $i$th coordinate on $F$, as defined e.g. in \cite[Section 1.3]{GarbanSteif14} (see also Remark \ref{r:influ} below). 
Albeit fundamentally useful in many situations (see again \cite[Chapter 3]{BLM}, as well as \cite{HR09} for a representative geometric application), the estimates \eqref{poincare}--\eqref{e:es} are in general suboptimal, and not sufficiently tight for several applications. See \cite{Nourdin2020} (and Example \ref{exemptyspace} below) for several illustrations of this point. 

Now consider a {\em stopping set} $Z(\eta)$. This notion, which is formally defined in Appendix \ref{appendix1}, indicates that $Z(\eta)$ is a random subset of $\BX$, satisfying some adequate measurability conditions and such that $Z(\eta)$ only depends on the restricition $\eta_Z$ of $\eta$ to $Z(\eta)$. We say that a mapping $f(\cdot)$ is {\em determined} by $Z(\eta)$ if $f(\eta)$ only depends on $\eta_Z$. In Theorem \ref{t:POSSS} and Corollary \ref{c:sspoisson} below, we will prove the following two estimates \eqref{e:osss_intro} and \eqref{e:ss_intro}, allowing one to control the fluctuations of a random variable $f(\eta)$ such that $f$ is determined by $Z$.

\begin{enumerate}[\bf (a)]

\item Assume that the measure $\lambda$ is diffuse and that the stopping set $Z(\eta)$ is the output of a sufficiently regular continuous-time decision tree (CTDT), that is,
of a sequence of increasing stopping sets $\{Z_t(\eta) : t\geq 0\}$ such that
$Z_t(\eta) \uparrow Z(\eta)$ (see Section \ref{s:CTDT}). Then, if the random variable $f(\eta)$ is integrable and $f$ is determined by $Z(\eta)$,
 \begin{align}
\label{e:osss_intro}
\BE[|f(\eta)-f(\eta')|]\le 2 \int \BP(x\in Z)\BE[|D_xf(\eta)|]\,\lambda(dx).
\end{align}
Relation \eqref{e:osss_intro} is an infinite-dimensional counterpart to \cite[Theorem 3.1]{OSSS05}, implying the following bound: if $F(X)$ is an integrable functional of a vector of independent random variables and $F$ is determined by a {\em randomized algorithm} $T$, then
\begin{equation}\label{e:trueosss}
\BE[| F(X) - F(X')|]\leq \sum_{i=1}^n \BP( i \mbox{ is revealed by} \,\, T)\, \BE[| F(X) - F(X^i)]],
\end{equation}
where we have used the previously introduced notation, and the event
inside the probability corresponds to the event that the coordinate
$i$ is visited by $T$. In the applications developed in { Sections
\ref{s:kpercbm} and \ref{s:confettiperc}} we will show that
\eqref{e:osss_intro} and its generalizations can be used to prove
sharp thresholds in continuum percolation models. In Appendix
\ref{appendix2}, we will also show that not every stopping set
$Z(\eta)$ can be represented as the output of a regular CTDT; as
discussed below, the extension of \eqref{e:osss_intro} to general
stopping sets is an open problem.

\item Consider a stopping set $Z(\eta)$ (not necessarily generated by a CTDT), and assume that $f(\eta)$ is square-integrable and $f$ is determined by $Z(\eta)$. Then, if $u_k(\eta)$ denotes the $k$th summand in the {\em Wiener-It\^o chaotic decomposition} of $f(\eta)$ ($k=1,2,...$) (see \eqref{e:chaos} for definition),
\begin{align}
\label{e:ss_intro}
\BE[ u_k(\eta)^2] \leq k \left[ \sup_{x\in \BX}\BP( x\in Z(\eta)) \right] \times \BV(f(\eta)).
\end{align}
The bound \eqref{e:ss_intro} is a counterpart to \cite[Theorem A.1]{OSSS05} (of which \cite[Theorem 1.8]{SS10} is a special case), yielding the following: if $F(X)$ is a a square-integrable functional of a vector of independent random variables and $F$ is determined by a {\em discrete stopping set} $S$, then, denoting by $F_k$ the $k$th component in the Hoeffding-ANOVA decomposition of $F$ ($k=1,...,n$) (see e.g. \cite[Chapter 5]{serfling}),
\begin{equation}\label{e:truess}
\BV( F_k(X) )\leq k\left[ \max_{i=1,...,n} \BP( i \in S) \right] \times \BV(F(X)).
\end{equation}
Observe that, if the random set $S$ is the collection of all
coordinates of $X$ revealed by a randomized algorithm, then
$\BP( i \in S) = \BP( i \mbox{ is revealed})$, but such an additional
feature is not required for \eqref{e:truess} to hold. In Sections
\ref{s:kpercbm}, Section \ref{s:confettiperc} and Section 
\ref{s:pbmunbdd}, we will show that \eqref{e:ss_intro} and its
extensions are useful for establishing noise sensitivity in continuum
percolation. The crucial technical step for deriving the original
Schramm-Steif inequality \eqref{e:truess} is a bound on the
conditional expectations of degenerate $U$-statistics (see
\cite[formula (12)]{OSSS05}). In order to prove \eqref{e:ss_intro}, we
will derive an analogous bound for multiple Wiener-It\^o integrals
(Theorem \ref{t:condwitobd}) which is of independent interest.
\end{enumerate}

\begin{remark}\label{r:influ}{\rm Let $F: \{-1,1\}^n \to \{-1,1\} $,
    and let $X = (X_1,...,X_n)$ be a collection of i.i.d.\ Rademacher
    random variables with parameter $p$. One remarkable consequence of
    \cite{OSSS05} (see Corollary 1.2 therein) is the general
    bound \begin{equation}\label{e:swsw} {\bf Inf}_{max}(F) \geq
      \frac{ \BV{F(X)}}{\Delta(T)}, \end{equation} where
    $${\bf Inf}_{max}(F) := \max_{i=1,...,n} {\bf Inf}_i(F):=
    4p(1-p)\max_{i=1,...,n} \BP[F(X)\neq F(X^{(i)} ) ],$$ $X^{(i)}$ is
    the element of $\{-1,1\}^n$ obtained by switching the sign of the
    $i$th coordinate of $X$, and $\Delta(T)$ is the average number of
    coordinates visited by $T$. As mentioned above, the quantity
    ${\bf Inf}_i(F)$ is defined as the {\it influence} of the $i$th
    coordinate on $F$. One can derive an estimate with the same
    flavour starting from \eqref{e:osss_intro}. Consider indeed
    $f(\eta)$ with values in $\{-1,1\}$ and determined by the output
    $Z$ of a CTDT, and set
$$
{\bf Inf}_{max}(f) := 2 \sup_{x\in \BX} \, \BP( f(\eta)\neq f(\eta+\delta_x)).
$$
Then, one infers from \eqref{e:osss_intro} that
\begin{equation}\label{e:sw}
2{\bf Inf}_{max}(f)\geq \frac{\BV(f(\eta))}{\BE[\lambda(Z)]},
\end{equation}
where we have applied a Fubini argument. The full analogy with \eqref{e:swsw} is obtained from \eqref{e:number_stopset}, yielding the identity $\BE[\lambda(Z)] = \BE[\eta(Z)]$, that is, the denominator on the right-hand side of \eqref{e:sw} coincides with the average number of Poisson points that are revealed by the stopping set $Z$.

}
\end{remark}

We will now provide an outline of the applications developed in 
Sections \ref{s:kpercbm}, \ref{s:confettiperc} and
  \ref{s:pbmunbdd}.

\subsection{Overview of applications}\label{ss:introapp}

In Sections \ref{s:kpercbm}, \ref{s:confettiperc} and
  \ref{s:pbmunbdd}, we will develop three applications of our
abstract results to continuum percolations models, namely, to $k$-{\em
  percolation} in the Poisson Boolean model (see
  e.g. \cite{BY2013} and the references therein), 
to {\em confetti  percolation} (see
\cite{Hall85,MeesterRoy96,Bollobas06,Benjamini98,Ahlbergsharp18}) and
  to the planar Poisson Boolean model with 
{\it unbounded  grains} (see \cite{Ahlbergsharp18}). These are by no means
  exhaustive, and more potential applications are described in Section
  \ref{s:furthapplns}.

\smallskip

\noindent\underline{\it General considerations}. We first observe that the finite-dimensional nature of the original OSSS and Schramm-Steif inequalities has not been --- in general --- a major obstacle for applying them to random geometric models based on infinite-dimensional inputs. Indeed, using suitable discretization schemes and
some extra technical work, these estimates have been successfully
applied to geometric models based on stationary Euclidean Poisson
point processes, see
\cite{ABGM14,Ahlberg18,DCRT18,DCRT19b,Faggionato2019,Ghosh18,LRM2019}. This
fact notwithstanding, the applications developed in the present paper
indicate that the our intrinsic approach is a new valuable tool for
establishing quantitative noise sensitivity, existence of exceptional
times and sharp phase transition in (possibly dynamical) continuum
percolation models, under minimal assumptions and by using strategies
of proofs that only require one to prove non-degenaracy of some class
of functionals and decay on arm probabilities. Another important point
is that our estimates allow one to prove quantitative noise
sensitivity in percolation models for which discretization techniques
would be difficult to apply. One typical example of this situation is
the confetti percolation discussed in Section \ref{s:confettiperc},
whose dynamical nature makes discretization procedures particularly
challenging to implement (see e.g. \cite{Ghosh18}), and a
similar remark also applies for the  Poisson Boolean model with unbounded
grains as in Section \ref{s:pbmunbdd}. In general, while the
  discretization approach has been widely applied to the detection of
  sharp thresholds, it has been exploited less intensively in the
  proof of noise sensitivity (see \cite{Ahlberg18} for some
  distinguished examples). We believe that the our intrinsic approach
will also be helpful when dealing with more general Poissonian
percolation models such as those suggested in Section
\ref{s:furthapplns} as well as in the study of Poisson-Voronoi
tesselations on the hyperbolic plane \cite{Benjamini2001} or the
Poisson cylinder model \cite{Tykesson2012}.

\smallskip

\noindent\underline{\it Description of the models}. The $k$-percolation model arises by placing  i.i.d.\ bounded grains (shapes) at each point in the support of a stationary Poisson point process on $\R^d, \, d \geq 2$ with intensity $\gamma$. The $k$-covered region is the region covered by at least $k$ of these grains and
$k$-percolation refers to existence of an unbounded connected
component in the $k$-covered region. If $k = 1$, we obtain the
standard Poisson Boolean model. The confetti percolation model is
defined in terms of a space-time Poisson point process. The Poisson
points are independently coloured either in black or white, with
probability $p$ and $1-p$, respectively. The black and white points
carry i.i.d.\, bounded grains as well and can have different
distributions. A spatial location $x \in \R^d$ is covered in black or
white depending on the colour of the first grain that falls on
$x$. The black region is the region of black points and confetti
percolation refers to existence of an infinite component in the black
region. Apart from boundedness, we require that the grains are not too
thin i.e., contain some ball with positive probability. We give more
details on these two models as well as some more background literature
in the corresponding sections. The varying parameters in the
$k$-percolation model is $\gamma$, whereas it is $p$ in the confetti
percolation model.

\smallskip

\noindent\underline{\it Overview of results: phase transitions}. In both models, we first prove a sharp phase transition in
arbitrary dimensions (Theorems \ref{t:bmexpdecay} and \ref{t:confettisharp}). This involves showing exponential decay for the radius of the component 
containing the origin in the sub-critical regime and the mean-field lower bound for the radius in the
super-critical regime. The proof uses the Poisson OSSS variance inequality \eqref{e:osss_intro}
and the necessary CTDTs are constructed via a technical adaptation of
the ideas of Duminil-Copin, Raoufi and Tassion \cite{DCRT19} to the
continuum setting.  

We now describe some consequences of the sharp phase transition result
derived above. In the $k$-percolation model, we show equality of the
usual critical intensity and of the critical intensities defined in
terms of the expected volume of the component containing the origin
and of the non-triviality of certain crossing probabilities introduced
by Gou\'{e}r\'{e} in \cite{Gouere14} (see Theorem
\ref{t:eqcritprob}). In the confetti percolation model, we prove that
the critical parameter is $1/2$ in a {very general planar confetti
  percolation model in the general setting of bounded grains
  satisfying $\pi/2$-rotation invariance and invariance under
  reflection by coordinates (see Theorem
  \ref{t:critprobconfplanar}). This result was conjectured by
  Benjamini and Schramm \cite[Problem 5]{Benjamini98} for the case of
  deterministic balls. It was first proven in the case of deteministic
  boxes by Hirsch \cite{Hirsch15} and later by M\"{u}ller
  \cite{Muller17} in the case of deterministic balls and then by Ghosh
  and Roy \cite{Ghosh18} in the case of bounded random boxes.

\smallskip

\noindent\underline{\it Overview of results: noise sensitivity}. The
Schramm-Steif inequality \eqref{e:truess} arose in the study of {\it
  noise sensitivity} for lattice percolation models. Informally noise
sensitivity of a sequence of functions $f_n$, refers to the phenomenon
that under a small perturbation of the random input, the functions
become asymptotically independent or de-correlated. Apart from the
natural motivation coming from the study of random structures under
perturbations, there are varied reasons to study noise sensitivity
coming from statistical physics and computer science (see
\cite[Section 1.2]{Garban2011}, \cite[Chapter 12]{GarbanSteif14}). By
considering the perturbation as induced by a Markov process, one may
naturally relate noise sensitivity to the study of functions of a
dynamically evolving random input. In percolation theory, these are
known as {\it dynamical percolation models} (see
\cite{Haggstrom1997,Steif2009}) where one studies percolation models
such that the random input is varied dynamically. In the case
of Poisson-based models, we consider the dynamics induced by
the spatial birth-death process i.e., the existing points are deleted
after exponential lifetimes and new points are born independently and
stay for exponential lifetimes. The corresponding perturbation is to
delete a small fraction of existing points and compensate it by adding
a small fraction of new points independently. From the Poisson
Schramm-Steif inequality \eqref{e:ss_intro} and Wiener-It\^{o} chaos
decomposition \eqref{e:chaos}, we can obtain easily that a sequence of
$\{0,1\}$-valued functions $f_n$ is noise sensitive if there exists a
corresponding sequence of randomized stopping sets $Z_n$ determining
$f_n$ such that the supremum of the revealement probabilities
$\sup_{x \in \BX}\BP(x \in Z_n)$ vanishes asymptotically. Quantitative
noise sensitivity can be used to prove existence of exceptional times
in dynamical models as well as to give bounds on the critical window
for phase transitions.

In the two continuum percolation models ($k$-percolation in the
Poisson Boolean model and confetti percolation), we take $f_n$ to be
the indicator functions that there is a crossing of large boxes (say
$ [0,\kappa n] \times [0,n]^{d-1}$ for $\kappa > 0$) along the
$x$-axis by occupied components at criticality. We prove that noise
sensitivity and exceptional times in dynamical percolation for the
above crossing events follow from the non-degeneracy of the crossing
probabilities and decay of one-arm probabilities (Theorems
\ref{t:nscritbm} and \ref{t:nscritconf}). This helps us to settle a
conjecture by Ahlberg, Broman, Griffiths and Morris \cite[Conjecture
9.1]{ABGM14} regarding noise sensitivity of the critical planar
Poisson Boolean model with bounded random balls straightforwardly
using the recent one-arm probability estimates in Ahlberg, Tassion and
Texeira \cite{Ahlbergsharp18} and also yields easily noise sensitivity
of the critical planar confetti percolation model where grains are
balls with bounded random radii (see Corollary \ref{c:nsplanarbm} and
\ref{c:nsplanarconf}). Further, we also prove quantiative noise
sensitivity of the planar Poisson Boolean model with random balls
whose radius distribution satisfies a finite $(2+\alpha)$-moment
condition (Theorem \ref{t:nsplbmunbdd}). This answers an open question
in Ahlberg, Tassion and Texeira \cite{Ahlbergsharp18}.  (As
  indicated in \cite{ABGM14}, a version of this result can in
  principle be obtained by combining the BKS randomized algorithm
  method from \cite{BKS99} with the estimates of \cite{Ahlbergsharp18}
  -- but it is not clear that our near-optimal moment condition on the
  random radii is preserved by this alternative strategy).  Less
quantitatively, we also prove noise sensitivity of the critical planar
confetti percolation with black and white particles both having the
same $\pi/2$-rotation invariant and reflection invariant bounded grain
distribution (Corollary \ref{c:nsplanarconfsymm}).

\subsection{Organization of the paper}\label{ss:introplan}

The paper is composed of three parts. Part \ref{p:OSSS} contains basic
notions about CTDTs (Section \ref{s:CTDT}), the statement and proof of
the OSSS inequality for Poisson functionals (Section \ref{s:POSSS}),
and several useful variants (Section \ref{s:variants_OSSS}). Part
\ref{p:quantns} focusses on the Schramm-Steif inequality (Section
\ref{s:sspoisson}), its consequences for quantiative noise sensitivity
and noise stability (Section \ref{s:quantns}) and their applications
to exceptional times and sharp phase transition (Section
\ref{s:dynbm}). Part \ref{p:applns} presents applications to
$k$-percolation (Section \ref{s:kpercbm}), confetti percolation
(Section \ref{s:confettiperc}), noise sensitivity for the planar
Boolean model with unbounded balls (Section \ref{s:pbmunbdd}) and a
discussion of further applications (Section
\ref{s:furthapplns}). Appendix \ref{appendix1} contains a general
theory of stopping sets, that is used throughout the paper, under
minimal topological assumptions. Appendix \ref{appendix2} is devoted
to zero-one laws for CTDTs and non-attainable stopping sets.

{The next Subsection \ref{ss:intropoisson} contains some basic
  definitions and facts about Poisson processes. The reader is referred to the monographs \cite{Kallenberg17,LastPenrose17} for a comprehensive discussion; see also \cite{Last16}.}

\subsection{Some preliminaries on Poisson processes}\label{ss:intropoisson}

{Throughout the paper, we let $(\BX,\cX)$ denote a Borel space.
We fix an increasing sequence $\{B_n : n\geq 1\}$ 
 of sets in $\mathcal{X}$ satisfying $B_n\uparrow\BX$
and define $\cX_0\subset\cX$ as the system of all sets of the
form $B\cap B_n$ for some $B\in\cX$ and $n\in\N$.
(This is an example of {\em localizing ring} as defined in \cite{Kallenberg17}.) 
Note that $\sigma(\mathcal{X}_0) = \cX$. A measure $\nu$ on $(\BX, \cX)$
which is finite on $\cX_0$ is called {\em locally finite} 
(note that the property of being locally finite depends on the choice of the
sequence $\{B_n : n\geq 1\}$).}

{ Write $\N_0 := \{0,1,2,...\}$ and}
let $\bN\equiv \bN(\BX)$ be the space of all measures on $\BX$
which are $\N_0$-valued on $\cX_0$ and
let $\cN$ denote the smallest $\sigma$-field
such that $\mu\mapsto \mu(B)$ is measurable for
all $B\in\cX$. A {\em point process} is a random
element $\eta$ of $\bN$, defined over some fixed
probability space $(\Omega,\cF,\BP)$.
The {\em intensity measure} of $\eta$ is the
measure $\BE[\eta]$ defined by $\BE[\eta](B):=\BE[\eta(B)]$, $B\in\cX$.
By our assumption on $\BX$ every point process $\eta$ is
{\em proper}, that is
$ \eta = \sum_{n=1}^{\eta(\mathbb{X})}\delta_{Y_n}$, where $\delta_y$
is the Dirac mass at $y$, and $\{Y_n : n\geq 1\}$ is a collection of
random elements with values in $\mathbb{X}$; see \cite[Section 6.1]{LastPenrose17} 
for more details. Given a measure $\nu$ on $\BX$ and $B\in \mathcal X$, we write
  $\nu_B$ to indicate the trace measure $ A\mapsto \nu( A\cap  B)$.
For $\mu\in\bN$ and $A\in \cX$, we write $\mu\subset A$ if
  $\mu(A^c)=0$. 

Fix a locally finite measure $\lambda$ on $\BX$.
By construction of
$\cX_0$, $\lambda$ is automatically $\sigma$-finite. 
(Sometimes it is convenient to start with a $\sigma$-finite measure $\lambda$
and to choose the $B_n$ such that $\lambda(B_n)<\infty$.)
A {\em  Poisson process} with intensity measure $\lambda$ is a point process $\eta$
enjoying the following two properties: (i) for every
$B\in \mathcal{X}$, the random variable $\eta(B)$ has a Poisson
distribution with parameter $\lambda(B)$, and (ii) given $m\in\N$ and
disjoint sets
$B_1,...,B_m\in \mathcal{X}$, the random variables
$\eta(B_1),...,\eta(B_m)$ are stochastically independent. Property
(ii) is often described as $\eta$ being {\it completely independent}
or, equivalently, {\it independently scattered}. 

It is a well-known fact (see e.g. \cite[Chapter 2]{LastPenrose17}) that, 
if $\lambda$ is non-atomic, then $\BP$-a.s.\ 
every point in the support of $\eta$ is charged with mass 1.
In other words this means that $\BP(\eta\in\bN_s)=1$,
where $\bN_s$ is the set of all $\mu\in\bN$ with
$\mu(\{x\})\le 1$ for each $x\in\BX$.
Given $\mu\in \bN$ and $n\geq 1$, we denote by $\mu^{(n)}$ the
$n$th {\it factorial measure} associated with $\mu$, as defined
in \cite[Chapter 4]{LastPenrose17}.
If $\mu\in\bN_s$ then
$\mu^{(n)}$ is the measure on $(\BX^n, \mathcal{X}^{ n})$ obtained by
removing from the support of $\mu^n$ every point $(x_1,...,x_n)$ such
that $x_i = x_j$ for some $i\neq j$ (with $\mu^{(1)} = \mu$). We will
often make use of the following (multivariate) {\it Mecke formula}
(see \cite[Theorem 4.4]{LastPenrose17}). If $\eta$ is a Poisson
process with intensity $\lambda$ then, for all $n\geq 1$ and all
measurable mappings $f\colon\BX^n \times \bN \to [0,\infty)$,
\begin{align}\label{e:mecke}
 \BE&\left[ \int_{\BX^n}  f(x_1,...,x_n, \eta) \, \eta^{(n)}(dx_1,...,dx_n)   \right] \\ 
&\qquad\qquad =\BE\left[ \int_{\BX^n} f(x_1,...,x_n, \eta+\delta_{x_1} +\cdots +\delta_{x_n} ) 
\, \lambda^{n}(dx_1,...,dx_n)   \right], \notag
\end{align}
where $\BE$ denotes expectation with respect to $\BP$. A further crucial fact is that every measurable function
$f\colon \bN \to \R$ such that $\BE[f(\eta)^2] < \infty$ admits a unique
chaotic decomposition of the form
\begin{equation}\label{e:chaos}
f(\eta) = \sum_{k=0}^\infty I_k(u_k),
\end{equation}
where $I_0(u_0) = u_0: = \BE[f(\eta) ]$, $I_k$ indicates a multiple
Wiener-It\^o integral of order $k$ with respect to the compensated
measure $\hat{\eta} := \eta-\lambda$, the kernels
$u_k\in L^2(\lambda^k)$ ($k\geq 1$) are $\lambda^k$-a.e. symmetric,
and the series converges in $L^2(\BP)$. Multiple Wiener-It\^o
integrals are the exact equivalent of homogeneous sums in the
framework of random variables depending on an independent Rademacher
sequence (or, more generally, of Hoeffding-ANOVA decompositions for
functions of independent random vectors). We refer the reader to
\cite[Chapters 12 and 18] {LastPenrose17} {{} and \cite[Chapter
  5]{PTbook}} for more details. For future use, we record here the
classical orthonormality relation, valid for all $k,m\geq 0$ and all
pairs of a.e. symmetric kernels $u\in L^2(\lambda^k)$,
$v\in L^2(\lambda^m)$ (with the obvious identification
$L^2(\lambda^0) = \R$):
\begin{equation}\label{e:onreln}
\BE[I_k(u) I_m(v)] = \I\{k=m\}\,k!\, \langle u, v \rangle_{L^2(\lambda^k)}.
\end{equation}
Recall from \cite[(18.20)]{LastPenrose17} that, if  $f : \bN \to \R$ is as in \eqref{e:chaos},  then 
\begin{equation}
\label{e:lpformula}
u_k(x_1,\ldots,x_k) := \frac{1}{k!}\BE[D^k_{x_1,\ldots,x_k}f(\eta)],
\end{equation}
where $D^k$ is the iterated difference operator defined as $D^k_{x_1,\ldots,x_k}f(\eta) := D_{x_1}(D^{k-1}_{x_2,\ldots,x_k}f(\eta))$ with $D_x$ being the add-one cost operator defined below \eqref{poincare}.
Applying \eqref{e:chaos}, \eqref{e:onreln} and \eqref{e:lpformula}, yields the relation
\begin{equation}
\label{e:FSident}
\BE[f(\eta)^2] = \sum_{k=0}^{\infty} k! \|u_k\|^2_{L^2(\lambda^k)} = \BE[f(\eta)]^2 + \sum_{k=1}^{\infty}\int_{\BX^k}\frac{1}{k!}\BE[D^k_{x_1,\ldots,x_k}f(\eta)]^2\lambda(\md x_1)\ldots \lambda (\md x_k).
\end{equation}
Formula \eqref{e:FSident} will be used in Section \ref{s:sspoisson}.

\part{The OSSS Inequality for Poisson functionals}\label{p:OSSS}

Throughout this part, we work in the general setting of Section \ref{ss:intropoisson}. 
We fix a Poisson process $\eta$ with a locally finite intensity measure $\lambda$.

\section{Continuous time decision trees}\label{secCTDT}
\label{s:CTDT}

The following key definition relies on the concept of
a stopping set, discussed in Appendix \ref{appendix1}.

A family $\{ Z_t : t\in\R_+\}$ of stopping sets is called
a {\em continuous-time decision tree} (CTDT) if $Z_t \in \cX_0$
for $t\in\R_+$, $\BE[\lambda(Z_0)]= 0$ and 
if the following properties are satisfied:
\begin{align}\label{etr1}
Z_s &\subset Z_t,\quad s\le t,\\
\label{etr2}
Z_t&=\bigcap_{s>t}Z_s,\quad t\in\R_+.
\end{align}
If $\{ Z_t : t\in\R_+\}$ is a CTDT, we then define
\begin{align}
Z_\infty&:=\bigcup_{t\in\R_+} Z_t,\\
Z_{t-}&:=\bigcup_{s<t}Z_s,\quad t\in\R_+,
\end{align}
as well as $Z_{0-}:=\emptyset$. As monotone unions of stopping
sets, $Z_\infty$ and $Z_{t-}$ are also stopping sets; see Theorem \ref{tAstopping2}. Note that $\Z_{\infty}$ is not required to be $\mathcal{X}_0$-valued. The aim of this section is to prove some technical results
for CTDTs, laying the ground for the main OSSS-type estimate stated
in Theorem \ref{t:POSSS}. 

Given a CTDT $\{ Z_t : t\in\R_+\}$ we define
\begin{align}\label{e:txmu}
t(x,\mu):=\inf \{t\in\R_+:x\in Z_t(\mu)\},\quad x\in\BX,\,\mu\in\bN,
\end{align}
where $\inf\emptyset:=\infty$. Observe that, by \eqref{etr1}  and \eqref{etr2} we have
\begin{align}\label{e3.2}
x\in Z_{t(x,\mu)}(\mu)\setminus Z_{t(x,\mu)-}(\mu),\quad \text{if $t(x,\mu)<\infty$}.
\end{align}
Furthermore we can replace $\R_+$ on the right-hand side of \eqref{e:txmu}
by the non-negative rational numbers.
Hence we obtain from Proposition \ref{pA22} that
\begin{align}\label{e2.15a}
t(x,\eta)=t(x,\eta+\delta_x),\quad \BP\text{-a.s.},\, \lambda\text{-a.e.\ $x$}.
\end{align}
The following result is crucial.

\begin{proposition}\label{p2.1} Suppose that $\{Z_t:t\in\R_+\}$ is a CTDT
and let $\eta$ and $\eta'$ be two independent Poisson processes
with the same intensity measure $\lambda$.
Let $g\colon [0,\infty]\times\bN\to\R_+$ be measurable and $x\in\BX$. Then
\begin{align*}
\BE[g(t(x,\eta),\eta_{\BX\setminus Z_{t(x,\eta)}(\eta)}+\eta'_{Z_{t(x,\eta)}(\eta)})]
=\BE[g(t(x,\eta),\eta')].
\end{align*}
\end{proposition}

\bigskip

The proof of Proposition \ref{p2.1} requires the following (purely deterministic)
lemma.

\begin{lemma}\label{l2.2} Suppose that $\{Z_t:t\in\R_+\}$ is a CTDT. Then, $\mu\mapsto Z_x(\mu):=Z_{t(x,\mu)}(\mu)$ is a stopping set for each $x\in\BX$.
Moreover, for each $x\in\BX$ there exist $\cX_0$-valued stopping sets $Z_n'$, $n\in\N$,
such that $Z_n'\uparrow Z_x$.
\end{lemma}
\begin{proof}
We first need to settle measurability. It follows from
the right-continuity \eqref{etr2} and the graph-measurability
of $Z_t$ that
$(y,t,\mu)\mapsto \I\{y\in Z_{t}(\mu)\}$ is measurable.
Moreover, for each $s\in\R_+$ and all $(x,\mu)\in\BX\times\bN$ we have
that $t(x,\mu)\le s$ if and only if
$x\in Z_s(\mu)$. Therefore $(x,\mu)\mapsto t(x,\mu)$ is measurable,  
and consequently $(x,y,\mu)\mapsto \I\{y\in Z_{t(x,\mu)}(\mu)\}$ is a measurable
mapping. Let us now fix $x\in\BX$.
For $\mu\in\bN$ we use the shorthand notation $Z(\mu):=Z_{t(x,\mu)}(\mu)$.
To prove that $Z$ is a stopping set, we check \eqref{estopset}.
Let $\mu,\psi\in\bN$ such that $\psi(Z(\mu))=0$. 
We assert that
\begin{align}\label{en1}
t(x,\mu_{Z(\mu)}+\psi)=t(x,\mu).
\end{align}
Assume that $t(x,\mu)=\infty$ (implying $Z(\mu)=Z_\infty(\mu)$), which is equivalent to
$x\notin Z_\infty(\mu)$. By \eqref{estopset} this means that
$x\notin Z_\infty(\mu_{Z_\infty(\mu)}+\psi)$, so that
$t(x,\mu_{Z(\mu)}+\psi)=\infty$.
Assume now that $t:=t(x,\mu)<\infty$ and abbreviate $\nu:=\mu_{Z(\mu)}+\psi$.
Since $x\in Z_t(\mu)=Z_t(\mu_{Z_t(\mu)}+\psi)=Z_t(\nu)$ we observe that
$t(x,\nu)\le t$. Assume now that $t>0$. 
Then, 
\begin{align*}
x\notin Z_{t-}(\mu)=\bigcup_{s<t}Z_s(\mu)=\bigcup_{s<t}Z_s(\mu_{Z_t(\mu)}+\psi)
=Z_{t-}(\nu).
\end{align*}
It follows that $t(x,\mu)=t(x,\nu)$, from which one deduces \eqref{en1} 
and consequently the following chain of equalities:
\begin{align*}
Z(\mu_{Z(\mu)}+\psi)
&=Z_{t(x,\mu_{Z(\mu)}+\psi)}(\mu_{Z(\mu)}+\psi)
=Z_{t(x,\mu)}(\mu_{Z(\mu)}+\psi)\\
&=Z_{t(x,\mu)}(\mu_{Z_{t(x,\mu)}(\mu)}+\psi).
\end{align*}
Using \eqref{estopset} once again we obtain that
\begin{align*}
Z(\mu_{Z(\mu)}+\psi)=Z_{t(x,\mu)}(\mu)=Z(\mu).
\end{align*}
This proves that $Z$ is a stopping set.
 To prove the second assertion we take $n\in\N$ and define
\begin{align*}
t_n(x,\mu):=\inf \{t\in\R_+:x\in Z_{t\wedge n}(\mu)\},\quad x\in\BX,\,\mu\in\bN.
\end{align*}
Applying the previous result to the CTDT $\{Z_{t\wedge n}:t\in\R_+\}$ we see
that
\begin{align*}
Z_n'(\mu):=Z_{t_n(x,\mu)\wedge n}(\mu),\quad \mu\in\bN,
\end{align*}
defines a stopping set $Z_n'$ with values in $\cX_0$.
By definition, we have that $t_n(x,\mu)=t(x,\mu)$ if
$t(x,\mu)\le n$ and $t_n(x,\mu)=\infty$ otherwise.
As a consequence, we obtain that $t_n(x,\mu)\wedge n=t(x,\mu)\wedge n$
and therefore $Z_n'\uparrow Z$, as required.
\end{proof}

\begin{proof}[Proof of Proposition \ref{p2.1}] 
By Lemma \ref{l2.2},
$\mu\mapsto Z(\mu):=Z_{t(x,\mu)}(\mu)$  is a stopping set 
satisfying the assumptions of Theorem \ref{tAstopping2}.
Using the independence of $\eta$ and $\eta'$, one infers that
\begin{align}\label{e2.18}
\BE\big[g(t(x,\eta),\eta_{\BX\setminus Z(\eta)}+\eta'_{Z(\eta)})\big]
=\int \BE[g(t(x,\eta_{Z(\eta)}),\eta_{\BX\setminus Z(\eta)}+\mu_{Z(\eta_{Z(\eta)})})]\, \Pi_\lambda(d\mu),
\end{align}
where $\Pi_\lambda$ denotes the distribution of $\eta$
and where we have also used \eqref{en1} and \eqref{estopset} with $\psi=0$.
By virtue of \eqref{eMarkov2}, the right-hand side of \eqref{e2.18} equals
\begin{align*}
\iint \BE[g(t(x,\eta_{Z(\eta)}),\nu_{\BX\setminus Z(\eta)}+\mu_{Z(\eta_{Z(\eta)})})]
&\,\Pi_\lambda(d\nu)\, \Pi_\lambda(d\mu)\\
&=\int \BE[g(t(x,\eta_{Z(\eta)}),\mu)]\, \Pi_\lambda(d\mu),
\end{align*}
where the equality comes from the complete independence property
of a Poisson process (and again \eqref{etr5}).
By \eqref{en1}, the last term in the previous equality coincides with the right-hand side of 
the asserted identity.
\end{proof}

The next statement is used in the proof of our main estimates.

\begin{proposition}\label{p2.4} Suppose that $\{ Z_t : t\in\R_+\}$ is a CTDT. 
Let $\eta$ and $\eta'$ be two independent Poisson processes with intensity measure $\lambda$.
Let $m\in\N$ and $g\colon \bN^m \to\R_+$ be measurable. Then,
for all $t_1,\ldots,t_{m}\in[0,\infty]$ with
$t_1<\cdots<t_{m}$,
\begin{align} \label{e:eqd}
\BE[g(\eta_{Z_{t_1}(\eta)},\ldots,\eta_{Z_{t_m}(\eta)},\eta_{\BX\setminus Z_{t_m}(\eta)})]
=\BE[g(\eta_{Z_{t_1}(\eta)},\ldots,\eta_{Z_{t_m}(\eta)},\eta'_{\BX\setminus Z_{t_m}(\eta)})].
\end{align}
\end{proposition}
\begin{proof}
We apply Theorem \ref{tAstopping2} (or equivalently \eqref{eMarkov3})
with $Z=Z_{t_m}$. By \eqref{estopset} we have for each $i\in\{1,\ldots,m\}$
that $Z_{t_i}(\eta)=Z_{t_i}(\eta_{Z_{t_m}})$. The result follows.
\end{proof}

\section{The OSSS inequality for Poisson functionals}
\label{s:POSSS}

The main result of the section is an OSSS inequality for Poisson
  functionals, stated in the forthcoming Theorem \ref{t:POSSS}.

We will now list a set of assumptions and notational conventions.
We fix a locally finite measure $\lambda$ on $\BX$, and assume moreover that $\lambda$ is diffuse.

\begin{enumerate}[\bf (a)]
\item Let $\eta$ and $\eta'$ be independent Poisson processes on $\BX$
with intensity measure $\lambda$.
\item Let $\{ Z_t:t\in\R_+\} $ be a CTDT. As before, set 
$Z_\infty:=\cup_{t\in\R_+} Z_t$ and write $Z:=Z_\infty$. For $t\in[0,\infty]$, we use the notation
\begin{align}\label{e:def_zetat}
  \zeta_t\equiv\zeta_t(\eta,\eta'):=\eta_{Z(\eta)\setminus Z_t(\eta)}+\eta'_{Z_t(\eta)}+\eta'_{\BX\setminus Z(\eta)};
\end{align}
observe that $\zeta_0=\eta_Z+\eta'_{\BX\setminus Z}$.
\item \label{pointc}
Consider a mapping $f\colon\bN\to\R$ such that
  $\BE|f(\eta)|<\infty$ and $\{ Z_t : t\in \R_+\} $ 
{\it is a  CTDT determining $f$}, meaning that
\begin{align}\label{eCTDTf}
f(\mu)=f(\mu_{Z_\infty(\mu)}),\quad \mu\in\bN,
\end{align}
and moreover
\begin{align}\label{eCTDTf2}
f(\zeta_t)\overset{\BP}{\longrightarrow} f(\eta'), \quad \mbox{as}\,\, t\to\infty.
\end{align}
\item Assume that
\begin{align}\label{ea1}
\lambda(Z_t(\mu)\setminus Z_{t-}(\mu))=0,\quad \mu\in\bN,\,\,t\in\R_+
\end{align}
and
\begin{align}\label{ea2}
\BP(\text{$\eta(Z_t(\eta)\setminus Z_{t-}(\eta))\le 1$ for all $t\in\R_+$})=1.
\end{align}
\end{enumerate}

\begin{remark}\label{r3.3}
{\rm  If $Z$ is $\cX_0$-valued, then assumption \eqref{eCTDTf2}
holds for any function $f$. Indeed, in this case
we have that $\zeta_t\to\zeta_\infty$ in the discrete topology.
For analogous reasons, such an assumption also holds if
$f$ is supported by a set in $\cX_0$.}
\end{remark}

A CTDT verifying \eqref{ea1} is said to be $\lambda$-{\em continuous}. In order to simplify the notation, in the discussion
to follow we will sometimes (but not always) write $Z_t$, $Z$, ...,
instead of $Z_t(\eta), Z(\eta)$, and so on. The next statement is one
of the main results of the paper.

\begin{theorem}[OSSS inequality for Poisson functionals] \label{t:POSSS}
Let assumptions {\bf (a)}--{\bf (d)} above prevail. Then,
\begin{align}\label{e3.42}
\BE[|f(\eta)-f(\eta')|]\le 2\int \BP(x\in Z(\eta))\BE[|D_xf(\eta)|]\,\lambda(dx),
\end{align}
where $D_xf(\eta) = f(\eta+\delta_x)-f(\eta)$ is the previously introduced add-one cost operator at $x$.
\end{theorem}
A discrete time version of the above theorem which provides some
insights into the above inequality can be found in the first arxiv
version of the paper.

\begin{proof} 
A direct application of \eqref{estopset} yields
$Z(\eta_Z+\eta'_{Z^c})=Z$. By virtue of \eqref{eCTDTf},\eqref{edetermined}, \eqref{eCTDTf2} 
and $\BE[\lambda(Z_0)]=0$ we have 
that $f(\zeta_0)=f(\eta)$, $\BP$-a.s., and also
\begin{align}
\label{e:interpoln_f}
f(\eta) - f(\zeta_t) \stackrel{\BP} {\longrightarrow}f(\eta)-f(\eta')= f(\zeta_0)-f(\zeta_\infty), \quad t\to\infty. 
\end{align}
Since the sets $Z_t$ are $\cX_0$-valued and $\eta,\eta'$ are locally finite,
it follows from \eqref{etr1} and \eqref{etr2} that the function $t\mapsto \zeta_t$ is right-continuous 
(w.r.t.\ the discrete metric) and that it has left-hand limits on $(0,\infty)$, given by
\begin{align}\label{e:lag}
\zeta_{t-}=\eta_{Z(\eta)\setminus Z_{t-}(\eta)}+\eta'_{Z_{t-}(\eta)}+\eta'_{\BX\setminus Z(\eta)}.
\end{align}
Since $\zeta_{t-}\ne \zeta_t$ implies that
$\eta(Z_t(\eta)\setminus Z_{t-}(\eta))+\eta'(Z_t(\eta)\setminus Z_{t-}(\eta))>0$
we see that the set $\{t>0:\zeta_{t-}\ne \zeta_t\}$ is contained in  
\begin{align*}
A:=\{t(x,\eta):x\in\supp\eta\cup\supp\eta',t(x,\eta)<\infty\},
\end{align*}
where we have used the notation \eqref{e:txmu}. The set $A$ is locally
finite and $t\mapsto \zeta_t$ is constant on any connected component
of $[0,\infty)\setminus A$. Using \eqref{e:interpoln_f}, we
  infer that $f(\zeta_0)-f(\zeta_\infty)$ is the limit in probability,
  as $T\to\infty$, of
  $f(\zeta_0)-f(\zeta_T) = - \sum_{t\leq T} \big(f(\zeta_t) -
  f(\zeta_{t-})\big)$. Selecting a sequence $T_n\to\infty$ such that
  the above convergence takes place $\BP$-a.s. and applying the
  triangle inequality, yields the bound
\begin{align}\label{e3.28}
\BE[|f(\zeta_0)-f(\zeta_\infty)|]
\le \BE \bigg[\sum_{t\in A} |f(\zeta_t)-f(\zeta_{t-})|\bigg]\le J+J',
\end{align}
where
\begin{align*}
J&:=\BE\bigg[\int\I\{t(x,\eta)<\infty\}\I\{\eta(Z_{t(x,\eta)}(\eta)\setminus Z_{t(x,\eta)-}(\eta))=1\}
|f(\zeta_{t(x,\eta)})-f(\zeta_{t(x,\eta)-})|\,\eta(dx)\bigg],\\
J'&:=\BE\bigg[ \int\I\{t(x,\eta)<\infty\}\I\{\eta(Z_{t(x,\eta)}(\eta)\setminus Z_{t(x,\eta)-}(\eta))=0\}
|f(\zeta_{t(x,\eta)})-f(\zeta_{t(x,\eta)-})|\,\eta'(dx)\bigg]
\end{align*}
and where we have used \eqref{ea2} in the expression of $J$.
By independence of $\eta$ and $\eta'$ and the Mecke equation \eqref{e:mecke},
\begin{align*}
J=\int\BE[\I\{t(x,\eta_x)<\infty\}
&\I\{\eta_x(Z_{t(x,\eta)}(\eta_x)\setminus Z_{t(x,\eta)-}(\eta_x))=1\}\\
&\times |f(\zeta_{t(x,\eta)}(\eta_x,\eta'))-f(\zeta_{t(x,\eta_x)-}(\eta_x,\eta'))|]\,\lambda(dx),
\end{align*}
where $\mu_x:=\mu+\delta_x$ for $\mu\in\bN$ and $x\in\BX$ 
and where we have used \eqref{e2.15a}.
Similarly,
\begin{align*}
J'=\int\BE[\I\{t(x,\eta)<\infty,&\eta(Z_{t(x,\eta)}(\eta)\setminus Z_{t(x,\eta)-}(\eta))=0\}\\
&\times |f(\zeta_{t(x,\eta)}(\eta,\eta'_x))-f(\zeta_{t(x,\eta)-}(\eta,\eta'_x))|]\,\lambda(dx).
\end{align*}
Fix $x\in\BX$. By \eqref{ea1} and the independence of $\eta$ and $\eta'$ 
we have that, $\BP$-a.s.,
\begin{align}\label{e3.298}
\eta'(Z_{t(x,\eta)}(\eta)\setminus Z_{t(x,\eta)-}(\eta))
=\eta'(Z_{t(x,\eta)}(\eta_x)\setminus Z_{t(x,\eta)-}(\eta_x))=0\quad \text{if $t(x,\eta)<\infty$}.
\end{align}
If $\eta_x(Z_{t(x,\eta)}(\eta_x)\setminus Z_{t(x,\eta)-}(\eta_x))=1$ then
\eqref{e3.2} shows that 
$\eta(Z_{t(x,\eta)}(\eta_x)\setminus Z_{t(x,\eta)-}(\eta_x))=0$.
Therefore we obtain from \eqref{e3.298} that, $\BP$-a.s.,
\begin{align*}
\zeta_{t(x,\eta)-}(\eta_x,\eta')&=\zeta_{t(x,\eta)}(\eta_x,\eta')+\delta_x
\quad \text{if $t(x,\eta)<\infty$}.
\end{align*}
By \eqref{e3.2},
\begin{align*}
\zeta_{t(x,\eta)}(\eta,\eta'_x)&=\zeta_{t(x,\eta)}(\eta,\eta')+\delta_x,
\end{align*}
Furthermore, if $t(x,\eta)<\infty$ we obtain from
\eqref{e3.298} that
\begin{align*}
\zeta_{t(x,\eta)-}(\eta,\eta'_x)&=\zeta_{t(x,\eta)}(\eta,\eta').
\end{align*}
Therefore,
\begin{align*}
\max(J,J')&\le\BE\bigg[ \int\I\{t(x,\eta_x)<\infty\}|f(\zeta_{t(x,\eta)}(\eta_x,\eta')+\delta_x)
-f(\zeta_{t(x,\eta)}(\eta_x,\eta'))|\,\lambda(dx)\bigg].
\end{align*}
To finish the proof we use the identity 
\begin{align}
\label{e3.7}
\BP(\zeta_{t(x,\eta)}\in\cdot\mid t(x,\eta))&=\BP(\eta\in\cdot),\quad 
\BP\text{-a.s.\ on $\{t(x,\eta)<\infty\}$},
\end{align}
which holds for all $x\in\BX$.
To see this, we fix $x\in\BX$ and consider
a measurable function $g\colon\R_+\times\bN\to\R_+$.
Let $\eta''$ be a copy of $\eta$, independent of $(\eta,\eta')$
and set $T:=t(x,\eta)$.
Since $\eta$ is independently scattered, we have that
\begin{align*}
\BE[\I\{T<\infty\}g(T,\zeta_{T})]
=\BE [\I\{T<\infty\}
g(T,\eta_{Z(\eta)\setminus Z_T(\eta)}+\eta'_{Z_T(\eta)}+\eta''_{\BX\setminus Z(\eta)})].
\end{align*}
By Proposition \ref{p2.4}  the conditional distribution
of $\eta_{\BX\setminus Z}$ given $\{ \eta_{Z_t}) : {t\in[0,\infty]}\}$ is that of $\eta''_{\BX\setminus Z(\eta)}$.
By the continuity properties of a CTDT,
$(T,\eta_{Z(\eta)\setminus Z_T(\eta)},Z_T(\eta))$
is measurable with respect to $\sigma\{ \eta_{Z_t} : {t\in[0,\infty]}\}$. As a consequence,
\begin{align*}
\BE[\I\{T<\infty\}g(T,\zeta_{T})]
=\BE[\I\{T<\infty\}g(T,\eta_{\BX\setminus Z_T(\eta)}+\eta'_{Z_T(\eta)})]
\end{align*}
In view of Proposition \ref{p2.1} this implies \eqref{e3.7}. From \eqref{e3.7} we obtain by conditioning that
\begin{align*}
\max(J,J')\le \int\BP(t(x,\eta)<\infty)\BE[|f(\eta+\delta_x)-f(\eta)|]\,\lambda(dx).
\end{align*}
Since $t(x,\eta)<\infty$ if and only if $x\in Z_\infty$, the asserted inequality
\eqref{e3.42} now follows from
\eqref{e:interpoln_f} and \eqref{e3.28}.
\end{proof}

\begin{remark}\label{r:easycase}\rm Assume that $\BX$ is a Borel subset
of a complete separable metric space with metric $\rho$ equipped with the Borel $\sigma$-field
$\cX$. Let $\lambda$ be a diffuse measure on $\BX$ which is finite on
(metrically) bounded Borel sets. 
Fix $x_0\in\BX$ and let $Z_t:=\{x\in\BX:\rho(x_0,x)\le t\}$, $t\ge 0$.
Then $\{Z_t:t\geq 0\}$ is a CTDT which does not
depend on $\mu\in\bN$. If, moreover,
$\lambda(\partial Z_t)=0$, for each $t\ge 0$,
then it is easy to see that \eqref{ea1} and \eqref{ea2} hold. 
If $f\colon\bN\to\R$ satisfies $\BE[|f(\eta)|]<\infty$ and 
$f(\eta_{\BX\setminus Z_t}+\eta_{Z'_t})\overset{\BP}{\longrightarrow}f(\eta)$
as $t\to\infty$
(which is always true if $\lambda(\BX)<\infty$), then
Theorem \ref{t:POSSS} implies that
\begin{align*}
  \BE[|f(\eta)-f(\eta')|]\le 2\int \BE[|D_xf(\eta)|]\,\lambda(dx).
\end{align*}
In fact, the proof is much simpler in this case.
\end{remark}

  \begin{remark}\label{r:attainable}{\rm We will prove in Appendix
      \ref{appendix2} that there exist stopping sets $Z$ that are not
      {\em $\lambda$-attainable}, that is, such that
      $Z\neq \cup_{t\geq 0} Z_t$ for {\em any} $\lambda$-continuous
      CTDT $\{Z_t\}$. Whether the OSSS inequality \eqref{e3.42}
      continues to hold for non-attainable stopping sets is a
      challenging problem that we leave open for future research.}
\end{remark}

\begin{remark}\rm It might be desirable to start the CTDT
with given non-trivial stopping set $Z_0$. So assume
that $\{ Z_t : t\in\R_+\}$ has all properties of a CTDT
except for $\BE[\lambda(Z_0)]=0$. Let the function $f$ satisfy the
assumptions of Theorem \ref{t:POSSS} and let
$\eta'$ be an independent copy of $\eta$. We assert that
\begin{align}\label{e3.42b}\notag
\BE[|f(\eta)-f(\eta')|]&\le \BE[|f(\eta)-f(\eta_{Z\setminus Z_0}+\eta'_{Z_0})|]\\
&\qquad +2\int \BP(x\in Z(\eta)\setminus Z_0(\eta))\BE[|D_xf(\eta)|]\,\lambda(dx).
\end{align}
To see this, we define $\zeta_t$, $t\in[0,\infty]$, by \eqref{e:def_zetat}
and set $\zeta_{0-}:=\eta_Z+\eta_{\BX\setminus Z}$.
Then
\begin{align*}
\BE[|f(\eta)-f(\eta')|]=\BE[|f(\zeta_{0-})-f(\zeta_\infty)|]. 
\le \BE[|f(\zeta_{0-})-f(\zeta_0)|]+\BE[|f(\zeta_{0})-f(\zeta_\infty)|].
\end{align*}
Here the second term can be treated as before, while the first
term can be seen to equal
$\BE[|f(\eta_{Z\setminus Z_0}+\eta'_{Z_0})-f(\eta)|]$
(Further details on this point are left to the reader).
\end{remark}

The Poincar\'e inequality is sharp for linear functionals of $\eta$.
The following example shows that the OSSS inequality
is sharp also for other functionals.
\begin{example}\label{exemptyspace}\rm Assume that $\BX=\R^d$
and that $\lambda$ is the Lebesgue measure $\lambda_d$.
Let $W\subset\R^d$ be compact and $x_0\in\R^d$. Define
$\tau\colon \bN\to[0,\infty]$ by
\begin{align*}
\tau(\mu):=\inf \{s\ge 0: \mu(W\cap B(x_0,s))>0\},
\end{align*}
where $B(x_0,t)$ is the ball with radius $t$ centred at $x_0$.
We assert that
\begin{align*}
Z_t(\mu):=B(x_0,\tau(\mu)\wedge t),\quad t\in\R_+,\,\mu\in\bN,
\end{align*}
is a CTDT. To see this we first note that, given $\mu\in\bN$ and
$s\in\R_+$, that the inequalities $\tau(\mu)>s$ and
$\mu(B(x_0,s)\cap W)>0$ are equivalent. In particular $\tau$
is measurable so that $Z_t$ is graph-measurable for each $t\in\R_+$.
To check \eqref{estopset} for $Z=Z_t$ we need to show that
\begin{align*}
\tau(\mu)\wedge t=\tau(\mu_{B(x_0,\tau(\mu)\wedge t)}+\psi)\wedge t,\quad \mu\in\bN
\end{align*}
provided that $\psi(B(x_0,\tau(\mu)\wedge t))=0$. This can be directly
checked.
Define $f\colon\bN\to \R$ by $f(\mu):=\I\{\mu(W)=0\}$.
Standard arguments yield that $\{ Z_t : t\in \R_+\} $ 
and  $f$ satisfy the assumptions of Theorem \ref{t:POSSS}.  We note that
\begin{align}\label{e:exadd}
f(\mu+\delta_x)-f(\mu)
&=\I\{\mu(W)=0,x\notin W\}-\I\{\mu(W))=0\}\\
&=-\I\{x\in W,\mu(W)=0\} 
 = - \I\{x\in W\} f(\mu) .\notag
\end{align}
Further we have $x\in Z_\infty(\mu)$ iff $\mu(W\cap B^0(x_0,\|x-x_0\|))=0$,
where $B^0(x_0,t)$ is the interior of $B^0(x_0,t)$.
Therefore  \eqref{e3.42} means that
\begin{align}\label{erhs}
\BV[f(\eta)]\le e^{-\lambda_d(W)}
\int_W \exp[-\lambda_d(W\cap B(x_0,\|x-x_0\|))]\,dx.
\end{align}
To compare this with the exact variance
$e^{-\lambda_d(W)}\big(1-e^{-\lambda_d(W)}\big)$ of $f(\eta)$
we assume that $W-x_0$ is isotropic, that
is invariant under rotations. In fact, we can then assume
$x_0=0$ without loss of generality.
Let $I\subset\R_+$ the a Borel set such that
$x\in W$ iff $\|x\|\in I$. Using polar coordinates, we have
for each $x\in\R^d$ that
\begin{align*}
\lambda_d(W\cap B(0,\|x\|))=d\kappa_d\int\I\{s\in I,s\le \|x\|\}s^{d-1}\,ds
=\int\I\{s\le \|x\|\}\,\nu(ds)
\end{align*}
where $\kappa_d:=\lambda_d(B(x,1))$ is the volume of the unit ball
and $\nu$ is the measure on $\R_+$ given by
$\nu(ds):=d\kappa_d\I\{s\in I\}s^{d-1}\,ds$.
Using polar coordinates once again we see that the
right-hand side of \eqref{erhs} equals
\begin{align*}
e^{-\lambda_d(W)}\int e^{-\nu([0,r])}\,\nu(dr)=e^{-\lambda_d(W)}\big(1-e^{-\nu(\R_+)}\big)
=e^{-\lambda_d(W)}\big(1-e^{-\lambda_d(W)}\big),
\end{align*}
where we have used a well-known formula of Lebesgue--Stieltjes
calculus, see e.g.\ \cite[Proposition A.31]{LastPenrose17}.  Hence
\eqref{e3.42} is sharp in this case. In view of \eqref{e:exadd},
  one sees immediately that the right-hand side of the Poincar\'e
  inequality \eqref{poincare} equals $\lambda_d(W)e^{-\lambda_d(W)}$,
  which is seen to be suboptimal with respect to the exact variance by
  letting $\lambda_d(W) \to \infty$. It is also interesting to notice
  that, since $D_xf(\mu) = - \I\{x\in W\} f(\mu)\leq 0$ and therefore
  $D_yD_xf(\mu) = \I\{x,y\in W\} f(\mu)\geq 0$, then the variance of
  $f(\eta)$ can be bounded by using the {\it restricted $L^1$-$L^2$
    inequality} proved in \cite[Theorem 1.6]{Nourdin2020}. Such a
  result yields indeed that
$$
\BV[f(\eta)]  \leq \frac{1}{2}\int_W 
\frac{\BE[(D_xf(\eta))^2]}{1+\log\frac{1}{\BE[ |D_xf(\eta)| ]^{1/2} }} \lambda_d(dx)
= \left(1 \!-\! \frac{2}{2\!+\!\lambda_d(W)} \right) e^{-\lambda_d(W)},
$$
an estimate that also improves -- albeit not as sharply -- the Poincar\'e inequality.
\end{example}

\section{Variants of the OSSS inequality}
\label{s:variants_OSSS}

In this section, we present some variants and extensions of the
OSSS inequality for Poisson functionals stated in Theorem
\ref{t:POSSS}. Firstly, in the spirit of \cite[Theorem B.1]{OSSS05},
we extend Theorem \ref{t:POSSS} to two functions and to the case of
a randomized CTDT (Corollary \ref{c:OSSS_COV} below).  The two function version has recently been used by \cite{Hutchcroft20} to derive critical exponent inequalities for random cluster models on transitive graphs. We will then present a variant of the OSSS inequality for marked Poisson point processes (see Theorem \ref{t:POSSSmarked}) -- a
declination of our main findings that will be very useful for the
applications in Part \ref{p:applns}. For the rest of the section, $\eta$ is
a Poisson process on $\BX$, with a locally  finite and diffuse
intensity measure $\lambda$.

Let $(\mathbb{Y},\mathcal{Y})$ be a measurable space and suppose
that $Z^y$ is, for $\BP(Y\in\cdot)$-a.e.\ $y$, a stopping set, such
that the mapping $(\mu,x,y)\mapsto \I\{x\in Z^y(\mu)\}$ is (jointly)
measurable on $\bN\times \BX\times\BY$. If $Y$ is an independent
$\mathbb{Y}$-valued random element, then $Z^Y$ is called a 
{\em  randomized stopping set}.  If $\{ Z_t^y : t\in\R_+\}$ is a
CTDT for all $y \in \BY$ such that the above measurability
properties are satisfied, then $\{ Z_t^Y: t\in\R_+\}$ is
called a {\em randomized continuous time decision tree} (randomized
CTDT).  In this case, we say that $\{ Z_t^Y: t \in \R_+\}$ is a
randomized CTDT {\em determining} $f \colon \bN \to \R$ if
$\{ Z_t^y: t\in\R_+\}$ is a CTDT determining $f$ for
all $y \in \BY$, where the property of being a CTDT
determining a given function $f$ corresponds to the two properties
\eqref{eCTDTf}--\eqref{eCTDTf2} introduced above.

\begin{corollary}\label{c:OSSS_COV}
Let $f \colon\bN\to [-1,1]$ and $g \colon \bN \to \R$ be measurable and
such that $\BE[|g(\eta)|]<\infty$. Let $Y$ be an independent
$\mathbb{Y}$-valued random element.  Let
$\{ Z_t : t\in \R_+ \}= \{ Z_t^Y: t\in \R_+\} $ be a
randomized CTDT determining $f$ such that, for
$\BP(Y\in \cdot)$-a.e.\ $y\in\BY$, $\{ Z_t^y \}$ satisfies \eqref{ea1} and
\eqref{ea2}. Let $g$ also satisfy \eqref{eCTDTf2}. Then,
\begin{align}
|\BC(f,g)| \leq  2 \int \BP(x\in Z(\eta))\BE[|D_xg(\eta)|]\,\lambda(dx),
\end{align}
where $Z(\eta) := Z_\infty(\eta) =Z^Y_\infty(\eta)$, and therefore
\begin{align}\label{e:bosss}
\BV(f) \leq  2\int \BP(x\in Z(\eta)) \BE[|D_xf(\eta)|]\,\lambda(dx).
\end{align}
\end{corollary}
\begin{proof}
We first observe that the specific form of the two inequalities in
the statement implies that it suffices to prove them for a generic
CTDT determining $f$ and satisfying \eqref{ea1} and \eqref{ea2}; one
can then apply the obtained estimate to $\{Z_t^y\}$,
for $\BP(Y\in\cdot)$-a.e.\ $y$, and take expectations with
respect to $Y$. Thus, we will now prove the statement for a CTDT
$\{ Z_t : t\in \R_+\} $ determining $f$ and such that
\eqref{ea1} and \eqref{ea2} are verified. Let $\eta'$ be an
independent copy of $\eta$. We will again use the definition of
$\zeta_t$ given in \eqref{e:def_zetat}. Since $f(\eta) = f(\zeta_0)$
as $Z_t$ is a CTDT for $f$, $\zeta_0 \overset{d}{=} \eta$ and
$\zeta_{\infty} = \eta'$, we have that
$$ 
f(\eta)g(\eta) \overset{d}{=} f(\zeta_0)g(\zeta_0) , \, f(\eta)g(\eta') = f(\zeta_0)g(\zeta_{\infty}).
$$
Thus we obtain that
$$
|\BC(f,g)| = |\BE[f(\eta)(g(\eta) - g(\eta'))]| = |\BE(f(\zeta_0)(g(\zeta_0) - g(\zeta_{\infty}))]| 
\leq   \BE[|g(\zeta_0) - g(\zeta_{\infty})|],$$
where we have used that $|f| \leq 1$ in the last inequality. Observe
that $\BE[|g(\zeta_0) - g(\zeta_{\infty})|]$ is the same as the LHS of
\eqref{e3.28} with $g$ replacing $f$.  Now, the rest of the
proof follows the same lines as that of Theorem \ref{t:POSSS} once it
is observed that, after having established \eqref{e3.28} in
such a proof using the fact that $g$ also satisfies \eqref{eCTDTf2}, the fact that $\{Z_t\}$ is a CTDT determining $f$ is
never used, and only the stopping set properties of $Z_t$ are exploited.
\end{proof}

We stress that -- as opposed to the main OSSS inequality \eqref{e3.42} -- relation \eqref{e:bosss} is derived under the assumption that $|f|\leq 1$, and is consequently not homogeneous in $f$. 

We will now generalize Theorem \ref{t:POSSS} to marked Poisson processes. 
We let $(\BM,\mathcal{M})$ be another Borel space  
and consider  the product $\BX\times\BM$. As in Subsection \ref{ss:intropoisson}
we take  $B_n\in\cX$, $n\in\N$, satisfying $B_n\uparrow\BX$. However, this time
we define the localizing subring of $\cX\otimes\mathcal{M}$
as the system of all sets of the 
form $B\cap (B_n\times\BM)$ for some $B\in\cX\otimes \mathcal M$ and $n\in\N$.
A measure $\nu$ on $(\BX, \cX)$
which is finite on $\cX_0$ is called {\em locally finite}.
In that case $\nu(B_n\times\BM)<\infty$ for each $n\in\N$. 

Let $\lambda$ be a locally finite measure on $\BX\times\BM$ such that
$\bar\lambda:=\lambda(\cdot\times\BM)$ is diffuse and let $\eta$ be a
Poisson process with intensity measure $\lambda$. According to
  Remark \ref{rA1}, in such a framework a {\em stopping set} (on
$\BX$) is a graph-measurable mapping $Z\colon \bN(\BX\times\BM)\to\cX$
satisfying \eqref{estopset}.
A CTDT on $\BX$ is a family $\{ Z_t : t\in\R_+\}$ of stopping sets on
$\BX$ having the properties \eqref{etr1}--\eqref{etr2}; we
  stress once again that, in the present framework, each $Z_t$ is a
  $\mathcal{X}$-valued mapping defined on $\bN(\BX\times\BM)$.  A
CTDT $\{ Z_t : t\in \R_+\} $ is said to {\em determine} a given
measurable mapping $f\colon \bN(\BX\times\BM) \to\R$ if
\begin{align}\label{eCTDTfmarked}
f(\mu)=f(\mu_{Z(\mu)\times\BM}),\quad \mu\in \bN(\BX\times\BM),
\end{align}
where $Z:=Z_\infty$, and relation \eqref{eCTDTf2} holds, with
  $Z_t(\eta)$ and $Z(\eta) = Z_\infty(\eta)$ replaced by
  $Z_t(\eta)\times \BM$ and
  $Z_\infty(\eta)\times \BM = Z(\eta)\times \BM$, respectively.  
Let 
$\bar\eta:=\eta(\cdot\times\BM)$.  In the next theorem we shall assume
that
\begin{align}\label{maea1}
\bar\lambda(Z_t(\mu)\setminus Z_{t-}(\mu))=0,\quad \mu\in\bN,\,t\in\R_+
\end{align}
and
\begin{align}\label{maea2}
\BP(\text{$\bar\eta(Z_t(\eta)\setminus Z_{t-}(\eta))\le 1$ for all $t\in\R_+$})=1.
\end{align}
\begin{theorem}[OSSS inequality for marked processes] \label{t:POSSSmarked} Let
$f\colon\bN(\BX\times\BM)\to\R$ be measurable such that
$\BE[|f(\eta)|]<\infty$. Let $Y$ be an independent
$\mathbb{Y}$-valued random element. Let $\{Z_t\} = \{Z_t^Y\}$ be a
randomized CTDT determining $f$ such that for $\BP(Y\in \cdot)$-a.e.\ $y\in\BY$,
$\{Z_t^y\}$ satisfies \eqref{maea1} and \eqref{maea2}.  Let $\eta'$
be an independent copy of $\eta$.  Then
\begin{align}\label{e9.52g}
\BE[|f(\eta)-f(\eta')|]\le 2 \int \BP(x\in Z(\eta))
\BE[|D_{(x,w)}f(\eta)|]\,\lambda(d(x,w)).
\end{align}
\end{theorem}
\begin{proof} As with the proof of Corollary \ref{c:OSSS_COV}, we will
  prove the result for a deterministic CTDT. In particular, we will
  show that $\tilde{Z}_t =Z_t \times \BM,\, t\geq 0,$ is a CTDT
    on $\BX\times \BM$ satisfying the assumptions of Theorem
  \ref{t:POSSS}.  The stopping set properties are immediate, see
  Remark \ref{rA2}.  Further, it is easy to check that
  $\tilde{Z}_t$ satisfies \eqref{etr1} and \eqref{etr2}
(as a collection of set-valued mappings with values in
    $\mathcal{X}\otimes \mathcal{M}$), since $Z_t$ satisfies the same
    conditions (as a collection of mappings with values in
    $\mathcal{X}$). Further, by the definition of
  $\tilde{Z}_t, \tilde{\lambda}, \tilde{\eta}$, \eqref{maea1} and
  \eqref{maea2} imply \eqref{ea1} and \eqref{ea2} respectively. Also, \eqref{eCTDTf2} is assumed. Thus,
  we have that $\tilde{Z}_t$ is a CTDT as in Theorem
  \ref{t:POSSS} and further
  $\BP(x \in Z(\eta)) = \BP((x,w) \in \tilde{Z}(\eta))$ for any
  $w \in \BM$. This completes the proof.
\end{proof}

\part{Quantitative Noise Sensitivity}\label{p:quantns}

{Throughout this part, we adopt the general framework and
  notation for Poisson processes and stopping sets introduced in Section \ref{ss:intropoisson} and Appendix \ref{appendix1}, respectively}. { In particular, we denote by $\eta$ a Poisson process on $\BX$ with locally finite (and not necessarily diffuse) intensity $\lambda$.}

\section{Schramm-Steif inequalities on the Poisson space}
\label{s:sspoisson}

{{} The next result is a crucial finding of the paper: it provides an upper bound on the variance of conditional expectations of multiple integrals, where the conditioning is with respect to the restriction of $\eta$ to a randomized stopping set, a notion that has been introduced in Section \ref{s:variants_OSSS}. For the sake of conciseness, the proof of Theorem \ref{t:condwitobd} will only be provided in the case where $Z$ is non-randomized. The general case follows from an averaging argument analogous to the one used in the proof of Corollary \ref{c:OSSS_COV} (to simplify the notation we also suppress the dependence on the independent random element $Y$). A similar strategy is adopted (sometimes tacitly) elsewhere in Part \ref{p:quantns} --- see e.g. the proof of \eqref{e:l2osss} below. }

In the subsequent discussion, we will mostly focus on (possibly randomized)
stopping sets $Z$ verifying the assumption that there exists an increasing sequence
of stopping sets $\{Z_k : k\geq 1\}$ such that
\begin{equation}\label{e:finitapp}
\BP( \eta(Z_k(\eta)) <\infty ) =1, \mbox{  $k\geq 1$, \quad and }\quad Z = \bigcup_{k \geq 1} Z_k.
\end{equation}
We stress that, in the case of randomized stopping sets, we implicitly
assume that the sets $Z_k = Z_k^Y$ are such that the randomization $Y$
is the same for every $k$.  The reason for requiring
\eqref{e:finitapp} is that, according to Theorem \ref{tAstopping2},
such a condition is sufficient for the Markov property
\eqref{eMarkov3} to hold.
As already done in other parts of the paper, given a (randomized) stopping set $Z$
we will often write $Z$  for $Z(\eta)$ (resp.\ $Z^Y(\eta)$).

\begin{theorem}
\label{t:condwitobd}
Let $Z$ be a {randomized} stopping set verifying \eqref{e:finitapp}. Let $k\in\N$ and
$u \in L^2(\lambda^k)$ {{} be $\lambda^k$-a.e. symmetric}.  Then,
\begin{align*}
\BE[\BE[I_k(u) \mid \eta_Z]^2] \leq k! \int u(x_1,\ldots,x_k)^2
\,\BP(\{x_1,\ldots,x_k\} \cap Z(\eta) \neq \emptyset)\, \lambda^k(\md (x_1,\ldots,x_k)),
\end{align*}
where $\eta_Z:=\eta_{Z(\eta)}$ denotes the restriction of $\eta$ to $Z(\eta)$.
\end{theorem}
\begin{proof} {{} As announced, in the proof we assume that $Z$ is non-randomized. For every $k\geq 1$, define $\mathcal{S}_k$ to be the subset of $L^2(\lambda^k)\cap L^1(\lambda^k)$ composed of symmetric bounded kernels $u$, such that the support of $u$ is contained in a set of the type $A\times \cdots \times A$, with $A\in \mathcal{X}_0$. By virtue of the local finiteness of $\lambda$ and of the fact that $\sigma(\cX_0) = \cX$, one has that, for every a.e. symmetric $u\in L^2(\lambda^k)$, there exists a sequence $\{u_n : n\geq 1\}\subset \mathcal{S}_k$ such that $u_n \to u$ in $L^2(\lambda^k)$. Since both sides of the inequality in the statement are trivially continuous in $u$ (because of \eqref{e:onreln} and the contractive properties of conditional expectations), it follows that it is enough to prove the desired estimate for $u\in \mathcal{S}_k$, in which case one has also that
\begin{align}\label{e-628}
\int \bigg(\int u(x_1,\ldots,x_k)\lambda^{k-i}(\md(x_{i+1},\ldots,x_k))\bigg)^2
\,\lambda^i(\md(x_1,\ldots,x_i))<\infty,\quad i\in\{1,\ldots,k\}.
\end{align}}
For $k\in\N_0$, $B\in\cX$ and $\mu\in\bN$ we define
\begin{align}\label{e-WienerItopath}
  I_{k,B,\mu}(u) := \sum^k_{i=0}(-1)^{k-i}\binom{k}{i}\mu^{(i)}_B\otimes\lambda^{k-i}_B(u),
\end{align}
where $\mu^{(i)}$ is the $i$-th factorial measure of $\mu$
and $\mu^{(i)}_B:=(\mu_B)^{(i)}$ is the 
$i$-th factorial measure of the restriction $\mu_B$ of $\mu$ to $B$.
(We adopt the standard conventions that
$\mu^{(0)}(c)=\lambda^0(c)=c$ for all $c\in\R$,
$\mu^{(0)}\otimes\lambda^{k}:=\lambda^{k}$
and $\mu^{(k)}\otimes\lambda^{0}:=\mu^{(k)}$.)
Since $u\in \mathcal{S}_k$ we have from the multivariate
Mecke formula {{} (see \cite[Section 4.2]{LastPenrose17})} that $I_{k,B,\mu}(u)$ is well-defined and finite
for $\BP(\eta\in\cdot)$-a.e.\ $\mu\in\bN$. By \cite[Proposition 12.9]{LastPenrose17} we have $\BP$-a.s.
the pathwise identity
\begin{align}\label{e-WienerIto}
  I_{k}(u)=I_{k,\BX,\eta}(u).
\end{align}
We wish to exploit the conditional variance formula
\begin{align}\label{e-condvariance}
\BE\big[\BE[I_k(u)\mid \eta_Z]^2\big]
=\BE\big[I_k(u)^2\big]-\BE[\BV[I_k(u)\mid \eta_Z]],
\end{align}
which is a direct consequence of the law of the total expectation.
To deal with the conditional variance of the right-hand side of \eqref{e-condvariance},
we use Theorem \ref{tAstopping2} together with \eqref{e:finitapp} to infer that
the conditional distribution of $I_k(u)$ given $\eta_Z=\mu$
coincides (for $\BP(\eta_Z\in\cdot)$-a.e.\ $\mu$) with that of
\begin{align*}
{{} J} := \sum^k_{i=0}(-1)^{k-i}{\binom{k}{i}}\big((\mu+\eta_{Z(\mu)^c})^{(i)}\otimes\lambda^{k-i}\big)(u).
\end{align*}
{We stress that, since $u\in \mathcal{S}_k$, such a quantity is well-defined for $\BP(\eta_Z\in\cdot)$-a.e. $\mu$}. Writing $\lambda=\lambda_{Z(\mu)}+\lambda_{Z(\mu)^c}$ we deduce
from a simple calculation that
\begin{align*}
{{} J } = \sum^k_{i=0}\binom{k}{i}I_{i,Z(\mu)^c,\eta}(I_{k-i,Z(\mu),\mu}(u)),
\end{align*}
where the inner (deterministic) {{} multiple} integral
refers to the arguments of $u$ with index in $\{i+1,\ldots,k\}$
(with the first $i$ variables fixed) and the outer (stochastic)
integral is performed with respect to the remaining variables. {{} We observe that, since $u\in \mathcal{S}_k$, one has that, for $\BP(\eta_Z\in\cdot)$-a.e. $\mu$,
$$
\int_{(\BX\setminus Z(\mu))^i} \big(I_{k-i,Z(\mu),\mu}(u_{(x_1,\ldots,x_i)})\big)^2
\,\lambda^i(\md(x_1,\ldots,x_i))<\infty,\quad i=1,...,k,
$$
where $u_{(x_1,\ldots,x_i)}$ denotes the function
$(x_{i+1},\ldots,x_k)\mapsto u(x_1,\ldots,x_k)$}. 
By virtue of \eqref{e-628} we can
now apply the orthogonality relations  \eqref{e:onreln}
with $\eta$ replaced by $\eta_{Z(\mu)^c}$ to obtain
for $\BP(\eta_Z\in\cdot)$-a.e.\ $\mu$
\begin{align*}
\BV[I_k(u)\mid \eta_Z=\mu]
=\sum^{k}_{i=1}\binom{k}{i}^2i!
\int_{(\BX\setminus Z(\mu))^i} \big(I_{k-i,Z(\mu),\mu}(u_{(x_1,\ldots,x_i)})\big)^2
\,\lambda^i(\md(x_1,\ldots,x_i)).
\end{align*}
Dropping all summands except the $k$-th and integrating
w.r.t.\ the distribution of $\eta$
we get
\begin{align*}
\BE[\BV[I_k(u)\mid \eta_Z]]
\ge k!\, \BE\bigg[\int_{(\BX\setminus Z(\eta))^k} u(x_1,\ldots,x_k)^2
\,\lambda^k(\md(x_1,\ldots,x_k))\bigg].
\end{align*}
Inserting this into \eqref{e-condvariance} we obtain that
\begin{align*}
\BE\big[&\BE[I_k(u)\mid \eta_Z]^2\big]\\
&\le k! \BE\bigg[\int u(x_1,\ldots,x_k)^2 (1-\I\{x_1\notin Z(\eta),\ldots,x_k\notin Z(\eta)\}) \,\lambda^k(\md(x_1,\ldots,x_k)) \bigg]
\end{align*}
and hence the assertion.
\end{proof}

By inspection of the proof of Theorem \ref{t:condwitobd}, one sees that 
our arguments only exploit the fact that $Z$ verifies the Markov property \eqref{eMarkov3}.  
It is allowed to choose $Z \equiv \BX$ in Theorem \ref{t:condwitobd}.  Then the inequality boils down
to an equality, namely to the case $k=m$ and $u=v$ of the isometry property \eqref{e:onreln}.

\medskip

\begin{corollary}\label{c:2fn_l2OSSS}
  Consider measurable mappings  $f,g : \bN \to \R$ such that $\BE[f(\eta)^2], \, \BE[g(\eta)^2]<\infty$. {{} We denote by $\{u_k : k = 0,1, ...\}$ the sequence of kernels of the Wiener-It\^o chaos expansion of $f(\eta)$ (see \eqref{e:chaos}), and by $\{v_k : k = 0,1, ...\}$ the kernels of the decomposition of $g$.} We assume that $f$ is determined by {a randomized stopping set $Z$} verifying \eqref{e:finitapp}. Then, for all $k \geq 1$, we have that
\begin{align}\label{e:l2osss}\notag
(\BE&[I_k(u_k)I_k(v_k)])^2 \\
&\leq k!\, \BE[f(\eta)^2] \int \BP(\{x_1,\ldots,x_k\} \cap Z(\eta) \ne \emptyset) 
v_k^2(x_1,\ldots,x_k)\, \lambda^k(\md(x_1,\ldots,x_k)).
\end{align}
\end{corollary}
\begin{proof}[Proof of Corollary \ref{c:2fn_l2OSSS}]
 Fix $k\geq 1$, and recall from \eqref{etr5}
that $f(\eta)$ is $\sigma(\eta_Z)$-measurable. The orthonormality
  properties of multiple integrals \eqref{e:onreln} and 
the Cauchy--Schwarz inequality yield that
\begin{align*}
(\BE[I_k(u_k)I_k(v_k)])^2&= (\BE[f(\eta) I_k(v_k)])^2 = (\BE[f(\eta)\,\BE[I_k(v_k)\mid \eta_Z]])^2 \\
&\leq  \BE[f(\eta)^2] \times \BE[\BE[I_k(v_k) \mid \eta_Z]^2]. 
\end{align*}
An application of Theorem \ref{t:condwitobd} implies the result.
\end{proof}
We will now obtain an equivalent of the Schramm-Steif quantitative
noise sensitivity theorem (see \cite[Theorem 1.8]{SS10}). {{} Several
  applications of this result will be illustrated in the subsequent
  sections}.
\begin{corollary}[Schramm-Steif inequality for Poisson functionals]
  \label{c:sspoisson} Consider a measurable and square-integrable $f : \bN \to \R$ {{} as in \eqref{e:chaos}}, and assume that $f$ is determined by a randomized stopping set $Z$ verifying \eqref{e:finitapp}. Then, for all $k\geq 1$,
\begin{equation}\label{e:ss}
\BE[I_k(u_k)^2] \leq k \delta(Z)\, \BE[f(\eta)^2],
\end{equation}
where $\delta(Z) := \sup_{x\in \BX} \BP(x\in Z(\eta))$.
\end{corollary}

\begin{proof}
  Setting $g = f$ in \eqref{e:l2osss} and using the union bound for
  $\BP(\{x_1,\ldots,x_k\} \cap Z \neq \emptyset)$, we derive that
$$ 
\int\BP(\{x_1,\ldots,x_k\} \cap Z(\eta) \neq \emptyset)
u_k^2(x_1,\ldots,x_k)\, \lambda^k(\md(x_1,\ldots,\md x_k)) 
\leq k \delta(Z) \|u_k\| ^2_{L^2(\lambda^k)}.
$$
The proof is completed using the orthonormality relation \eqref{e:onreln} and Corollary \ref{c:2fn_l2OSSS}.
\end{proof}

To conclude the section, we will now extend the above OSSS-type
inequality to more general functions, albeit with square-root of the
revealement probability. The corollary below was once again motivated
by a similar result in the discrete case, whose proof was shared with
us by Hugo Vanneuville. Some other consequences of the above
  inequality can be found in the first arxiv version of the
  article.
\begin{corollary}
\label{c:sqrtOSSS}
Consider a measurable and square-integrable function $f : \bN \to \R$ and assume that $f$ is determined by a randomized stopping set $Z$ verifying \eqref{e:finitapp}. Then,%
$$ \BV[f(\eta)] \leq {{} 3} \sqrt{\delta(Z)} \int_{\BX} \BE[|D_{x}f(\eta)|^2]\lambda(\md x).$$
\end{corollary}
Deriving the above corollary with $\delta(Z)$ instead of $\sqrt{\delta(Z)}$ remains a challenge and would give a $L^2$ version of OSSS inequality with randomized stopping sets instead of randomized CTDTs. In view of Remark \ref{r:attainable} and being an $L^2$ version, this would significantly enlarge the scope of applications of the OSSS inequality.
\begin{proof}
{{} We can assume $\delta(Z) \in (0,1)$.} Using \eqref{e:onreln} --- \eqref{e:FSident} and Corollary \ref{c:sspoisson}, we deduce that that, for all $m \geq 1$,
\begin{align*}
&\int_{\BX} \BE[|D_{x_1}f(\eta)|^2]\lambda(\md x_1)  = \sum_{k=1}^{\infty} \int_{\BX^k} \frac{1}{(k-1)!}\BE[D^k_{x_1,\ldots,x_k}f(\eta)]^2\lambda(\md x_1) \lambda(\md x_2)\ldots \lambda (\md x_k) \\
& \geq m \sum_{k=m}^{\infty} \int_{\BX^k} \frac{1}{k!}\BE[D^k_{x_1,\ldots,x_k}f(\eta)]^2\lambda(\md x_1) \lambda(\md x_2)\ldots \lambda (\md x_k)  \\
& = m \left( \BV[f(\eta)] - \sum_{k=1}^{m-1} \int_{\BX^k} \frac{1}{k!}\BE[D^k_{x_1,\ldots,x_k}f(\eta)]^2\lambda(\md x_1) \lambda(\md x_2)\ldots \lambda (\md x_k) \right)  \\
& \geq  m \left( \BV[f(\eta)]  - \delta(Z) \sum_{k=1}^{m-1}k\BV[f(\eta)] \right) \\
& \geq m \BV[f(\eta)]\left( 1 - \frac{m^2\delta(Z)}{2} \right) .
\end{align*}
{{} Selecting $m = \lceil  \delta(Z)^{-1} \rceil$ and exploiting the usual Poincar\'e inequality yields the bound    
$$ \BV[f(\eta)] \leq 2 \min\left\{ \frac12\, ; \, \frac{\sqrt{\delta(Z)}}{\1-\sqrt{\delta(Z)}}\right\}  \int_{\BX} \BE[|D_{x}f(\eta)|^2]\lambda(\md x),$$
and the conclusion follows from elementary considerations.}
\end{proof}

\section{Quantitative noise sensitivity}
\label{s:quantns}


As indicated before, Corollary \ref{c:sspoisson} has direct
applications to noise sensitivity, that one can use to derive further
results on exceptional times and bounds on critical windows. We will
develop these themes in this section and the next.
In order to formally define the notion of noise sensitivity, we need to introduce a collection of resampling procedures for the Poisson point process $\eta$, that can then be used to define the Ornstein-Uhlenbeck semigroup.

\smallskip

\subsection{The Ornstein-Uhlenbeck semigroup}\label{ss:ou}

For every $t\ge 0$ we define $\eta^t$ to be
the Poisson process with intensity measure
$\lambda$ obtained by independently
deleting each point in the support of $\eta$ with probability
$1-e^{-t}$, and then adding an independent Poisson process with
intensity measure $(1-e^{-t})\, \lambda$.
For $f \colon\bN \to \R$ such
that $\BE[|f(\eta)|] < \infty$, define the operator $T_t$ as
\begin{equation}\label{e:repetaep}
T_t f(\eta) := \BE[f(\eta^t) \mid \eta].
\end{equation} 
Setting $T_{\infty} F:=\BE[f(\eta) ]$, one sees immediately that the operators $P_s:=T_{-\log s}$, $s\geq0$, coincide with those defined in \cite[Section 20.1]{LastPenrose17}. Now assume that $f \colon\bN \to \R$ is such that $\BE[f(\eta)^2]<\infty$ and that $f$ admits the chaotic representation \eqref{e:chaos}. In this case, according to {\em Mehler's formula} (see \cite[formula (3.13)]{LPS} or \cite[formula (80)]{Last16}) one has that, for each $t\ge 0$,
\begin{align}\label{e:condMehler}
T_t f(\eta) = \BE[f(\eta^{t}) \mid \eta ] = \sum_{k=0}^\infty e^{-kt}I_k(u_k),
\end{align}
in such a way that the family of operators $\{T_t:t\ge 0\}$ {{} (restricted to $L^2(\BP)$)} coincides with the classical
{\it Ornstein-Uhlenbeck semigroup} associated with $\eta$; see e.g.\ \cite[Section 7]{Last16} for a full discussion.

\begin{remark}[Markov process representation\label{r:markovrep}]{\rm In the sections to follow, we will sometimes need to realise the resampled processes $\{ \eta^t : t \geq 0\}$ as a Markov process with values in $\bN$. To do so, we recall from \cite[Proposition 4 and Section 7]{Last16}
that {{} the infinitesimal generator $L$ of the semigroup} $\{T_t:t\ge 0\}$ is explicitly given as
\begin{align*}
Lf(\eta):=\int (f(\eta-\delta_x)-f(\eta))\,\eta(dx)
-\int (f(\eta+\delta_x)-f(\eta))\,\lambda(dx)
\end{align*}
for all $f : {\bf N} \to \R$ verifying suitable integrability assumptions.
This is the generator of a free birth and death
process on $\bN$, with $\BP(\eta\in\cdot)$
as its stationary measure. If $\lambda$ is finite, then it
is straightforward to construct a Markov process $\{ \tilde{\eta}^t : t\geq 0\}$
with generator $L$. Indeed, start with some arbitrary
initial configuration $\mu\in\bN$  
and attach independent unit rate exponential
life times to each point of $\mu$ (multiplicities have to be
taken into account). At the end of its lifetime, the point is removed. Independently, new points are born with intensity $\lambda(\BX)$, distributed according to the normalized measure $\lambda$, and having
independent exponential life times as well. 
The arising Markov process $\{ \tilde{\eta}^t : t\geq 0\}$ is right-continuous
and has left-hand limits (c\`adl\`ag) with respect to the discrete topology. If $\lambda(\BX)=\infty$ we can 
partition $\BX$ into sets of finite $\lambda$-measure
and paste together independent birth and death processes. This procedure yields eventually a homogeneous
Markov process $\{ \tilde{\eta}^t : t\geq 0\}$ that has c\`adl\`ag paths with respect to the topology of weak convergence of point measures; we refer the reader to \cite{Preston1975, Xia2005} for more details. What is of importance for us is that, in the case $\tilde\eta^0=\eta$,
one has that $\BE[F(\tilde\eta^t)\mid \tilde\eta^0] = \BE[F(\tilde\eta^t)\mid \eta] = T_tf(\eta) $, {{} as one can verify by a direct computation. From now on, we will refer to the formula $T_tf(\eta) = \BE[F(\tilde\eta^t)\mid \eta] $, $t\geq 0$, as the {\it Markov process representation } of the Ornstein-Uhlenbeck semigroup. When adopting the Markov process representation of $\{T_t\}$ we will write $\tilde\eta^t = \eta^t$, by a slight abuse of notation. }
}
\end{remark}

\subsection{Noise sensitivity}\label{ss:ns}
{{}
With the above notation at hand, we will now present the central definition of the section.

\begin{definition}\label{d:ns}\rm Let $f_n : \bN \to \R, \, n\ge 1$ be a
  sequence of measurable mappings such that $\BE[ f_n(\eta)^2] <\infty$, for each $n$.
The sequence $\{f_n\}$ is said to be 
{\em noise sensitive} if, for every $t>0$,
$$
\BE[f_n(\eta) f_n(\eta^t)] - \BE[f_n(\eta)]^2\to 0, \quad n\to\infty.
$$
\end{definition}

\medskip

In order to detect noise sensitivity, it is often useful to exploit the fact, if $\BE[ f(\eta)^2] <\infty$ and $f$ admits the chaos expansion \eqref{e:chaos}, then \eqref{e:condMehler} and \eqref{e:onreln} imply that
\begin{equation}
\label{e:CovMehler}
\BE[f(\eta) f(\eta^t)]  = \sum_{k=1}^\infty e^{-kt} \BE[I_k(u_k)^2] {{} + \BE[f(\eta)]^2}.
\end{equation}
Relation \eqref{e:CovMehler} is the pivotal element in the proof of the next statement.

\begin{proposition}\label{p:covariance} {{} Let $f_n : \bN\to \R$, $n\ge 1$, be a sequence of measurable mappings such that
\begin{equation}\label{e:bvar}
\sup_n\BE[f_n(\eta)^2] := C<\infty,
\end{equation} }
and denote by $u^{(n)}_k$ the $k$th kernel in the chaotic expansion \eqref{e:chaos} of $f_n(\eta)$.
\begin{enumerate}
\item The sequence $\{ f_n : n\geq 1\}$ is noise sensitive if and only if
\begin{equation}\label{e:singleconv}
\lim_{n\to\infty}\BE\big[I_k\big(u^{(n)}_k\big)^2\big]=0, \quad \, k\in\N.
\end{equation}

\item For every nonnegative sequence $\{ t_n : n\ge 1\}$, the following double implication holds: 
\begin{equation}\label{e:covconv}
\BC[f_n(\eta),f_n(\eta^{t_n})] \to 0, \, \, \mbox{as $n \to \infty$,}
\end{equation}
if and only if
\begin{align}\label{e:NScondprob}
\BE[f_n(\eta^{t_n}) \mid \eta ] - \BE[f_n(\eta)] \overset{L^2(\BP)}{\longrightarrow} 0,\, \, \mbox{as $n \to \infty$}.
\end{align}
\end{enumerate}

\end{proposition}
\begin{proof} The fact that noise sensitivity implies \eqref{e:singleconv} is a direct consequence of the relation $\BE[f_n(\eta) f_n(\eta^t)] - \BE[f_n(\eta)]^2\geq e^{-kt} \BE[I_k(u^{(n)}_k)^2]$, which follows in turn from \eqref{e:CovMehler}. On the other hand, if \eqref{e:singleconv} is in order one can use again \eqref{e:CovMehler} together with the bound $\BE[I_k(u^{(n)}_k)^2]\leq \BE[f_n(\eta)^2]\leq  C$ (valid for all $k,n$), and infer noise sensitivity by dominated convergence. This proves Part 1. Part 2 is deduced from \eqref{e:bvar} and from the two relations
$$
\BC[f_n(\eta),f_n(\eta^{t_n})] = \sum_{k=1}^\infty e^{-kt_n} \BE[I_k(u^{(n)}_k)^2]\geq \sum_{k=1}^\infty e^{-2kt_n} \BE[I_k(u^{(n)}_k)^2] = \BV[\BE[f_n(\eta^{t_n}) \mid \eta ],
$$ 
and 
$$
\BC[f_n(\eta),f_n(\eta^{t_n})] = \BE\Big[\big(f_n(\eta)  - \BE[f_n(\eta)]\big)\big(\BE[f_n(\eta^{t_n}) \mid \eta ] - \BE[f_n(\eta)]\big)\Big].
$$
\end{proof}

Combining Part 1 of Proposition \ref{p:covariance} with Corollary \ref{c:sspoisson}, we
deduce the following criterion for noise sensitivity. 
\begin{proposition}
\label{p:randalgns}
{{} Let $f_n : \bN\to \R$, $n\ge 1$, be a sequence of measurable mappings satisfying the assumptions of Proposition \ref{p:covariance}}.
Assume that there exist randomized
stopping sets $Z_n$, $n\geq 1$, such that,  for every $n$, $Z_n$ verifies \eqref{e:finitapp} and
determines $f_n$, and moreover
\begin{equation*}
\delta_n := \delta(Z_n) := \sup_{x\in \BX} \BP(x\in Z_n)\longrightarrow 0, \quad \mbox{as $n\to\infty$}.
\end{equation*}
Then, $\{ f_n:n\geq 1\} $ is noise sensitive.
\end{proposition}

One can actually prove a more quantitative version of the previous result, that will be very useful in the forthcoming sections. Indeed, using \eqref{e:CovMehler} and Corollary \ref{c:sspoisson}, one derives
the estimate
\begin{equation}
\label{e:quantns1}
\BC[f_n(\eta),f_n(\eta^{t})] \leq {{} C}\, \delta_n \frac{e^{-t}}{(1-e^{-t})^2},
\end{equation}
from which we deduce that
\begin{equation}
\label{e:quantnsconv}
\BC[f_n(\eta),f_n(\eta^{t_n})] \to 0, \, \, \mbox{as $n \to \infty$.} 
\end{equation}
for any nonnegative bounded sequence $\{ t_n \}$ such that 
$\delta_n = o((1-e^{-t_n})^2)$, as $n \to \infty$.

}
{
\subsection{Noise stability}
\label{s:nstab}
We now define noise stability, and then prove a statement containing a criterion for it expressed in terms of chaos expansion (analogous to \cite[Proposition 4.3]{GarbanSteif14}), as well as a quantitative noise stability estimate similar to \cite[Proposition 6.1.2]{GarbanSteif14}. 

\begin{definition}\label{d:nstab}\rm Let $f_n : \bN \to \{0,1\}$ be a sequence of measurable mappings. The sequence $\{f_n\}$ is said to be {\em noise stable} if, for every $t>0$,
$$
\lim_{t \to 0} \lim_{n \to \infty} \BP[f_n(\eta) \neq f_n(\eta^t)] = 0,
$$
where $\eta^t$ has been defined in Section \ref{ss:ou}. 
\end{definition}
We recall that, if a sequence of mappings $f_n$, $n\geq 1$, takes values in a bounded set, then the random variables $f_n(\eta)$ admit a chaos decomposition: as before, we will denote by $u_k^{(n)}$ the $k$th kernel in the chaotic decomposition \eqref{e:chaos} of $f_n(\eta)$. The following statement contains the announced criterion for noise stability, as well as an extension to mappings with values in $\{-1,1\}$.

\begin{proposition}
\label{p:quantnstab}

\begin{itemize}
\item[\rm (a)] Let $f_n : \bN \to \{0,1\}$ be a sequence of measurable mappings. Then, the sequence $\{f_n\}$ is noise stable if and only if, for every $\epsilon > 0$, there exists $m$ such that for all $n$,
\begin{equation}
\label{e:nstabcrit}
\sum_{k=m}^{\infty}\BE[I_k(u^{(n)}_k)^2] < \epsilon.
\end{equation}

\item[\rm (b)] Let $f_n : \bN \to \{-1,1\}$ be a sequence of measurable mappings. If $\{ t_n : n\geq 1\}\subset \mathbb{R}_+$ is a sequence such that $t_n \int \BE[|D_xf_n(\eta)|^2] \lambda(\md x) \to 0$ as $n \to \infty$, then 
$$ \lim_{n \to \infty}  \BP(f_n(\eta) \neq f_n(\eta^{t_n})) = 0.$$

\end{itemize}

\end{proposition}
\begin{proof}

 Using the covariance formula \eqref{e:CovMehler} one infers that
\begin{align*}
\BP[f_n(\eta) \neq f_n(\eta^t)] &= \BE[f_n(\eta)(1-f_n(\eta^t))] + \BE[(1-f_n(\eta))f_n(\eta^t)] \\
& = 2( \BE[f_n(\eta)^2] - \BE[f_n(\eta)f_n(\eta^t)]) \\
& = 2\sum_{k=1}^{\infty} (1- e^{-kt})\BE[I_k(u^{(n)}_k)^2],
\end{align*}
which yields immediately Part (a). To prove Part (b), we first observe that, since each $f_n$ takes values in $\{-1,1\}$, then
$$ \sum_{k=0}^{\infty}\BE[I_k(u^{(n)}_k)^2] = 1,$$
that is: for every $n$, the weights $\BE[I_k(u^{(n)}_k)^2],\,\, k \geq 0$, constitute a probability distribution on $\{0,1,\ldots \}$. Combining the covariance formula \eqref{e:CovMehler} with Jensen's inequality, we deduce that, for each $t\geq 0$,
$$ \BE[f_n(\eta)f_n(\eta^t)] = \sum_{k=0}^{\infty}e^{-kt}\BE[I_k(u^{(n)}_k)^2] \geq \exp \{-t\sum_{k=0}^{\infty}k\BE[I_k(u^{(n)}_k)^2]\}.$$
Since
$$ \sum_{k=0}^{\infty}k\BE[I_k(u^{(n)}_k)^2]  = \int \BE[|D_xf_n|^2] \lambda(\md x),$$
we eventually obtain the lower bound%
$$  \BE[f_n(\eta)f_n(\eta^t)] \geq \exp \left\{-t\int \BE[|D_xf_n|^2] \lambda(\md x)\right\},$$
from which the desired conclusion follows at once.  
\end{proof}
}
From Propositions \ref{p:covariance} and \ref{p:quantnstab} and covariance formula \eqref{e:CovMehler}, it is easy to see that if $f_n$ is noise sensitive and noise stable then $f_n$ is asymptotically degenerate i.e., $\BV[f_n] \to 0$ as $n \to \infty$.
\section{Exceptional times and critical windows}
\label{s:dynbm}

{{} Throughout this section, we adopt the Markov process representation of the Ornstein-Uhlenbeck semigroup $\{T_t\}$, as put forward in Remark \ref{r:markovrep}.}

\medskip

We say that a sequence of Boolean functions $\{ f_n : n\geq 1\}$, $f_n : \bN\to \{0,1\}$ is {\em
  non-degenerate} if, for some $\gamma_0\in(0,1)$,
\begin{equation}
\label{e:deltadeg}
\gamma_0 \leq \BE[f_n(\eta)] \leq  1 - \gamma_0,\quad \, n\in\N.
\end{equation}
We say that a function $f\colon\bN\to\R$ is {\em jump regular}
(with respect to $\{ \eta^t : t\geq 0\}$) if 
the mapping $t\mapsto f(\eta^t)$ is almost surely piecewise constant, {{} that is, on any compact interval such a mapping equals a finite linear combination of indicators of bounded intervals of the form $[a , b)$}. We will abuse notation and denote the left-hand limit
of $t\mapsto f(\eta^t)$ at $t>0$ by $f(\eta^{t-})$ even though 
$\eta^{t-}$ may be not well defined. Observe that if $\lambda(\BX) < \infty$ and $f \colon\bN \to \{0,1\}$, then $f(\eta^t)$ is almost surely piecewise constant because of {{} the right-continuity of $\eta^t$} with respect to the discrete topology and this is sufficient to verify jump-regularity in many applications. Given a sequence of Boolean jump regular functions $f_n\colon\bN \to \{0,1\}$, $n\in\N$, we
define the {\em sets of exceptional times} as
$$ 
S_n := \{ t>0: f_n(\eta^t) \neq f_n(\eta^{t-}) \}, \quad n\in \N.
$$
One naturally expects that, if the sequence $\{f_n : n\geq 1\}$ is noise sensitive, then
$|S_n| \to \infty$ in some sense. Such an intuition is formally confirmed in the next statement, where we show that a non-degenerate and noise sensitive Boolean
function has an infinite set of exceptional times even in small time
intervals. This was proved in the specific case of critical percolation in
\cite[Lemma 5.1]{BKS99}, but the proof can be adapted in more
generality (see \cite[Theorem 1.23]{GarbanSteif14}). The proof of the next result follows closely the strategy developed in the aforementioned references.

\begin{theorem}[Existence of exceptional times]
\label{t:exctimdyn}
Let $f_n\colon\bN \to \{0,1\}$, $n\in\N$, be a sequence of jump regular
and non-degenerate {{} Boolean} functions. {{} Further, assume that there exists a sequence $c_n\in (0,1)$, $n\geq 1$, such that, for every $\beta\in (0,1)$, the sequence $b_n = \beta\, c_n$, $n\ge 1$, is such that
\begin{align}
\label{e:nscondn}
\BC[f_n(\eta),f_n(\eta^{b_n})] \to 0, \, \, \mbox{as $n \to \infty$.}
\end{align}
Then, $|S_n \cap [0,c_n] | \overset{\BP}{\longrightarrow} \infty$.}
\end{theorem}
\begin{proof}
  We claim that it suffices to show that for all
  $ 0 \leq a < b \leq 1$,
\begin{equation}
\label{e:snab}
\lim_{n \to \infty} \BP(S_n \cap [ac_n,bc_n] = \emptyset)= 0.
\end{equation}
To see this, assume \eqref{e:snab} is true, and fix an arbitrary integer $M \geq 1$ and
$\alpha > 0$. Partitioning $[0,{{} c_n}]$ into intervals of length $1/(2M)$,
we can choose $N$ large enough in order to have that, for all $n \geq N$,
$$ 
\BP\left(S_n \cap \left[0,\frac{c_n}{2M}\right] = \emptyset\right) \leq \alpha / 2.
$$
Set
$X_n := \sum_{l = 1}^{2M} \I\{S_n \cap [\frac{(l-1)c_n}{2M},\frac{lc_n}{2M}]= \emptyset\}$
and observe that $\BE[X_n] \leq \alpha M$ {{} by the homogeneous Markov property of $\eta^t$}. Using Markov's inequality,
$$ 
\BP({{} | S_n \cap [0,c_n]|} \leq M) \leq \BP(X_n \geq M) \leq \alpha.
$$
As $M,\alpha$ are arbitrary, the previous estimate completes the proof of the theorem assuming \eqref{e:snab}. We will now prove \eqref{e:snab} {{} for arbitrary $0\leq a< b\leq 1$ that we fix until the end of the proof}. For $\epsilon\in(0,1)$ and $n\in\N$, we set
$$ 
W_{n,\epsilon}=
\{ \mu \in \bN : \BP(f_n(\eta^{\epsilon}) = 1 \mid\eta = \mu) \in [0,\gamma_0{{} /2}] \cup [1 - \gamma_0{{} /2},1] \},
$$
where $\gamma_0$ is as in \eqref{e:deltadeg}.
Select an arbitrary $\gamma>0$, let $k\in\N$ and define
${{} b_n := (b-a)c_n/k}$. By the non-degeneracy condition
and part 2 of Proposition \ref{p:covariance}, one can find a large $n_0\in\N$ such that for
all $n \geq n_0$, $\BP(\eta \in W_{n,b_n}) \leq \gamma_0\gamma$;
from now on, we consider that $n\ge n_0$.
 Now take $t_0,\ldots,t_k$ such that
$ac_n = t_0 < t_1 \cdots < t_k = bc_n$
and $t_i - t_{i-1} = {{} b_n}$. 
For each $i\in\{1,\ldots,k\}$, we have that
\begin{align*}
\BP(S_n\cap(t_{0},t_i]=\emptyset)&\le \BP\big(\eta^{t_{i-1}}\in W_{n,{{} b_n}}\big)
+\BP\big(\eta^{t_{i-1}}\notin W_{n,{{} b_n}}, S_n\cap(t_{0},t_i]=\emptyset\big)\\ 
&\le \gamma_0\gamma
+\BP\big(S_n\cap(t_{0},t_{i-1}]=\emptyset,\eta^{t_{i-1}}\notin W_{n,{{} b_n}}, f_n(\eta^{t_{i-1}})=f_n(\eta^{t_{i}})\big), 
\end{align*}
where we have used that $f_n$ is jump regular.
Using the Markov property of $\eta^t$, we obtain that
\begin{align*}
\BP&(S_n\cap(t_{0},t_i]=\emptyset)\\
&\le \gamma_0\gamma+\BE\big[\I\{S_n\cap(t_{0},t_{i-1}]=\emptyset,
\eta^{t_{i-1}}\notin W_{n,{{} b_n}}\} 
\BP(f_n(\eta^{t_i}) = f_n(\eta^{t_{i-1}}) | \eta^{t_{i-1}} ) \big].
\end{align*}
By definition of $W_{n,{{}b_n}}$ (and since $f_n$ is Boolean) we have that, if 
$\eta^{t_{i-1}} \notin W_{n,{{}b_n}}$,
\begin{align*}
\BP(f_n(\eta^{t_i}) = f_n(\eta^{t_{i-1}}) | \eta^{t_{i-1}} ) \le 1-\gamma_0{{} /2}.
\end{align*}
This yields the bound
\begin{align*}
\BP(S_n\cap(t_{0},t_i]=\emptyset)\le \gamma_0\gamma
+(1-\gamma_0{{} /2})\BP(S_n\cap(t_{0},t_{i-1}]=\emptyset),
\end{align*}
and an induction argument allows one to infer that
\begin{align*}
\BP(S_n\cap(ac_n,bc_n]=\emptyset)
\le \gamma_0\gamma\sum^k_{i=0}(1-\gamma_0{{} /2})^i
+(1-\gamma_0/2)^k
\le \gamma+(1-\gamma_0{{} /2})^k,
\end{align*}
where the second inequality follows from a standard calculation.
Letting $k\to\infty$, we obtain that, for all $n\ge n_0$, 
$\BP(S_n\cap(ac_n,bc_n]=\emptyset)\le \gamma$,
and \eqref{e:snab} is proved.
\end{proof}

Combining Theorem \ref{t:exctimdyn} with \eqref{e:quantnsconv}, one deduces the following easily applicable corollary.
\begin{corollary}
\label{c:randalgdyn}
Let $f_n\colon\bN \to \{0,1\}, n\geq 1$, be a sequence of jump regular and non-degenerate mappings, and assume that there exist randomized stopping sets $Z_n$, $n\ge 1$, such that $Z_n$ determines $f_n(\eta)$ and $Z_n$ verifies \eqref{e:finitapp} for every $n$, and moreover
\begin{align*}
\delta_n := \delta(Z_n) := \sup_{x\in \BX} \BP(x\in Z_n)\longrightarrow 0, \quad \, \mbox{as $n\to\infty$}.
\end{align*}
Then, for any sequence $c_n\in (0,1)$ such that
${{} \delta_n} = o(c_n^2)$, we have
that $|S_n \cap [0,c_n]| \overset{\BP}{\longrightarrow} \infty$.
\end{corollary}

Following the ideas in the proof of \cite[Theorem 1.10]{Garban19}, we
use the above result to bound the critical window for phase-transition of a
monotonic Boolean function.
To smoothly formulate the result it is convenient
to introduce for any $\sigma$-finite measure $\nu$ on $\BX$ a
Poisson process $\eta_\nu$ with that intensity measure.
We shall use this notation only for measures $\nu$ that are multiples of the intensity measure $\lambda$.
\begin{theorem}[Bounds on critical window]
\label{t:sharpphase}
Let $f_n\colon \bN \to \{0,1\}, n\geq 1$, be a sequence of increasing, jump regular  and non-degenerate mappings. 
Assume that there exist randomized stopping sets $Z_n$, $n\ge 1$, such that $Z_n$
determines $f_n(\eta)$ and $Z_n$ verifies \eqref{e:finitapp} for every $n$, and moreover
\begin{align}\label{e987}
\delta_n := \delta(Z_n) := \sup_{x\in \BX} \BP(x\in Z_n)\to 0, \quad \mbox{as $n\to\infty$}.
\end{align}
Then, for any sequence $c_n {{} \in (0,1)}$ such that
${{} \delta_n} = o(c_n^2)$, we have
that $\BE[f_n(\eta_{(1+c_n)\lambda})] \to 1$ and $\BE[f_n(\eta_{(1-c_n)\lambda})] \to 0$,  as $n \to \infty$.
\end{theorem}
We remark that the proof only uses the fact that $|S_n \cap [0,c_n]| \overset{\BP}{\longrightarrow} \infty$ where $S_n$ is the set of exceptional times. 
\begin{proof}
Firstly, from Corollary \ref{c:randalgdyn}, we have that 
\begin{align}\label{e-3456}
\BP(f_n(\eta^t) = 1 \, \mbox{for some $t \in [0,c_n]$}) \to 1 \, \mbox{as $n \to \infty$}.
\end{align}
For $t\ge 0$, let $\eta'_t$ be the Poisson process
of points born before time $t$. This process is independent
of $\eta$ and has intensity measure $t\lambda$.
Then $\zeta_t:=\eta+\eta'_t$ is a Poisson process with intensity measure
$(1+t)\lambda$. By construction of $\eta^t$ and $\zeta_t$, we have
$\eta^t\le \zeta_t\le \zeta_{c_n}$ whenever $t\le c_n$.
By the monotonicity of $f_n$ and \eqref{e-3456} this implies
$\BE[f_n(\zeta_{c_n})]\to 1$ as $n\to\infty$ and hence the first assertion. Again, from Corollary \ref{c:randalgdyn}, we have that 
\begin{align}\label{e-3457}
\BP(f_n(\eta^t) = 0 \, \mbox{for some $t \in [0,c_n]$}) \to 1 \, \mbox{as $n \to \infty$}.
\end{align}
For $t\ge 0$, let $\zeta'_t$ be the point process of points from $\eta$
(counting multiplicities) which are still alive at time $t$. Since the lifetimes of points are exponential, $\zeta'_t$ is a Poisson process with intensity measure $e^{-t}\lambda$ and also trivially by construction $\zeta'_{c_n} \le \zeta'_t \le \eta^t$ whenever $t \leq c_n$. Thus, by the monotonicity of $f_n$ and \eqref{e-3457}, we obtain that $\BE[f_n(\eta_{e^{-c_n}\lambda})]= \BE[f_n(\zeta'_{c_n})]\to 0$ as $n\to\infty$. Now, the second assertion follows by first observing that $1- t \leq e^{-t}$ for $t \geq 0$ and then using the monotonicity of $f_n$ along with the thinning property of the Poisson process. 
\end{proof}
We now prove a strengthening of the above result when the
  randomized stopping sets $Z_n$ above can be attained by a CTDT as in
  Theorem \ref{t:POSSS}.  This proof proceeds similar to the
  Friedgut-Kalai sharp threshold theorem (see \cite[Theorem
  3.5]{GarbanSteif14}).  This refinement was pointed out to us by Stephen Muirhead. 

\begin{theorem}[Bounds on critical window]

\label{t:fksharp}
Let $f_n\colon \bN \to \{0,1\}, n\geq 1$, be 
as in Theorem \ref{t:sharpphase}. Assume for each $n\in\N$
that $f_n$ is determined by a randomized stopping set $Z_n$, satisfying
the assumptions of Theorem \ref{t:POSSS}.  Assume that
\eqref{e987} holds. 
Then, for any sequence $c_n\in (0,1)$ such that
${{} \delta_n} = o(c_n)$,  we have
that $\BE[f_n(\eta_{(1+c_n)\lambda})] \to 1$ and $\BE[f_n(\eta_{(1-c_n)\lambda})] \to 0$,  as $n \to \infty$.
\end{theorem}
We shall actually prove a stronger quantitative version than the one stated. 
\begin{proof}
  Set $p_n(1+t) := \BE[f_n(\eta_{(1+t)\lambda})]$ for $t \in [-1,\infty) $.  Let
  $\epsilon > 0$. We will show that, for $n$ large enough,
\begin{equation}
\label{e:fkquant}
 \mbox{if } \,  \,p_n(1) \geq \gamma_0 \, \mbox{ then } \, p_n(1+t_n) \geq 1 - \epsilon \, \mbox{ for } \, t_n = \frac{-\delta_n\log(\gamma_0)}{\epsilon + \delta_n\log(\gamma_0)}.
 \end{equation}
 This suffices to show that $\BE[f_n(\eta_{(1+c_n)\lambda})] \to 1$
 for any sequence $c_n\in (0,1)$ such that ${{} \delta_n} = o(c_n)$.
 Applying the same argument to $1 - p_n(1-t)$ will yield the second
 claim in the theorem. For given $\epsilon > 0$ let $n$ be large enough such that
$t_n \geq 0$ in \eqref{e:fkquant}. This is possible because
$\delta_n \to 0$ as $n \to \infty$. Using the Russo-Margulis formula for Poisson functionals (see \cite[Theorem 19.4]{LastPenrose17}),  monotonicity of $f_n$ and the Poisson-OSSS inequality (Theorem \ref{t:POSSS}),  we derive that for $t \geq 0$, 
\begin{align*}
\frac{\md p_n(1+t)}{\md t} & = \int_{\BX} \BE[D_xf_n(\eta_{(1+t)\lambda})]\lambda(\md x) \\
& =  \int_{\BX} \BE[|D_xf_n(\eta_{(1+t)\lambda})|]\lambda(\md x) \\
& \geq \frac{1}{(1+t)\delta_n} \BV(f_n(\eta_{(1+t)\lambda})) \\
& =  \frac{1}{(1+t)\delta_n} p_n(1+t)(1- p_n(1+t)). 
 \end{align*}
Suppose that $p_n(1) \geq \gamma_0$ and $p_n(1+t_n)< 1 - \epsilon$ for $t_n$ as defined in \eqref{e:fkquant}.  Since $f_n$ is monotonic,  we have that $p_n(1+t) < 1 - \epsilon$ for all $t \leq t_n$ and hence using this in the above differential inequality,  we obtain that for all $t \in [0,t_n]$, 
$$ \frac{\md \log p_n(1+t)}{\md t} \geq \frac{\epsilon}{(1+t_n)\delta_n}.$$
Exploiting the fact that $p_n(1+t) \geq p_n(1) \geq \gamma_0$, one infers that
$$ \log p_n(1+t_n) \geq \log \gamma_0 + \frac{\epsilon\,  t_n}{(1+t_n)\delta_n}  =  0,$$
which implies $p_n(1+t_n) \geq 1$ and therefore a contradiction.  So,  we have that $p_n(1+t_n) \geq 1 - \epsilon$ as required by \eqref{e:fkquant}.  This completes the proof of \eqref{e:fkquant} and hence the theorem as well.
\end{proof}
\part{Applications to Continuum percolation models}
\label{p:applns}
In this part, we fix $d \geq 2, k \geq 1$ and $r > 0$. Let
$\eta = \{X_i : i\geq 1\}$ be the stationary Poisson point process in
$\R^d$ of intensity $\gamma > 0$ (identified as usual with its
  support).  We write $ \cK_r$ to indicate the space of non-empty
compact subsets of $B_r(0)$ equipped with the Hausdorff distance and
$\cK^b_r := \{B_s(0) : s \in [0,r]\}$ denote the subset of $\cK$
consisting of balls centred at the origin. Let $\BQ$ be a probability
measure on compact sets of $\R^d$ such that it satisfies
\begin{equation}
\label{e:assumq}
\BQ(\cK_r) = 1 \, \, \mbox{and} \, \, \BQ\{ K : B_{r_0}(0) \subset K \} > 0, \tag{8.0}
\end{equation}
for two constants $r, r_0 > 0$. The two assumptions of compactly supported and containing small balls ensure non-triviality of percolation phase transition in our models. The former is a strong assumption and in the case of unbounded grains,  some of the results could change depending on the distribution of grain sizes;  see \cite{Ahlbergsharp18,DCRT18} for results in the case of balls with unbounded radii. In Section \ref{s:pbmunbdd} alone,  we shall work with unbounded balls and this already indicates that the first assumption can be removed with some additional work under suitable moment assumptions on the size of the grains.
\section{$k$-Percolation in the Poisson Boolean model}
\label{s:kpercbm}
We will now consider the $k$-percolation model. We denote the marked point process by $\teta = \{(X_i,M_i)\}_{i \geq 1}$ where $M_i, i \geq 1$ are i.i.d.\ compact sets (or also referred to as {\em grains}) distributed according to $\BQ$. We shall use symbol $M_i$'s for random grains and $K$ for deterministic grains. The {\em $k$-covered or $k$-occupied region of the Poisson Boolean model} on $\teta$ is defined as
\begin{equation}
\label{e:defbm}
\cO(\gamma) := \cO_k(\teta) = \bigcup_{1\leq i_1< \ldots < i_k < \infty} (X_{i_1}  + M_{i_1}) \cap \ldots \cap (X_{i_k} + M_{i_k})
\end{equation}
Apart from being a natural extension of the usual continuum percolation, $k$-percolation can also be seen as percolation of $(k-1)$-faces in the random C\v{e}ch complex on $\eta$ with i.i.d.\ grains distributed according to $\BQ$ (see \cite[Remark 3.8]{BY2013}). 

We have suppressed the dependence on $k,\BQ$ in the above definition as $\BQ$ and $k$ will remain fixed but $\gamma$ will vary.  If $k = 1$, this is the classical Boolean model. Study of continuum percolation was initiated by Gilbert \cite{Gilbert61} where he considered the Poisson Boolean model with fixed radii balls and much later, it was extended to random radii with unbounded support by Hall \cite{Hall85}. We refer the reader to the monograph of Meester and Roy \cite{MeesterRoy96} as well as that of Bollobas and Riordan \cite{Bollobas06} for detailed accounts on continuum percolation. 

We now define the percolation events and the corresponding probabilities. For connected subsets $A,B,R \subset \R^d$ such that $A \cup B \subset R$, we define
\begin{align}
\{A \stackrel{R}{\leftrightarrow} B\} &:= \{ \mbox{there exists a path in $\cO(\gamma)\cap R$ from $x$ to $y$ for some $x \in A, y \in B$} \}, \nonumber \\
\{A \leftrightarrow B \} &:= \{ A \stackrel{\R^d}{\leftrightarrow} B \}, \nonumber \\
\{0 \leftrightarrow \infty\} &:= \{\mbox{there is an unbounded path from $0$ in $\cO(\gamma)$} \},  \nonumber \\
\theta_s(\gamma) &:= \BP(0 \leftrightarrow \partial B_s(0)), s > 0,  \nonumber  \\
\theta(\gamma) &:= \BP(0 \leftrightarrow \infty),  \nonumber \\
Arm_{s,t}(\gamma) &:= \{ B^{\infty}_s(0) \leftrightarrow \partial B^{\infty}_t(0) \}, \, \, 0 < s < t, \nonumber \\
Cross_{s,t}(\gamma) &:= \{ \{0\} \times [0,t]^{d-1} \stackrel{[0,s] \times [0,t]^{d-1}}{\leftrightarrow} \{s\} \times [0,t]^{d-1} \}, \, \, s, t > 0,
 \label{e:defpercprob1}
\end{align}
where in the fourth and fifth definitions, we have replaced the singleton set $\{0\}$ by $0$ for convenience and $B^{\infty}_s(x) = x \oplus [-s,+s]^d$ is the $\ell_{\infty}$-ball of side-length $2s$ at $x$  where $\oplus$ denotes Minkowski sum of sets. $Arm_{s,t}(\gamma)$ is the usual {\em one-arm event} and $Cross_{\kappa t,t}(\gamma)$ is the {\em crossing event}. We say that the origin $k$-percolates if $0 \leftrightarrow \infty$ holds and $\theta(\gamma)$ is the {\em percolation probability}. 

We now define the critical intensity for $k$-percolation. Since the event $\{0 \leftrightarrow \partial B_n(0)\}$ is decreasing in $n$, we have that 
$$ \theta(\gamma) = \lim_{n \to \infty}\theta_n(\gamma).$$

The {\em critical intensity} of the model is defined as
\begin{equation}
\label{e:critint1}
\gamma_c := \gamma_c(\BQ) =  \inf \{ \gamma \geq 0 : \theta(\gamma) > 0 \}.
\end{equation}
The above definition is justified because $\theta(\gamma)$ is monotonically increasing in $\gamma$. 
Further, let $C_0$ be the connected component containing $0$ in $\cO(\gamma)$ if $0 \in \cO(\gamma)$ else $C_0 = \emptyset$. We define two other critical intensities related to percolation of $\cO(\gamma)$ :
\begin{align}
\label{e:defpercprob2}
\hat{\gamma}_c := \sup  \{ \gamma \geq 0 : \BE[|C_0|] < \infty \}, 
\tilde{\gamma}_c := \inf \{ \gamma \geq 0 : \inf_{s > 0} \BP(B_s(0) \leftrightarrow \partial B_{2s}(0)) > 0 \},
\end{align}
where $|.|$ denotes the Lebesgue measure of a set. Again, the definitions are justified by the monotonicity in $\gamma$ of the respective quantities. We have that $\hat{\gamma}_c \leq \gamma_c$ (using the first equality in \eqref{e:ECO})  and that $\tilde{\gamma}_c \leq \gamma_c$ follows via a simple monotonicity argument. While $\hat{\gamma}_c$ is a natural notion of critical intensity, $\tilde{\gamma}_c$ is often very useful in initiating renormalization arguments and was introduced in \cite{Gouere14}.

Our first main result is that the three critical intensities are equal.
\begin{theorem}[Equality of critical intensities]
\label{t:eqcritprob}
Let $\cO(\gamma)$ be the $k$-covered region of the Poisson Boolean model as defined in  \eqref{e:defbm} with grain distribution $\BQ$ satisfying assumptions in \eqref{e:assumq}. Then we have that $\gamma_c = \hat{\gamma}_c = \tilde{\gamma}_c \in (0,\infty)$. 
\end{theorem}
Similar to the Poisson Boolean model, one can associate various other critical densities using box crossing events, diameter, number of grains in $C_0$ (see \cite[Sections 3.4 and 3.5]{MeesterRoy96}). When $\BQ$ is supported on deterministic bounded balls, equality of some of these critical intensities are shown in \cite[Theorems 3.4 and 3.5]{MeesterRoy96} for $k = 1$. In \cite[Theorem 1.2]{DCRT18}, this was extended to the case of unbounded balls with the radius distribution having finite $5d-3$ moments. The above theorem for $k = 1$ also follows straightforwardly from the results of \cite[Theorem 3.1]{Ziesche2018}. Similar to \cite[Theorem 3.1]{Ziesche2018}, we prove an exponential decay bound for $\theta_n(\gamma)$ to deduce equality of critical intensities. One of the advantages of our proof is that it yields a mean-field lower bound in the super-critical regime also easily.
\begin{theorem}[Sharp threshold for $k$-percolation]
\label{t:bmexpdecay}
Let $\cO(\gamma)$ be the $k$-covered region of the Poisson Boolean model as defined in  \eqref{e:defbm} with grain distribution $\BQ$ satisfying assumptions in \eqref{e:assumq}. Then the critical intensity $\gamma_c$ defined in \eqref{e:critint1} is non-degenerate (i.e., $\gamma_c \in (0,\infty)$) and the following statements hold.
\begin{enumerate}
\item[(i)] For all $\gamma < \gamma_c$, there exists a constant $\alpha(\gamma) \in (0,\infty)$ such that $\theta_s(\gamma) \leq \exp\{-\alpha(\gamma)s\}$ for all $s > 0$. 
\item[(ii)] For any $b > \gamma_c$, there exists a a constant $\alpha(b) \in (0,\infty)$ such that for all $\gamma \in (\gamma_c,b)$, $\theta(\gamma) \geq \alpha(b)(\gamma - \gamma_c)$.
\end{enumerate}
\end{theorem}
The above theorem for $k = 1$ was also proven in \cite[Theorem 3.1]{Ziesche2018} generalizing the result of \cite [Lemma 3.3]{MeesterRoy96} which was for the case of $\BQ$ supported on balls. Again, for the case of unbounded balls with finite $5d-3$ moments of the radius distribution, part (ii) was shown to hold in \cite[Theorem 1.2]{DCRT18}. Further, if the radius distribution has finite exponential moments, part (i) of the above theorem was shown in \cite[Theorem 1.4]{DCRT18}. It was also shown in \cite[Theorem 1.5]{DCRT18} that radius distribution with slower decay of tails can exhibit a different behaviour in the subcritical regime.

Traditionally, the proof of sharp phase transition in percolation models have relied upon adaptation of Menshikov's arguments \cite{Menshikov1986,Menshikov1986a}. For example,  see Meester and Roy \cite[Sections 3.4 and 3.5]{MeesterRoy96} for sharp phase transition in subcritical Poisson Boolean models ($k = 1$ in our $k$-percolation model). There is also an independent proof by Aizenmann and Barsky \cite{Aizenman1987}. Another simpler proof that emerged in recent years is that of Duminil-Copin and Tassion \cite{DCT16}. This proof is by showing sharp phase transition with respect to a new critical intensity which is defined by existence of a set containing the origin such that the expected number of paths exiting the set is strictly bounded above by $1$.  This was adapted to the Poisson Boolean model by Ziesche \cite{Ziesche2018}.  Another proof of sharp phase transition using randomized algorithms was pioneered in \cite{DCRT19} and also applied to continuum models in \cite{DCRT19b} and \cite{DCRT18} via suitable discretization.  Our proof technique based on CTDTs was inspired by these works and possibly enables a much easier execution of the approach initiated in \cite{DCRT19} for continuum percolation models.  Also, observe that the critical intensities $\gamma_c$ are increasing in $k$ and possibly even strictly increasing and so Theorems \ref{t:eqcritprob} and \ref{t:bmexpdecay} do not follow from the corresponding results for $k = 1$. 

We postpone the proof of Theorem \ref{t:bmexpdecay} to the end of the section and now show how the proof of Theorem \ref{t:eqcritprob} follows from Theorem \ref{t:bmexpdecay}.
\begin{proof}[Proof of Theorem \ref{t:eqcritprob}]
Observe that by Fubini's theorem, we have from Theorem \ref{t:bmexpdecay}(i) for $\gamma < \gamma_c$
\begin{equation}
\label{e:ECO}
 \BE[|C_0|] = \int_{\R^d} \BP(0 \leftrightarrow x) \md x \leq \int_{\R^d}\exp\{-c_{\gamma}|x|\} \md x < \infty.
 \end{equation}
This with the trivial bound yields that $\hat{\gamma}_c = \gamma_c$. 

 For the second part, consider $s > 10$ and $l_s \in \N$, such that $x_1,\ldots,x_{l_s} \in \partial B_s(0)$ with $\partial B_s(0) \subset \cup_{i=1}^{l_s} B_1(x_i)$.  Note that $l_s$ can be chosen such that $l_s \leq cs^{d-1}$ for some $c >0$ not depending on $s$. For $t > 0$, define $A_t(x) := \{x \leftrightarrow \partial B_{2t}(x) \} \cap \{ \partial  B_{2t}(x) \subset \cO(\gamma) \}$ i.e., $x$ is connected to the boundary of $\partial B_{2t}(x)$ and $\partial B_{2t}(x)$ is contained in a single component. Observe that $A_t(x)$ is an increasing event for all $t > 0$ and also by assumption \eqref{e:assumq}, we have that $\BP(A_t(x)) = \BP(A_t(0)) \geq \BP(B_{2t}(0) \subset \cO(\gamma)) > 0 $ for all $t > 0$.  Now,  using Markov's inequality, isotropy of the Poisson point process, monotonicity and positivity of $A_{1/2}(x_1)$, Harris-FKG inequality \cite[Theorem 20.4]{LastPenrose17} and the bound on $l_s$, we can derive that
\begin{align*}
\BP(B_s(0) \leftrightarrow \partial B_{2s}(0))  & \leq \BP(\cup_{i=1}^{l_s} \{B_1(x_i) \leftrightarrow \partial B_{2s}(0) \})  \leq \sum_{i=1}^{l_s} \BP(B_1(x_i) \leftrightarrow \partial B_{2s}(0)) \\
& \leq l_s\BP(B_1(x_1) \leftrightarrow \partial B_{2s}(0)) \\
& \leq l_s\BP(A_{1/2}(x_1))^{-1}\BP(\{B_1(x_1) \leftrightarrow \partial B_{2s}(0)\} \cap A_{1/2}(x_1)) \\
& \leq l_s \BP(A_{1/2}(0))^{-1} \BP(x_1 \leftrightarrow \partial B_{2s}(0)) \leq  \BP(A_{1/2}(0))^{-1}l_s \theta_s(\gamma) \\
& \leq c\BP(A_{1/2}(0))^{-1} s^{d-1}\theta_s(\gamma).
\end{align*}
Now using Theorem \ref{t:bmexpdecay}(i) again, we obtain that for all $\gamma < \gamma_c$
$$ \limsup_{s \to \infty} \BP(B_s(0) \leftrightarrow \partial B_{2s}(0))  \leq c\BP(A_{1/2}(0))^{-1} \lim_{s \to \infty} s^{d-1}\theta_s(\gamma) = 0.$$
Thus we have that $\gamma_c \leq \tilde{\gamma}_c$ and combined with the trivial inequality in the other direction, this completes the proof. 
\end{proof}
We now give a simple criterion for verifying noise sensitivity, existence of infinite exceptional times as well as determining the critical window in the above models. 
\begin{theorem}[Noise Sensitivity Criteria]
\label{t:nscritbm}
Let $\cO(\gamma)$ be the $k$-covered region of the Poisson Boolean model as defined in \eqref{e:defbm} with grain distribution $\BQ$ satisfying assumptions in \eqref{e:assumq}. Let $\kappa > 0$ be given. Let $f_n := f_n(\gamma) = \1\{Cross_{\kappa n,n}(\gamma)\}$ where $Cross_{\kappa n,n}(\gamma)$ is the crossing event defined in \eqref{e:defpercprob1}.  Assume that $\{f_n(\gamma_c)\}_{n \geq 1}$ is non-degenerate as in \eqref{e:deltadeg} and $\BP(Arm_{r,s}(\gamma_c)) \to 0$ as $s \to \infty$. Then we have the following :
\begin{enumerate}
\item[\rm (i)] $\{f_n(\gamma_c)\}_{n \geq 1}$ is noise sensitive. 

\item[\rm (ii)] $\{f_n(\gamma_c)\}_{n \geq 1}$ has an infinite set of exceptional times in $[0,1]$ i.e., $|S_n \cap [0,1]| \overset{\BP}{\longrightarrow} \infty$  where the set of exceptional times $S_n$ is as defined below \eqref{e:deltadeg}.

\item[\rm (iii)] Further, assume that there exists $c_{arm} > 0$ such that for all $s > 2r$, the following holds :
$$ \BP(Arm_{r,s}(\gamma_c))  \leq C s^{-c_{arm}}.$$
Let $c_n = n^{-\frac{c_{arm}}{2} + \epsilon}$ for some $\epsilon >0$. Then, we have that $\BE[f_n((1-c_n)\gamma_c)] \to 0$ and $\BE[f_n((1+c_n)\gamma_c)] \to 1.$
\end{enumerate}
\end{theorem}
Before presenting a corollary of the above to the standard planar Poisson Boolean model, we make some remarks about the theorem. 
\begin{remark}
\label{r:exns}
{\rm
\begin{enumerate}

\item Using Theorem \ref{t:fksharp} one can actually prove that the conclusion of Theorem \ref{t:nscritbm}-(iii) continues to hold for sequences of the type $c_n = n^{-c_{arm}+\epsilon}$, with $\epsilon>0$. In order to obtain such an improvement, one has to prove that the randomized stopping set $Z_n$ determining $f_n$, as appearing in the proof, can be attained by a CTDT satisfying the assumptions in Theorem \ref{t:POSSSmarked}. This fact can be checked by a recursive construction similar to the one performed at the end of the forthcoming proof of Theorem \ref{t:bmexpdecay}. Details are omitted.

\item We can use $\theta_s(\gamma_c) \to 0$ as $s \to \infty$ in the proof instead of $\BP(Arm_{r,s}(\gamma_c)) \to 0$ as $s \to \infty$ and in item (iii), we can use that $\theta_s(\gamma_c) \leq C s^{-c_{arm}}.$  We can argue this as follows : Set $A^{\infty}_r(x) := \{x \leftrightarrow \partial B^{\infty}_{r}(x) \} \cap \{ \partial B^{\infty}_{r}(x) \subset \cO(\gamma) \}$ and note that for all $x$, we have that $\BP(A^{\infty}_r(x)) = \BP(A^{\infty}_r(0))> 0$. Since $A_r(x)$ and $\{B^{\infty}_{r}(x) \leftrightarrow L\}$ are increasing events, by Harris-FKG inequality for Poisson point process \cite[Theorem 20.4]{LastPenrose17}, we derive that
\begin{align*}
\BP(Arm_{r,s}(\gamma_c)) & \leq \BP(A^{\infty}_r(0))^{-1} \BP(A^{\infty}_r(0) \cap \{B^{\infty}_{r}(0) \leftrightarrow \partial B^{\infty}_{s}(0)\}) \\
& \leq \BP(A^{\infty}_r(0))^{-1} \BP(0 \leftrightarrow \partial B^{\infty}_{s}(0)) \\
& \leq  \BP(A^{\infty}_r(0))^{-1} \BP(0 \leftrightarrow \partial B_{s}(0)) \\
& \leq \BP(A^{\infty}_r(0))^{-1} \theta_{s}(\gamma_c).
\end{align*}

\item The above theorem has reduced the proof of noise sensitivity,  exceptional times and sharp phase transition for crossing events at criticality in the $k$-percolation model to showing non-degeneracy of the crossing events and bounds for one-arm probabilities or non-percolation at criticality.  Showing these properties is a seperate percolation theoretic question, often model-specific and in the planar case, these have been achieved in some models via RSW-type estimates (see \cite{Roy90, Roy91, Alexander96,Tassion2016,Ahlbergsharp18,Iyer2019,LRM2019}). See Corollary \ref{c:nsplanarbm} below for a case in which the above estimates are known.

\item Another case in which it is known that
  $\theta_s(\gamma_c) \to 0$ as $s \to \infty$ is for $k = 1$, determinstic balls  and 
  large $d$ (see \cite[Section 6]{Heydenreich19}). 


\item It was suggested in \cite[Section 8]{ABGM14} on how one may use the Schramm-Steif quantitative noise sensitivity result on the Boolean hypercube and use discretization to obtain noise sensitivity exponents. As illustrated in the above theorem, the Poisson analogue of Schramm-Steif inequality (Corollary \ref{c:sspoisson} that we apply via Proposition \ref{p:randalgns}) helps to achieve these goals by using stopping sets and without resorting to discretization.  Furthermore, this makes the results more easily applicable and also gives a transparent way to quantify the noise sensitivity exponents in terms of the exponent $c_{arm}$ in arm-event probabilities. 


\end{enumerate}
}
\end{remark}

For the case of planar Poisson Boolean model (i.e., $k = 1, d= 2$) with grains supported on balls centred at origin, the assumptions of the above theorem follow immediately from \cite[Theorems 1.1(ii) and 1.3(i)]{Ahlbergsharp18} and hence we derive the following corollary easily from Theorem \ref{t:nscritbm}.
\begin{corollary}
\label{c:nsplanarbm}
Consider the $1$-covered region $\cO_1(\teta)$ of the Poisson Boolean model on $\R^2$ (i.e., $k = 1,d=2$ in Theorem \ref{t:nscritbm}) as defined in \eqref{e:defbm} with grain distribution $\BQ$ supported on $\cK^b_r$ and satisfying assumptions in \eqref{e:assumq}. Then for $f_n := f_n(\gamma_c)$ as defined in Theorem \ref{t:nscritbm}, the conclusions of the Theorem \ref{t:nscritbm} hold.
\end{corollary}
The above corollary settles \cite[Conjecture 9.1]{ABGM14} and our Theorem \ref{t:nscritbm} gives percolation theoretic criteria to prove \cite[Conjecture 9.2 and Question 1]{ABGM14}. See the discussion at the end of Section \ref{ss:introapp} for comparison with approach of \cite{ABGM14}.

\begin{proof}[Proof of Theorem \ref{t:nscritbm}]
We shall assume $\gamma = \gamma_c$ and fix a $\kappa > 0$ in our proof. Let $R_n = [0,\kappa n] \times [0,n]^{d-1}$ be the rectangle to be crossed. Our proof is by constructing a randomized stopping set $\tilde{Z}_n := Z_n \times \cK_r \subset R^+_n \times \cK_r, n \geq 10r$ with $R^+_n = R_n \oplus B_r(0)$ as follows : For every $s \in (0,\kappa n)$, we shall construct a stopping set $\tilde{Z}_n^s = Z_n^s \times \cK_r$ such that for all $x = (x_1,\ldots,x_d) \in R_n$,
\begin{align}
\BP((x,K) \in \tilde{Z}_n^s) &= \BP(x \in Z_n^s) \nonumber \\
& \leq \1\{|s-x_1| \leq r\} + \BP(Arm_{r,|s-x_1|}(\gamma_c))\1\{|s-x_1| >r\} \nonumber \\
\label{e:revprobkpercns}  & \leq \BP(Arm_{r,|s-x_1|}(\gamma_c)),
\end{align}
where the last inequality is by setting the convention that $Arm_{r,s}(\gamma_c) = 1$ for $s \leq r$. Now, we choose our randomized stopping set by first choosing $Y$ uniformly at random in $(0,\kappa n)$ and then setting $Z_n = Z_n^Y, \tilde{Z}_n = \tilde{Z}_n^Y  = Z_n^Y \times \cK_r$. Thus, we obtain that \\ 
\begin{align*}
 \BP((x,K) \in  \tilde{Z}_n) &= \BP(x \in Z_n) \leq (\kappa n)^{-1} \int_0^{\kappa n} \BP(Arm_{r,|s-x_1|}(\gamma_c)) \md s \\
 & \leq 2(\kappa n)^{-1} \int_0^{\kappa n} \BP(Arm_{r,s}(\gamma_c)) \md s. 
 \end{align*}
Thus, we derive that 
$$\delta_n := \delta( \tilde{Z}_n) \leq 2(\kappa n)^{-1} \int_0^{\kappa n} \BP(Arm_{r,s}(\gamma_c)) \md s$$
and since $\BP(Arm_{r,s}(\gamma_c)) \to 0$ as $s \to \infty$, we have that $\delta_n \to 0$ by l'H\^{o}pital's rule. Note that $f_n$ is jump-regular as the intensity measure of $\teta \cap ( (R_n \oplus B_0(r)) \times \cK_r)$ is finite and also $f_n$ is monotonic. Now all the three items in the theorem follow from Proposition \ref{p:randalgns}, Corollary \ref{c:randalgdyn} and Theorem \ref{t:sharpphase} respectively. 

 Thus, we are left with construction of a stopping set $\tilde{Z}_n^s$ as above or equivalently $Z_n^s$ with revealment probability as in \eqref{e:revprobkpercns}. Given $s \in (0,\kappa n)$, let $L = \{s\} \times [0,n]^{d-1}$ and let $S$ be union of the connected components of $\cO(\teta) \cap R_n$ that intersect $L$. Note that $S = \emptyset$ if no such components exist. Define $Z_n^s := (S \cup L) \oplus B_{r}(0)$. 

 Now we show that $Z_n^s$ is a stopping set in $\R^d$ as in Remark \ref{rA1}. Let $x = (x_1,\ldots,x_d) \in \R^d$. Suppose that $B_{r}(x) \cap L \neq \emptyset$ (or equivalently $|x_1 - s| \leq r$), then $x \in Z_n^s$. Otherwise, $x \in Z_n^s$ and $B_{r}(x) \cap L = \emptyset$ i.e., $|x_1 - s| > r$. Then, we have that $B_{r}(x) \cap S \neq \emptyset$ which is equivalent to $B_r(x) \stackrel{R_n}{\leftrightarrow} L$. Thus, we derive that
\begin{equation}
\label{e:graphmble-nsbm}
\1\{x \in Z_n^s\} = \1\{B_{r}(x) \cap L \neq \emptyset\} + \1\{B_{r}(x) \cap L = \emptyset\}\1\{B_r(x) \stackrel{R_n}{\leftrightarrow} L \}.
\end{equation}
The above identity yields graph-measurability of $Z_n^s$. Observe that $S$ as defined above is actually $S(\teta)$ to be more explicit. Set $S'(\teta) = S(\teta) \oplus B_r(0)$. Then $S(\teta) = S(\teta_{S' \times \cK_r})$ because $(x,K) \in (S')^c \times \cK_r$ implies that $(x + K) \cap S = \emptyset$. Thus,  $S'(\teta) = S'(\teta_{S' \times \cK_r})$. In the same way, we can also deduce that $S'(\teta_{S' \times \cK_r}) = S'(\teta_{S' \times \cK_r} + \psi_{(S')^c \times \cK_r})$ for any $\psi \in \bN(\R^d \times \cK_r)$. This verifies \eqref{estopset} and shows that $Z_n^s$ is a stopping set as in Remark \ref{rA1}. 

 We obtain from the above arguments and Remark \ref{rA2} that $ \tilde{Z}_n^s = Z_n^s \times \cK_r$ is  a stopping set in $\R^d \times \cK_r$ as required. Further, it is easy to see that $ \tilde{Z}_n^s = Z_n^s \times \cK_r$ determines $f_n$ as well. Now, we will derive bounds on the revealment probability of $Z_n^s$. For $x \in R^d$ with $B_r(x) \cap L = \emptyset$, using \eqref{e:graphmble-nsbm}, we have that
$$ \BP(x \in Z_n^s) \leq \BP(B_{r}(x) \cap S \neq \emptyset) \leq \BP(B_{r}(x) \stackrel{R_n}{\leftrightarrow} L) \leq \BP(Arm_{r,|s-x_1|}(\gamma_c)).$$
Thus, we have shown \eqref{e:revprobkpercns} and the proof is complete. 
\end{proof}
We now conclude the section with the proof of Theorem \ref{t:bmexpdecay}. 
\begin{proof}[Proof of Theorem \ref{t:bmexpdecay}]
Let $\BQ_s := \delta_{B_s(0)}$ denote the probability distribution on grains supported on a deterministic ball of radius $s$. By the scaling property of the Poisson process, $k$-percolation in the Poisson Boolean model with intensity $\gamma$ and grain distribution $\BQ_s$ is equivalent to $k$-percolation in the Poisson Boolean model with intensity $1$ and grain distribution $\BQ_{\gamma^{1/d}s}$. Now, using \cite[Corollary 1.3]{BY2013} and the above scaling property, we have that $\gamma_c(\BQ_s) \in (0,\infty)$ for any $s \in (0,\infty)$. Now since $\BQ$ satisfies the assumption \eqref{e:assumq}, we have that $\gamma_c(\BQ) \geq \gamma_c(\BQ_r) > 0$. Further, let $\beta = \BQ\{ K : B_{r_0}(0) \subset K \}$ and by assumption \eqref{e:assumq}, $\beta > 0$. Let $\cO'$ be the $k$-occupied region defined on a Poisson point process with intensity $\beta \gamma$ and grain distribution $\BQ_{r_0}$. By the Poisson thinning property \cite[Theorem 5.8]{LastPenrose17}, we have that $\cO' \subset \cO(\gamma)$ and hence $\gamma_c(\BQ) \leq \beta \gamma_c(\BQ_{r_0}) < \infty$.  Thus $\gamma_c(\BQ) \in [\gamma_c(\BQ_r), \beta \gamma_c(\BQ_{r_0})] \subset (0,\infty)$ and we have shown the non-degeneracy of $\gamma_c(\BQ)$. 

  Further, $\cO_k(\teta) \subset \cO_1(\teta)$ for all $k \geq 1$ and hence by the exponential decay result for $\cO_1(\teta)$ in \cite[Theorem 3.1]{Ziesche2018}, we have that there exists $a > 0$ such that $\theta_s(\gamma)$ decays exponentially for $\gamma < a$. Trivially $a < \gamma_c$ and we shall fix such an $a$ in the rest of the proof. 

  Assume $n \geq 10r$ and choose $b \in (\gamma_c,\infty)$. Similar to the the proof of Theorem \ref{t:nscritbm}, we will define a randomized stopping set $Z_n$ in $\R^d$ that determines $\1\{0 \leftrightarrow \partial B_n(0) \}$ such that there exists a constant $C_a$ for all $\gamma \in [a,b]$ satisfying the following bound on revealment probability :
\begin{equation}
\label{e:revprobctdt}
 \delta_n = \max_{x \in \R^d} \BP(x \in Z_n) \leq C_an^{-1} \int_0^n \theta_s(\gamma) \md s.
\end{equation}
Further, we will show that $Z_n$ can be constructed via randomized CTDT as in Theorem \ref{t:POSSSmarked}. Then using Theorem \ref{t:POSSSmarked}, monotonicity of $\1\{0 \leftrightarrow \partial B_n(0)]\}$ and the Russo-Margulis formula for Poisson functionals (see \cite[Theorem 19.4]{LastPenrose17}), we derive that
\begin{align*}
\theta_n(\gamma)(1- \theta_n(\gamma)) & \leq \gamma \int_{\R^d \times \cK_r} \BP(x \in Z_n)\BE[D_{(x,K)}\1\{0 \leftrightarrow \partial B_n(0) \}] \, \BQ(dK) \, \md x \\ 
& \leq \gamma \delta_n \int_{B_{n+2r}(0) \times \cK_r} \BE[D_{(x,K)}\1\{0 \leftrightarrow \partial B_n(0)\}] \, \BQ(dK) \, \md x \\ 
& = \gamma \delta_n \frac{\md \theta_n(\gamma)}{\md \gamma} = \gamma \delta_n \theta'_n(\gamma). 
\end{align*}
Thus, using the bound for $\delta_n$, we obtain for $\gamma \in [a,b]$ the differential inequality ,
$$ \theta'_n(\gamma) \geq \frac{n}{C_a b \int_0^n \theta_s(\gamma) \md s}  \theta_n(\gamma)(1- \theta_n(\gamma)) \geq \frac{n}{C_ab \int_0^n \theta_s(\gamma) \md s}  \theta_n(\gamma)(1- \theta_1(b)),$$
where we have used that $\theta_s(\gamma)$ is increasing in $\gamma$ and decreasing in $s$. Now using a straightforward variant of  \cite[Lemma 3.1]{DCRT19} (see also \cite[Proof of Theorem 1.2]{DCRT18}), we obtain that there exists $\gamma_1 \in [a,b]$ such that
\begin{itemize}
\item For $\gamma \in [a,\gamma_1)$, $\theta_n(\gamma) \leq \exp\{-\alpha(\gamma)n\}$ for all $n \geq 1$ and some constants $\alpha(\gamma) \in (0,\infty)$. 
\item For $\gamma \in [\gamma_1,b]$, there exists a constant $\alpha(b) \in (0,\infty)$ such that $\theta(\gamma) \geq \alpha(b)(\gamma - \gamma_1)$. 
\end{itemize}
Since $a<\gamma_c < b$ by construction and since the previous properties imply that $\theta(\gamma)=0$ for $\gamma \in (a,\gamma_1)$ and $\theta(\gamma)>0$ for $\gamma \in (\gamma_1, b)$ we conclude that, necessarily, $\gamma_1 = \gamma_c$. Thus, both the claims in the theorem follow. 

  All that remains to complete the proof is to show that $Z_n$ can be constructed via a randomized CTDT satisfying suitable assumptions. First, we will describe the randomized stopping set $Z_n$, then show that it satisfies the required revealment probability and finally construct it as a randomized CTDT satisfying the necessary assumptions. 

 Fix $s \in (0,n)$. We now define $Z^s_n$. Let $S$ be the union of connected components of $\cO(\gamma) \cap B_n(0)$ that intersect $\partial B_s(0)$. Note that $S = \emptyset$ if there is no connected component of $\cO(\gamma) \cap B_n(0)$ that intersects $\partial B_s(0)$. Observe that the event $\{0 \leftrightarrow \partial B_n(0)\}$ is determined by $\teta \cap (B_{n+r}(0) \times \cK_r)$ and hence we restrict to this set. Define $Z^s_n =( S \cup \partial B_s(0)) \oplus B_{2r}(0)$ and $A_r(x) := \{x \leftrightarrow \partial B_{2r}(x) \} \cap \{ \partial  B_{2r}(x) \subset \cO(\gamma) \}$ similar to $A^{\infty}_r(x)$ defined in Remark \ref{r:exns}(i). Further, reasoning as in Remark \ref{r:exns}(i), we multiply and divide by $\BP(A_r(x))$ in the second term, then use the Harris-FKG inequality and stationarity of the Poisson point process to derive that
\begin{align}
\BP(x \in Z^s_n) & \leq  \1\{||x|-s| \leq 2r\} + \1\{||x|-s| > 2r\} \BP(B_{2r}(x) \leftrightarrow \partial B_s(0))\nonumber  \\
& \leq  \1\{||x|-s| \leq 2r\} + \1\{||x|-s| > 2r\} \BP(A_r(0))^{-1}\BP(x \leftrightarrow \partial B_s(0)) \nonumber \\
  \label{e:revprobkperc} & \leq  \1\{||x|-s| \leq 2r\} + \1\{||x|-s| > 2r\} \BP(A_r(0))^{-1}\BP(x \leftrightarrow \partial B_{|s-|x||}(x)).
\end{align}
Denote $C'_{\gamma} = \BP(A_r(0))$ and since $C'_{\gamma}$ is increasing in $\gamma$, $\inf_{\gamma \geq a} C'_{\gamma} = C'_a > 0$. Also, we have that $\inf_{\gamma \geq a} \inf_{s \leq 2r}\theta_s(\gamma) \geq \theta_{2r}(a) > 0$ and so from \eqref{e:revprobkperc}, we obtain that for some constant $C_a$ 
\begin{align*}
 \BP(x \in Z_n^s)&  \leq \1\{||x|-s| \leq 2r\} \theta_{2r}(a)^{-1} \theta_{|s-|x||}(\gamma) + \1\{||x|-s| > 2r\} (C'_a)^{-1} \theta_{|s-|x||}(\gamma) \\
 & \leq C_a \theta_{|s-|x||}(\gamma).
\end{align*}
Thus, choosing $Y$ uniformly at random in $(0,n)$ and setting $Z_n = Z_n^Y$, we obtain that
$$ \BP(x \in Z_n) \leq C_a n^{-1} \int_0^n \theta_{|s- |x||}(\gamma) \md s \leq 2C_an^{-1}\int_0^n \theta_{s}(\gamma) \md s.$$
This shows that our randomized stopping set has the revealment probability as required in \eqref{e:revprobctdt}. Now, we only need to show that it can be constructed as a randomized CTDT satisfying the assumptions in Theorem \ref{t:POSSSmarked}. 

 We will first describe a sequence of stopping sets in $\R^d$ and then describe how to build a CTDT on $\R^d$ from the same. Let $s \in (0,n)$. 
\begin{itemize}
\item Set $S_0 = \emptyset, E_0 = \partial B_s(0)$. 

\item Set $E_1 = (S_0 \cup E_0) \oplus B_{2r}(0)$ and $S_1$ is the union of components of $\cO_k(\teta \cap (E_1  \times \cK_r)) \cap B_{n+r}(0)$ intersecting $E_0$. 

\item Given $(S_i,E_i)$ for all $1 \leq i \leq m \in \N$, we construct $(S_{m+1},E_{m+1})$ as follows :  
\begin{itemize}
\item If $S_m = S_{m-1}$, then the algorithm terminates. Also, $\1\{0 \leftrightarrow \partial B_n(0)]\} = 1$  if there is a path in $S_m$ from $0$ to $\partial B_n(0)$ else $\1\{0 \leftrightarrow \partial B_n(0)]\} = 0$. Further, set $E_{\infty} = E_m$.

\item  If $S_m \neq S_{m-1}$. Set $E_{m+1} = (S_m \cup E_0) \oplus B_{2r}(0)$ and $S_{m+1}$ be the  union of components of $\cO_k(\teta \cap (E_{m+1}  \times \cK_r)) \cap B_{n+r}(0)$ intersecting $S_m$.
\end{itemize}
\end{itemize}
Observe that if the algorithm terminates after $m$ steps, then $S_m = S$ where $S$ is the union of connected components of $\cO(\gamma) \cap B_{n+r}(0)$ that intersect $\partial B_s(0)$ if there exists any and else $S = \emptyset$. Thus, by definition of $E_{\infty}$, we also obtain that $E_{\infty} = (S \cup \partial B_s(0)) \oplus B_{2r}(0)$. 

 Now we construct the CTDT. For $t \in (m,m+1)$ such that $S_m \neq S_{m-1}$, set $Z^s_{n,t} = (S_m \cup \partial B_s(0))  \oplus B_{2r(t-m)	}(0)$. Thus, we have that $Z^s_n = Z^s_{n,\infty} = E_{\infty}$ as required. The graph measurability and stopping set property (see Remark \ref{rA1}) of $Z^s_{n,t}, t \geq 0$ can be argued as in the proof of Theorem \ref{t:nscritbm}. We will now verify the other properties of CTDT. 

 Observe that for all $t \in (m,m+1]$ and any $\epsilon \in (0,t-m)$, 
$$Z^s_{n,t}\setminus Z^s_{n,t-\epsilon} \subset \{ x \in \R^d : d(x,S_m \cup \partial B_s(0)) \in (2r(t-\epsilon),2r(t-m)] \}.$$
Also note that $Z^s_{n,0} = \partial B_s(0)$ has zero measure and
further using the above observation, other properties of CTDT (namely
\eqref{etr1}, \eqref{etr2}) can be verified for $Z^s_{n,t}$. Next,
\eqref{maea1} holds trivially by the above observation and since the
intensity measure of $\teta$ is $\gamma \, \md x \, \BQ(\md K)$, a diffuse measure. 
Finally, because the intensity measure of $\teta$ is diffuse and the observation on $Z^s_{n,t} \ Z^s_{n,t-\epsilon}$, we have that $\{|X_i|\}_{i \geq 1}$ is a
simple point process and this verifies \eqref{maea2}.  Since our function $f_n(\teta)$ is a function of $f_n(\teta \cap (B_{n+2r}(0) \times \cK_R))$, assumption \eqref{eCTDTf2} holds because of Remark \ref{r3.3}. 

 This completes the proof as we have constructed a CTDT satisfying the
assumptions of Theorem \ref{t:POSSSmarked} and having revealment
probability as given in \eqref{e:revprobctdt}.
\end{proof}
\section{Confetti Percolation}
\label{s:confettiperc}

The confetti percolation model has its origins in the dead leaves model introduced by Matheron \cite{Matheron68}. Various geometric properties of the dead leaves model have been studied. For example, see \cite{Serra82,Jeulin97,Bordenave06}. The percolation-theoretic version of the  model called {\em confetti percolation} was introduced by Benjamini and Schramm in \cite{Benjamini98}. Since then, this model has been investigated in many works \cite{Hirsch15,Muller17,Ahlbergsharp18,Ghosh18}. Although the dead leaves model is defined with general grains and in arbitrary dimensions, studies of confetti percolation model are focussed on random balls or squares (i.e., $\ell_{\infty}$ balls) in two dimensions. We shall consider the former framework of general grains in arbitrary dimensions and prove a  sharp phase transition therein. 

Consider a Poisson point process $\teta$ on $\tilde{\BX} = \R^d \times \R_+  \times \cK_r \times \cK_r \times \{0,1\}$ with intensity measure $\tilde{\gamma}(\md (x,t,K_1,K_2,a)) =  \md x \times \md t   \times \BQ_1(\md K_1) \times \BQ_2(\md K_2) \times \nu_p(a)$ where $x \in \R^d, t \in \R_+, K_1, K_2 \in \cK_r, a \in \{0,1\}$ and $\nu_p = p\delta_{0} + (1-p)\delta_1$ is the standard Bernoulli($p$) distribution on $\{0,1\}$. The interpretation is that $x$ denotes the location of the particle, $t$ the arrival time, $a$ will determine the colour of the particle ({\em black} if $a = 0$ and {\em white} otherwise) and $K_1$ is the grain attached to black particles and $K_2$ is the grain attached to white particles. The grains fall on $x$ at time $t$ with colour determined by $a$ and accordingly the grain. Each point in the plane is coloured according to the first grain that covers it and the question of interest is percolation of the black region. We will now define this more formally using the framework of \cite[Section 2]{Hirsch15}. We will assume throughout the section that $\BQ_1,\BQ_2$ satisfy assumptions in \eqref{e:assumq}.

We define {\em the black region} as 
\begin{equation}
\label{d:occconfetti}
\cO_p := \cO_p(\teta) =  \cO_p(\BQ_1,\BQ_2) =  \bigcup_{(x,t,K_1,K_2,0) \in \teta} (x + K_1) \setminus \{ \bigcup_{(x',t',K'_1,K'_2,1) \in \teta, t' < t} (x' + K'_2)\}.
\end{equation}
The {\em white region} $\cV_p$ can be defined analogously or because $\teta$ is a space-time Poisson point process, we can observe that every point has to be coloured black or white because $\BQ_1,\BQ_2$ satisfy assumptions in \eqref{e:assumq}. Hence $\cV_p = \R^d \setminus \cO_p$ is the white region. 

We now define the percolation events and the corresponding probabilities as in \eqref{e:defpercprob1}. For connected subsets $A,B,R \subset \R^d$ such that $A \cup B \subset R$, define
\begin{align}
\{A \overset{R}{\longleftrightarrow} B\} &:= \{ \mbox{there exists a path in $\cO_p \cap R$ from $x$ to $y$ for some $x \in A, y \in B$} \}, \nonumber \\
\{A \longleftrightarrow B \} &:= \{ A \overset{\R^d}{\longleftrightarrow} B \}, \nonumber \\
\{ 0 \longleftrightarrow \infty \} &:=  \{ \mbox{there is an unbounded path from $0$ in $\cO_p$} \} \nonumber \\
\theta_s(p) &:=  \theta_n(p,\BQ_1,\BQ_2) = \BP(0 \longleftrightarrow \partial B_s(0)), s > 0 \nonumber  \\
\theta(p) &:=  \theta(p,\BQ_1,\BQ_2) =  \BP(0 \longleftrightarrow \infty) \nonumber \\ 
Arm_{s,t}(p) &:= \{ B^{\infty}_s(0) \longleftrightarrow \partial B^{\infty}_t(0) \}, \, \, 0 < s < t, \nonumber \\
Cross_{s,t}(p) &:=  \{ \{0\} \times [0,t]^{d-1} \overset{[0,s] \times [0,t]^{d-1}}{\longleftrightarrow} \{s\} \times [0,t]^{d-1} \}, \, \, \kappa, t > 0.
 \label{e:defpercprobconf}
\end{align}
Again, by monotonicity (in $n$) of $\theta_n(p)$, we have that $\theta(p) = \lim_{n \to \infty}\theta_n(p)$.  We suppress the dependence on $\BQ_1,\BQ_2$ as they are often fixed. Analogously to $\cO_p$, we can define the above events with respect to $\cV_p$ and in this case, we will denote the events and probabilities with a $*$ superscript i.e., $\longleftrightarrow^*, \theta^*_s(p), \theta^*(p), Arm^*_{s,t}(p),\theta^*(p), Cross^*_{s,t}(p)$ and so on.

Throughout this section, we have chosen the intensity to be $1$ but instead we could have also chosen the intensity to be $\gamma$ for some $\gamma \in (0,\infty)$ i.e., $\tilde{\gamma}(\md (x,t,K_1,K_2,a)) = \gamma  \md x \times \md t   \times \BQ_1(\md K_1) \times \BQ_2(\md K_2) \times \nu_p(a)$. Because of the scale-invariance of the Poisson point process in the time direction, the probabilities defined above do not depend on the value of $\gamma$ and hence we have chosen $\gamma = 1$ for convenience.

A measurable mapping $f : \bN(\tilde{\BX}) \to \R$ is said to be {\em black-increasing} if for all $\mu \in \bN$ and $x' = (x,t,K_1,K_2) \in \R^d \times \R_+ \times \cK_r \times \cK_r$, we have that
\begin{equation}
\label{e:blackinc}
 f_n(\mu +\delta_{(x',1)}) \leq f_n(\mu) \leq f_n(\mu + \delta_{(x',0)}).
\end{equation}
As usual an event is said to be {\em black-increasing} if its indicator function is. Trivially, if $f$ is a black-increasing function, $\BE[f(\teta)]$ is an increasing function in $p$. Since $\{0 \longleftrightarrow \infty \}$ is a black-increasing event, $\theta(p)$ is increasing in $p$ and $\theta^*(p)$ is decreasing in $p$. Hence, we can define the critical probabilities as 
\begin{align}
p_c & := p_c(\BQ_1,\BQ_2) = \inf \{ p : \theta(p,\BQ_1,\BQ_2) > 0 \}, \nonumber \\
\label{e:critprobconf} p_c^* & :=  p^*_c(\BQ_1,\BQ_2) = \sup \{p : \theta^*(p,\BQ_1,\BQ_2) > 0 \}. 
\end{align}
Out first theorem shows a sharp phase transition for $\theta_n(p)$.
\begin{theorem}[Sharp threshold for confetti percolation]
\label{t:confettisharp}
Consider the black occupied region $\cO_p$ in the confetti percolation model as defined in \eqref{d:occconfetti} with grain distributions $\BQ_1,\BQ_2$ satisfying assumptions in \eqref{e:assumq}.  Then, for $p_c$ as defined above, the following statements hold.
\begin{enumerate}
\item For all $p < p_c$, there exists a constant $\alpha(p) \in (0,\infty)$ such that $\theta_n(p) \leq \exp\{-\alpha(p)n\}$ for all $n \geq 1$. 
\item For all $p' < 1$, there exists a constant $\alpha \in (0,\infty)$ such that for all $p \in [p_c,p')$, we have that $\theta_n(p) \geq \alpha(p - p_c)$ for all $n \geq 1$.
\end{enumerate}
\end{theorem}

\begin{remark}
{\rm
\begin{enumerate}
\item Note that we do not say anything about non-triviality of $p_c$ i.e., $0 < p_c < 1$. Though we would expect this to hold under the assumptions of Theorem \ref{t:confettisharp}, this is beyond the scope of our project. In the planar case (i.e., $d = 2$) non-triviality is known in the the case when $\BQ_1,\BQ_2$ are supported on balls (see \cite[Theorem 8.3]{Ahlbergsharp18}). For the general case, one may use a suitable coupling argument similar to that in Theorem \ref{t:bmexpdecay}.

\item  The above theorem in the case of $d = 2$ and when $\BQ_1,\BQ_2$ are supported on boxes was shown in \cite[Proposition 1.1]{Ghosh18}. The proof therein uses the discrete OSSS inequality whereas we use the continuum version and thereby enabling us to provide a simpler proof that holds in greater generality. It was remarked in \cite[Remark 1.1]{Ghosh18} that their method can be adapted for other shapes but our proof allows one to treat different shapes at once without making any shape-specific argument. The same remark applies to the Theorem \ref{t:critprobconfplanar} as well. 
\end{enumerate}
}
\end{remark} 
\begin{proof}
We shall prove the sharp threshold result using the same proof strategy as in Theorem \ref{t:bmexpdecay} but with some additional technicalities due to the non-compactness of $\tilde{\BX}$. 

 We fix a $n$ large and $\epsilon \in (0,1/2)$. Suppose we show that for $p \in [\epsilon,1-\epsilon]$,
\begin{equation}
\label{e:diffeqnconf}
 \frac{\md \theta_n(p)}{\md p} \geq \frac{b_r(\epsilon) n}{2 \int_0^n\theta_{s}(p)\md s} \theta_n(p)(1- \theta_n(p)),
\end{equation}
where $b_r(\epsilon) > 0$ will be defined explicitly soon. Now observe that 
$$1 - \theta_n(p) \geq 1 - \theta_n(1-\epsilon) \geq \BP(B_1(0) \subset \cV_{1-\epsilon}) > 0,$$ 
i.e., there cannot be a path from $0$ to $\partial B_n(0)$ if the unit ball around $0$ is covered by the white region. Thus, we obtain from \eqref{e:diffeqnconf} and the above inequality that
$$\frac{\md \theta_n(p)}{\md p} \geq \frac{b_r(\epsilon) n}{2 \int_0^n\theta_{s}(p)\md s} \theta_n(p)(1- \theta_n(1-\epsilon)).$$
Now using \cite[Lemma 1]{DCRT19b}, we obtain that there exists $p_0 := p_0(\epsilon) \in [\epsilon,1-\epsilon]$ such that
\begin{enumerate}
\item For all $p \in [\epsilon,p_0)$, there exists a constant $\alpha(p) \in (0,\infty)$ such that $\theta_n(p) \leq \exp\{-\alpha(p) n\}$ for all $n \geq 1$. 
\item There exists a constant $\alpha(\epsilon) \in (0,\infty)$ such that for all $p \in (p_0,1-\epsilon]$, $\theta_n(p) \geq \alpha(\epsilon)(p - p_0)$ for all $n \geq 1$.
\end{enumerate}
By definition of $p_c$, we have that $\theta(p) = 0$ for $p < p_c$ and $\theta(p) > 0$ for $p > p_c$. From this and the above two statements, the proof of the two statements in the theorem is complete. 

 Now, we are left to prove \eqref{e:diffeqnconf}. Fix $s \in (0,n)$. For $h > 0$, let $\teta_h = \teta \cap \tilde{\BX}_h$ where $\tilde{\BX}_h = \R^d \times [0,h] \times \cK_r \times \cK_r \times \{0,1\}$. When we want to refer to events in \eqref{e:defpercprobconf} but with respect to $\cO_p(\teta_h)$, we shall use $\longleftrightarrow^h$ for existence of corresponding paths in  $\cO(\teta_h)$ and $\theta_s^h, \theta_h$ for the corresponding percolation probabilities. We define 
$$f^h_n := \1\{0 \longleftrightarrow^h \partial B_n(0)\}, f_n :=  \1\{0 \longleftrightarrow \partial B_n(0)\}.$$
Observe that $\theta^h_n(p) = \BE[f_n^h], \theta_n(p) = \BE[f_n].$ Our proof strategy now is to first use Theorem \ref{t:POSSSmarked} for $f_n^h$ and derive a version of the differential inequality \eqref{e:diffeqnconf} for $\theta_n^h$. Then we will complete the proof by showing that all the terms in the differential inequality \eqref{e:diffeqnconf} are well-approximated by the corresponding truncated versions for $f_n^h$. Since $f_n^h$ is dependent on the point process in a compact set, we can use ideas similar to those in the proof of Theorem \ref{t:bmexpdecay} to derive the differential inequality \eqref{e:diffeqnconf} for $\theta_n^h$.

 Now, we construct a CTDT satisfying assumptions of Theorem \ref{t:POSSSmarked} for $f_n^h$. Let $S$ be the union of connected components of $\cO_p(\teta_h)$ that intersect $\partial B_s(0)$. Note
that $S = \emptyset$ if no such component exists. Define
$Z_n^s := (S \cup \partial B_s(0)) \oplus B_{2r}(0)$. We will skip
the construction of $Z_n^s$ via a CTDT as it is similar to that in
Theorem \ref{t:bmexpdecay}.  We note that $Z_n^s$ determines $f_n$ and
further it is a stopping set as in Remark \ref{rA1}.
Now again as in the derivation of
\eqref{e:revprobkperc}, using Harris-FKG inequality (see
\cite[Definition 2.1 and below]{Ghosh18}) and translation invariance
of the Poisson point process, we obtain that for $x \in \R^d$,
\begin{align}
\BP(x \in Z_n^s) & \leq \BP(B_{2r}(x) \cap (S \cup \partial B_s(0)) \neq \emptyset) \nonumber \\ 
& = \1\{||x| - s| \leq 2r\} + \1\{||x| -s| > 2r\}\BP(B_{2r}(x) \longleftrightarrow^h \partial B_s(0)) \nonumber \\
& \leq  \1\{||x| - s| \leq 2r\} + (b^h_r(p))^{-1}\1\{||x| -s| > 2r\}\BP(x \longleftrightarrow^h \partial B_s(0)) \nonumber \\
\label{e:revprobbd_conf} & \leq  \1\{||x| - s| \leq 2r\} + (b^h_r(p))^{-1}\1\{||x| -s| > 2r\}\theta^h_{|s-|x||}(p),
\end{align}
where $b^h_r(p) = \BP(\{0 \longleftrightarrow^h \partial B_{2r}(0)\} \cap \{\partial B_{2r}(0) \subset \cO_p(\teta_h)\})$ and the usage of Harris-FKG inequality is justified as the events are black-increasing as defined in \eqref{e:blackinc}. Further $b^h_r(p) \geq b^h_r(\epsilon) \geq \BP(B_{2r}(0) \subset \cO_{\epsilon}(\teta_h))$ and the positivity of the latter guarantees the positivity of $b^h_r(\epsilon)$. Now we randomize over $s$ as before i.e., set $Z_n := Z_n^Y$ where $Y$ is a uniform $(0,n)$-valued random variable. Then we obtain that
\begin{align*}
 \delta_n &:= \sup_{x \in B_{n+2r}(0)} \BP(x \in Z_n)  \leq (b^h_r(\epsilon))^{-1}n^{-1}\sup_{x \in B_{n+2r}(0)} \int_0^n\theta^h_{|s-|x||}(p)\md s \\
 & \leq 2(b^h_r(\epsilon))^{-1}n^{-1}\int_0^n\theta^h_{s}(p)\md s,
\end{align*}
where in the first inequality follows from \eqref{e:revprobbd_conf} and from arguments analogous to those rehearsed in the proof of Theorem \ref{t:bmexpdecay}. 

 Having constructed a suitable CTDT, we now derive the differential inequality for $\theta_n^h$. By $x'$ we denote $(x,t,K_1,K_2)$. Using that $f^h_n$ is black-increasing as in \eqref{e:blackinc} and $f^h_n \in \{0,1\}$, we have that
$$|D_{(x',0)}f^h_n| = f^h_n(\teta + \delta_{(x',0)}) -  f^h_n(\teta) \leq f^h_n(\teta + \delta_{(x',0)}) -  f^h_n(\teta +\delta_{(x',1)}) \geq |D_{(x',1)}f^h_n|.$$
Let $\mathbb{X}$ be such that $\tilde{\BX} = \BX \times \{0,1\}$. Thus from Theorem \ref{t:POSSSmarked}, the definition of $\delta_n$, the above inequalities and a version of the Russo-Margulis formula in \cite[Exercise 19.8]{LastPenrose17}, we derive that
\begin{align}
\theta^h_n(p)(1- \theta^h_n(p)) & \leq \delta_n \left( p \int_{\BX} \BE[|D_{(x',0)}f^h_n(\teta)|] \BQ_1(\md K_1)\BQ_2(\md K_2) \md t \md x \right.  \nonumber \\
& \quad + \left. (1-p) \int_{\BX} \BE[|D_{(x',1)}f^h_n(\teta)|] \BQ_1(\md K_1)\BQ_2(\md K_2) \md t \md x \right) \nonumber  \\
& \leq \delta_n \int_{\BX} \BE[f^h_n(\teta + \delta_{(x',0)}) -  f^h_n(\teta +\delta_{(x',1)})] \BQ_1(\md K_1)\BQ_2(\md K_2) \md t \md x \nonumber \\
\label{e:diffeqh} & = \delta_n \frac{\md \theta^h_n(p)}{\md p} 
\leq \left( 2(b^h_r(\epsilon))^{-1}n^{-1}\int_0^n\theta^h_{s}(p)\md s \right) \frac{\md \theta^h_n(p)}{\md p}.
\end{align}
Now, we will complete the proof by showing that the terms in \eqref{e:diffeqh} approximate those of \eqref{e:diffeqnconf}. Define 
$$b_r(p) := \BP(\{0 \longleftrightarrow \partial B_{2r}(0)\} \cap \{\partial B_{2r}(0) \subset \cO_p\}).$$
By the same reasoning as in the forthcoming arguments in \eqref{e:ghpapprox}, we have that for every $s > 0$ and $p \in [\epsilon,1-\epsilon]$, 
$$\theta^h_s(p) \uparrow \theta_s(p), b^h_r(p) \uparrow b_r(p).$$
Thus, to show that the differential inequality \eqref{e:diffeqh} converges to the differential inequality \eqref{e:diffeqnconf}, it remains to show that $\theta_s(p)$ is differentiable in $p$ and that the derivatives $ \frac{\md \theta^h_n(p)}{\md p}$ converge as $h \to \infty$. To prove this, we will now show that derivatives $ \frac{\md \theta^h_n(p)}{\md p}$ converge uniformly as $h \to \infty$ and thus proving both differentiability and convergence of the derivatives. 

Set $\BX_h := \R^d \times [0,h] \times \cK_r \times \cK_r, \BX' := \R^d \times (h,\infty) \times \cK_r \times \cK_r $. Define,
\begin{align*}
 g_n^h(p) &:=  \int_{\BX_h} \BE[f^h_n(\teta + \delta_{(x',0)}) -  f^h_n(\teta +\delta_{(x',1)})] \BQ_1(\md K_1)\BQ_2(\md K_2) \md t \md x \\
  g_n(p) &:=  \int_{\BX} \BE[f_n(\teta + \delta_{(x',0)}) -  f_n(\teta +\delta_{(x',1)})] \BQ_1(\md K_1)\BQ_2(\md K_2) \md t \md x,
 \end{align*}
By the Russo-Margulis formula \cite[Exercise 19.8]{LastPenrose17}, $ \frac{\md \theta^h_n(p)}{\md p} = g^h_n(p)$. Hence to show uniform convergence of the derivatives, it is enough to show uniform convergence of $g^h_n(p)$ to $g_n(p)$ for $p \in [\epsilon,1-\epsilon]$. Now, using the fact that $f_n^h$ depends only on $\teta^h$, we derive that
\begin{align*}
|g_n(p) - g_n^h(p)| & \leq \int_{\BX'} \BE[f_n(\teta + \delta_{(x',0)}) -  f_n(\teta +\delta_{(x',1)})] \BQ_1(\md K_1)\BQ_2(\md K_2) \md t \md x \\
& \quad +  \int_{\BX_h} \BE[|f_n(\teta + \delta_{(x',0)}) -  f^h_n(\teta +\delta_{(x',0)})|]  \BQ_1(\md K_1)\BQ_2(\md K_2) \md t \md x \\
& \quad + \int_{\BX_h} \BE[|f^h_n(\teta + \delta_{(x',1)}) -  f_n(\teta +\delta_{(x',1)})|]  \BQ_1(\md K_1)\BQ_2(\md K_2) \md t \md x.
\end{align*}
We sub-divide $B_{n+2r}(0)$ into small cubes $R_{n1},\ldots,R_{nk_a}$ of side-length $a$ and $a$ small enough such that the diameter of each cube is at most $r_0/4$ for $r_0$ as in assumption \eqref{e:assumq}. Set $\cK'_{r_0} := \{K \in \cK_r : B_{r_0}(0) \subset K \}$ and $\beta_i = \BQ_i(\cK'_{r_0}), i =1,2$. Note that $\beta_i > 0$ for $i=1,2$ because of assumption \eqref{e:assumq}. 

 Our key observation is that if for some $t$ and all $i$, $\teta_h(R_{ni} \times [0,t) \times \cK'_{r_0} \times \cK'_{r_0} \times \{0,1\}) \neq 0$, then $\cO_p(\teta_h + \delta_{(x,t',K_1,K_2,0)}) \cap B_n(0) = \cO_p(\teta + \delta_{(x,t',K_1,K_2,1)}) \cap B_n(0)$ for $t' > t, x \in B_{n+2r}(0)$ and for all $p$.  Now from this observation, union bound over $R_{ni}$'s and the fact that $\teta_h(R_{ni} \times [0,t) \times \cK'_{r_0} \times \cK'_{r_0} \times \{0,1\})$ is a Poisson random variable with mean $a^d\beta_1 \beta_2 t$, we derive that for $x' = (x,t,K_1,K_2)$,
$$ \BE[f_n(\teta + \delta_{(x',0)}) -  f_n(\teta +\delta_{(x',1)})] \leq k_a \exp\{-a^d\beta_1 \beta_2 t\}.$$
Similarly, if for some $h$ and all $i$, $\teta_h(R_{ni} \times [0,h] \times \cK'_{r_0} \times \cK'_{r_0} \times \{0,1\}) \neq 0$, then $\cO_p(\teta_h)  \cap B_n(0)= \cO_p(\teta) \cap B_n(0)$ and $\cO_p(\teta_h + \delta_{(x,t,K_1,K_2,a)}) \cap B_n(0) = \cO_p(\teta + \delta_{(x,t,K_1,K_2,a)}) \cap B_n(0)$ for $t \leq h, a \in \{0,1\}, x \in B_{n+2r}(0)$ and for all $p$. Arguing as above, the latter two integrands are bounded by $k_a \exp\{-a^d\beta_1 \beta_2h\}$. Thus substituting these bounds in the above inequality and observing that by the definition of $f_n,f_n^h$, one can restrict to $x \in B_{n+2r}(0)$, we obtain that 
\begin{equation}
\label{e:ghpapprox}
|g_n(p) - g_n^h(p)| \leq  k_a|B_{n+2r}(0)| \left ( \int_h^{\infty} \exp\{-a^d\beta_1 \beta_2 t\} \md t + 2 \exp\{-a^d\beta_1 \beta_2h\} \right) .
\end{equation}
From the above bounds, we obtain uniform convergence of $g^h_n(p)$ to $g_n(p)$ for $p \in [\epsilon,1-\epsilon]$ and as argued below \eqref{e:diffeqh}, this completes the proof of the \eqref{e:diffeqnconf} and hence that of the theorem as well. 
\end{proof}
A powerful consequence of the above sharp phase transition result is the exact determination of the critical probabilities in the planar case when $\BQ_1 = \BQ_2 = \BQ$. By considerations of self-duality, it was conjectured that $p_c = p^*_c = \frac{1}{2}$ in the case of fixed balls by Benjamini and Schramm \cite[Problem 5]{Benjamini98}. It was proven in the case of fixed boxes by Hirsch \cite{Hirsch15} and then in the case of fixed balls by M\"{u}ller \cite{Muller17}. Recently, it was proven in the case of random boxes by Ghosh and Roy \cite{Ghosh18}.  Our next result will include all of the above as special cases. 
\begin{theorem}[Critical probability for planar confetti percolation]
\label{t:critprobconfplanar}
Consider the black occupied region $\cO_p$ in the planar (i.e., $d = 2$) confetti percolation model as defined in \eqref{d:occconfetti} with grain distribution $\BQ_1 = \BQ_2 = \BQ$ satisfying assumptions in \eqref{e:assumq}. Further, we assume that $\BQ$ is invariant under $\pi/2$-rotations and reflections at the coordinate axes. Then we have that $p_c(\BQ,\BQ) = p_c^*(\BQ,\BQ) = \frac{1}{2}$. 
\end{theorem}
Further, under a certain transitivity condition, it was shown in \cite[Theorem 1.2]{Ghosh18} for the planar case that $p_c(\BQ_1,\BQ_2) = \frac{|B_1|}{|B_1| + |B_2|}$ when $\BQ_1 = \delta_{B_1}$ and $\BQ_2 = \delta_{B_2}$ for $B_1, B_2 \in \cK_r$ and $|.|$ denotes the Lebesgue volume. It is not yet known if the transitivity condition holds always. 
\begin{proof}
Using the observation $\cV_p(\BQ_1,\BQ_2) \overset{d}{=} \cO_{1-p}(\BQ_2,\BQ_1)$, we obtain that $p^*_c(\BQ_1,\BQ_2) = 1 - p_c(\BQ_2,\BQ_1)$ and thus $p^*_c(\BQ,\BQ) + p_c(\BQ,\BQ) = 1$ when $\BQ_1 = \BQ_2 = \BQ$. Hence, to prove the theorem it suffices to show that $p_c(\BQ,\BQ) = p_c^*(\BQ,\BQ)$. 

 Now from the sharp phase transition result (Theorem \ref{t:confettisharp}), we can obtain that
\begin{equation}
\label{e:crosstrans}
\BP(Cross_{2n,2n}(p)) \overset{n \to \infty}{\to} 0  \, \,  \mbox{for}  \, \,  p < p_c \, \,  \mbox{and} \, \,   \BP(Cross_{2n,2n}(p))  \overset{n \to \infty}{\to} 1  \, \,  \mbox{for}  \, \,  p > p_c, \, \, 
\end{equation}
and similarly for $\BP(Cross^*_{2n,2n}(p))$ with respect to $p_c^*$. We will now conclude the proof of the theorem assuming \eqref{e:crosstrans} and later prove the same. By planarity and the definition of the model, we have that there exists a left-right crossing of $\cO_p$ iff there exists no top-down crossing of $\cV_p$. From this observation and $\pi/2$-rotation invariance of $\cO_p$ and $\cV_p$, we obtain that $\BP(Cross_{2n,2n}(p)) = 1 - \BP(Cross^*_{2n,2n}(p))$. Thus, we also have that 
$$ \BP(Cross_{2n,2n}(p)) \overset{n \to \infty}{\to} 0  \, \,  \mbox{for}  \, \,  p < p^*_c \, \,  \mbox{and} \, \,   \BP(Cross_{2n2,n}(p))  \overset{n \to \infty}{\to} 1  \, \,  \mbox{for}  \, \,  p > p^*_c, \, \, $$
and so $p_c = p_c^*$ as required to complete the proof of the theorem. 

 We now prove \eqref{e:crosstrans} for $\BP(Cross_{2n,2n}(p))$ and that for $\BP(Cross^*_{2n,2n}(p))$ follows the same arguments.  For the first statement, observe that a left-right crossing in $[0,2n]^d$ has to pass through one of the $(n-1)$ balls $B_{(n,2i+1)}(1), i =0,2,\ldots,n-1$ and hence there is a path from one of these balls to $\{2n\} \times [0,2n]$. As in the proof of $\tilde{\gamma}_c = \gamma_c$ in Theorem \ref{t:eqcritprob}, by union bound, stationarity of the Poisson process and reasoning as in the derivation of \eqref{e:revprobbd_conf}, we obtain that 
\begin{align*} 
\BP(Cross_{2n,2n}(p)) & \leq \BP(\cup_{i=0}^{n-1} \{B_1(n,2i+1) \longleftrightarrow \{\{2n\} \times [0,2n]\}) \nonumber \\
& \leq  \sum_{i=0}^{n-1} \BP(B_1(n,2i+1) \longleftrightarrow \partial B_n(n,2i+1))  \nonumber \\ 
& \leq n \BP(B_0(1) \longleftrightarrow \partial B_n(0)) \leq b_{1/2}(p)^{-1} n\theta_n(p),
\end{align*}
where $b_r(p)$ is defined similarly as below \eqref{e:revprobbd_conf}. Now, by exponential decay in Theorem \ref{t:confettisharp}(1), the RHS above converges to $0$ as $n \to \infty$ for $p < p_c$. For the second statement, we shall follow Zhang's argument (see \cite[Proposition 4.1]{DCpercnotes} and \cite[Proposition 3]{Hirsch15}). For Zhang's argument, it suffices that $\cO_p$ is invariant under translations of $\R^d$ and rotation by $\pi/2$, $\cO_p$ is ergodic with respect to translations of $\R^d$ and black-increasing events are positively correlated. We shall now sketch Zhang's argument. Let $n \geq 2k \geq 1$. Any $\cO_p$-path from $B^{\infty}_k$ to $\partial B^{\infty}_n$ ends in one of the four sides of $\partial B^{\infty}_n$. Now, by $\pi/2$-rotation invariance of $\cO_p$ and square-root trick \cite[Proposition 4.1]{Tassion2016}, we have that 
\begin{align*}
\BP(\mbox{there is a $\cO_p$-path from $B^{\infty}_k$ to the left of $\partial B^{\infty}_n$})
 & \geq 1 - \BP(B^{\infty}_k \nleftrightarrow \partial B_n^{\infty})^{1/4} \\
 & \geq 1 - \BP(B^{\infty}_k \nleftrightarrow \infty)^{1/4}.
 \end{align*}
Let $A_{n,k}$ denote the event that $B'_{n,k} := (n,n) + B^{\infty}_k$ is connected to the left of $[0,2n]^2$ and the right of $[0,2n]^2$ via paths in $[0,2n]^2$. Then from the above bound and positive correlation of black-increasing events, we derive that
$$ \BP(A_{n,k}) \geq 1 - 2\BP(B^{\infty}_k \nleftrightarrow \infty)^{1/4}.$$
Observe that $A_{n,k} \setminus Cross_{2n,2n}(p) \subset A'_{n,k},$ where the latter is defined as
$$A'_{n,k} := \{ \mbox{there are two distinct components in $\cO_p$ that intersect $B'_{n,k}$ and $\partial [0,2n]^2$} \}.$$
Further, $\cap_{n \geq 1}A'_{n,k} \subset \{\mbox{there are two distinct infinite components in $\cO_p$}\}$ and the latter event has zero probability by uniqueness of the infinite component of $\cO_p$ (see Remark \ref{rem:uniqueness}). Thus, we obtain that
$$ \lim_{n \to \infty}\BP(Cross_{2n,2n}(p)) = \lim_{n \to \infty} \BP(A_{n,k}) \geq 1 - 2\BP(B^{\infty}_k \nleftrightarrow \infty)^{1/4}.$$
By ergodicity and that $\theta(p) > 0$ as $p > p_c$, we have that $\BP(\mbox{$\cO_p$ has an infinite component}) = 1$ and so $\BP(B^{\infty}_k \longleftrightarrow \infty) \to 1$ as $k \to \infty$. Substituting this in the above bound, we obtain the second statement in \eqref{e:crosstrans} and thus completing the proof of \eqref{e:crosstrans}. 
\end{proof}
\begin{remark}
\label{rem:uniqueness}
{\rm 
We do not provide a detailed proof for the uniqueness of the infinite occupied component in the confetti percolation model. To prove this claim, one can follow the arguments in \cite{Gandolfi1992}. The uniqueness of the infinite component of $\cO_p$ follows by ergodicity (with respect to translations of $\R^d$), invariant under $\pi/2$-rotations and reflections at the coordinate axes, positive association of black increasing events and asymptotic independence property of $\teta$ (see \cite[Section 3.2]{Hirsch15}).
}
\end{remark}
We now give a simple criterion for verifying noise sensitivity and exceptional times in the confetti percolation model analogous to the one for $k$-percolation in the Poisson Boolean model. 
\begin{theorem}[Noise sensitivity criteria]
\label{t:nscritconf}
Consider the occupied region $\cO_p$ in the confetti percolation model defined in \eqref{d:occconfetti}  with grain distributions $\BQ_1, \BQ_2$ satisfying the assumptions in \eqref{e:assumq}.  Let $f_n := f_n(p) = \1\{Cross_{\kappa n,n}(p)\}$ for $\kappa > 0$ where $Cross_{\kappa n,n}(p)$ is the crossing event defined in \eqref{e:defpercprobconf}.  Assume that $\{f_n(p_c)\}_{n \geq 1}$ is non-degenerate as in \eqref{e:deltadeg} and $\BP(Arm_{r,s}(p_c)) \to 0$ as $s \to \infty$. Then we have the following.
\begin{enumerate}[\rm (i)]
\item $\{f_n(p_c)\}_{n \geq 1}$ is noise sensitive. 

\item $\{f_n(p_c)\}_{n \geq 1}$ has an infinite set of exceptional times in $[0,1]$ i.e., $|S_n \cap [0,1]| \overset{\BP}{\longrightarrow} \infty$  where the set of exceptional times $S_n$ is defined below \eqref{e:deltadeg}.

\item Assume that there exists $C < \infty$ and $c_{arm} > 0$ such that for all $s > 2r$, the following holds :
$$\BP(Arm_{r,s}(p_c)) \leq Cs^{-c_{arm}}.$$
Then,  $\BP(Cross_{\kappa n,n}((1+c_n)p_c)) \to 1$ and $\BP(Cross_{\kappa n,n}((1-c_n)p_c)) \to 0$ where $c_n = n^{-\frac{c_{arm}}{2} + \epsilon}$ for some $\epsilon > 0$.
\end{enumerate}
\end{theorem}

\begin{remark}
{\rm 
\begin{enumerate}
\item Similarly to what was observed in Remark \ref{r:exns}-1, one can use Theorem \ref{t:fksharp} to show that the conclusion of Point (iii) in the previous statement continues to hold for sequences of the type $c_n = n^{-c_{arm} +\epsilon}$.

\item Reasoning as in Remark \ref{r:exns}(i), using the Harris-FKG inequality and monotonicity of the events, we can derive that $\BP(Arm_{r,s}(p_c)) \leq b_r^{-1}\theta_{s-r}(p_c)$ where $b_r$ is as defined below \eqref{e:revprobbd_conf}. This will be useful in Corollary \ref{c:nsplanarconfsymm}.
\end{enumerate}
}
\end{remark}

\begin{proof}
As with the proof of Theorem \ref{t:nscritbm}, we construct a suitable randomized stopping set and then appeal to our general results in Proposition \ref{p:randalgns} and Corollary \ref{c:randalgdyn} to prove the first two statements. The third statement does not directly follow from Theorem \ref{t:sharpphase} but needs adaptation of the techniques therein. 

 We fix a $\kappa > 0$. Let $R_n = [0,\kappa n] \times [0,n]^{d-1}$ be the rectangle. Given $s \in (0,\kappa n)$, let $L = \{s\} \times [0,n]^{d-1}$ and let $S$ be union of the connected components of $\cO_p \cap R_n$ that intersect $L$. Again if no such component exists, $S = \emptyset$. Define $Z_n^s := (S \cup L) \oplus B_{r}(0)$. By arguments as in the proof of  Theorem \ref{t:nscritbm} and Remark \ref{rA2}, we have that $\tilde{Z}_n^s =Z_n^s \times \R_+ \times \cK_r \times \cK_r \times \{0,1\}$ is a stopping set and determines $f_n$. As usual, we choose our randomized stopping set $Z_n$ by first choosing $Y \in (0,n)$ uniformly at random and then setting $\tilde{Z}_n^s = \tilde{Z}_n^Y$. As in Theorem \ref{t:nscritbm}, we derive the following bound on the revealment probability :
$$\delta_n := \sup_{\tx = (x,t,K_1,K_2,a) \in \tilde{\mathbb{X}}} \BP(\tx \in \tilde{Z}_n) \leq 2(\kappa n)^{-1}\int_0^{\kappa n} \BP(Arm_{r,s}(p_c)) \md s.$$
Thus the proof of the first statement of our theorem is complete using our assumption on decay of arm probability and Proposition \ref{p:randalgns}. Now the second statement will follow from Corollary \ref{c:randalgdyn}, provided we can show jump regularity of $f_n$. We will show jump-regularity at the end of the proof. 

 Now, we will prove the third statement. Set $c_n = n^{-c_{arm}/2 + \epsilon}$ for some $\epsilon > 0$. Observe that by our assumption on decay of arm-probabilities in the third statement, Corollary \ref{c:randalgdyn} and the jump-regularity of $f_n$ (to be proved below), we have that
\begin{equation}
\label{e:sncvgconf}
|S_n \cap [0,c_n]| \overset{\BP}{\longrightarrow} \infty.
\end{equation}
Abusing notation, we will let $\teta$ be the Poisson point process on $R^+_n \times \R_+ \times \cK_r \times \cK_r \times \{0,1\}$ with $p = p_c$ and $R_n^+ = R_n \oplus B_0(2r)$. Let $\teta^t$ be the Markov process driven by Ornstein-Uhlenbeck semigroup as in Section \ref{ss:ou} with stationary distribution same as that of $\teta$ and also $\teta^0 = \teta$. To make explicit the dependence on the point process, we set $f^t_n(p_c) := f_n(\teta^t)$ to denote the left-right crossing event of $R_n$ in $\cO_{p_c}(\teta^t)$ i.e., $Cross_{kn,n}(p_c)$ for $\cO_{p_c}(\teta^t)$. Now, using \eqref{e:sncvgconf}, we have that 
\begin{align}
\label{e:excepttimeconf}
\BP(f^t_n(p_c) = 1 \, \mbox{for some $t \in [0,c_n]$}) \to 1 \, \mbox{as $n \to \infty$}.
\end{align}
For $t\ge 0$, let $\teta^t_+$ be the Poisson process of points born before time $t$ and $\teta^t_- \subset \teta$ be the Poisson process of points that have been removed by time $t$. Let $\teta^t_{b+} \subset \teta^t_+$ be the subset of `black grains' i.e., $\teta^t_{b+} = \teta^t_+ \cap (R^+_n \times \R_+ \times \cK_r \times \cK_r \times \{0\})$. Similarly, we denote $\teta^t_{w-} \subset \teta^t_-$ be the subset of `white grains' that were removed by time $t$.  As the process $\teta^t_+$ is independent of $\teta$, so is the process $\teta^t_{b+}$ and the latter process has intensity measure $p_ct\tilde{\gamma}$, where recall that $\tilde{\gamma}$ is the intensity measure of the Poisson point process $\teta$ in the above paragraph. Then $\zeta^t:=\teta - \teta^t_{w-} +\teta^t_{b+}$ is a Poisson process with intensity measure
$(p_c + e^{-t}(1-p_c) + tp_c)\tilde{\gamma}$ and the intensity measure of the `black grains' is $(1+t)p_c \tilde{\gamma}$. So, the probability of a grain being black is $\frac{(1+t)p_c}{p_c + e^{-t}(1-p_c) + tp_c}$. We let $\teta^t_b, \teta^t_w$ denote respectively the black and white grains in $\teta^t$. Similarly, we use the notation $\zeta^t_b,\zeta^t_w$. By construction of $\teta^t$ and $\zeta^t$, the `black grains' are never removed in $\zeta^t$ while they may be removed in $\teta^t$ depending on their lifetimes. Also, `white grains' are removed in $\zeta^t$ and not added while newer white grains may be added in $\teta^t$. Thus, $\teta^t_b \subset \zeta^t_b \subset \zeta^{c_n}_b$ and $\teta^t_w \supset \zeta^t_w \supset \zeta^{c_n}_w$ for $t \leq c_n$. Since $f_n$ is monotonic \eqref{e:blackinc}, $f_n(\eta^t) \leq f_n(\zeta^t) \leq f_n(\zeta^{c_n}), t \in [0,c_n]$. Now, by \eqref{e:excepttimeconf}, we have that $\BE[f_n(\zeta^{c_n})] \to 1$ as $n \to \infty$. Since the probability of a grain being black in $\zeta^{c_n}$ is at most $(1+c_n)p_c$, again by monotonicity, we have that $\BP(Cross_{kn,n}((1+c_n)p_c)) \to 1$ as $n \to \infty$. Adapting the arguments in Theorem \ref{t:sharpphase} as we did above, we can obtain the second part of the statement also. 

 Only jump-regularity remains to be proven now to complete the proof of the theorem. Unlike in the $k$-percolation case, it is not true that $\teta \cap (R_n^+ \times \R_+ \times \cK_r \times \cK_r \times \{0,1\})$ has a finite intensity measure for any $n$. Nevertheless, we can still show that $f_n$ is well-approximate by functionals that depend on $\teta$ restricted to a compact set and are hence jump-regular. 

 Let $\teta_h := \teta \cap \BX_h$ where $\BX_h := R^+_n \times [0,h] \times \cK_r \times \cK_r \times \{0,1\}$ i.e., $\teta_h$ is the Poisson process restricted to height $h$ for some $h > 0$. Note that $\teta^t_h := \teta^t \cap \BX_h$ is the Markov process driven by Ornstein-Uhlenbeck semigroup with stationary distribution same as that of $\teta_h$. Analogous to $f^t_n(p_c) := f_n(\teta^t)$, define $f^t_{n,h}(p_c) := f_n(\teta^t_h)$ to be the indicator of the crossing event $Cross_{\kappa n,n}(p_c)$ in $\cO_p(\teta_h^t)$ for $h,t > 0$. Since $\tilde{\gamma}( \BX_h) < \infty$, $f^t_{n,h}(p_c)$ is jump-regular as it depends only on the Poisson process inside the compact set $\BX_h$. Suppose we show that
\begin{equation}
\label{e:fnth}
 \lim_{h \to \infty} \sup_{t \in [0,1]}|f^t_n(p_c) - f^t_{n,h}(p_c)| \overset{p}{=} 0.
\end{equation}
Then, we claim that $f_n^t(p_c)$ is also jump-regular. We argue this as follows : Let $S_n$ be the set of exceptional times of $f_n^t$ restricted to $[0,1]$ (see the definition of exceptional times in Section \ref{s:dynbm}). Similarly, set $S_{n,h}$ to be the set of exceptional times of $f_{n,h}^t$ restricted to $[0,1]$. To show jump-regularity of $f_n^t(p_c)$, it suffices to show that $\BP(S_n < \infty) = 1$ as $f_n^t(p_c)$ is $\{0,1\}$-valued. Fix $M,h > 0$. Using that $S_n = S_{n,h}$ if $f^t_n \equiv f^t_{n,h}$ for $t \in [0,1]$ and that $f^t_n, f^t_{n,h}$ are $\{0,1\}$-valued, we derive that 
\begin{align*}
\BP(S_n \geq M) & \leq \BP(S_{n,h} \geq M) + \BP(\sup_{t \in [0,1]}|f^t_n(p_c) - f^t_{n,h}(p_c)| \neq 0) \\
& = \BP(S_{n,h} \geq M) + \BP(\sup_{t \in [0,1]}|f^t_n(p_c) - f^t_{n,h}(p_c)| \geq 1)
\end{align*}
Now, by letting $M \to \infty$ first and then letting $h \to \infty$ and using jump-regularity of $f^t_{n,h}(p_c)$ and \eqref{e:fnth}, we obtain that $\lim_{M \to  \infty}\BP(S_n \geq M) = 0$. Thus, we have that $\BP(S_n < \infty) = 1$ and hence that $f_n^t$ is jump-regular.  

 We are now left to prove \eqref{e:fnth}. We sub-divide $R^+_n$ into small cubes $R_{n1},\ldots,R_{nk_a}$ of side-length $a$ and $a$ small enough such that the diameter of each cube is at most $r_0/4$ for $r_0$ as in assumption \eqref{e:assumq}. Set $\cK'_{r_0} := \{K \in \cK_r : B_{r_0}(0) \subset K \}$. Our key observation is that if for some $t$ and all $i$, $\teta^t_h(R_{ni} \times [0,h] \times \cK'_{r_0} \times \cK'_{r_0} \times \{0,1\}) \neq 0$, then $\cO_p(\teta^t_h) = \cO_p(\teta^t)$ and so $f_n^t(p_c) = f_{n,h}^t(p_c)$. Thus, we have that
\begin{equation}
\label{e:covevent}
\{\sup_{t \in [0,1]}|f^t_n(p_c) - f^t_{n,h}(p_c)| \neq 0 \} \subset \bigcup_{i=1}^{k_a}\bigcup_{t \in [0,1]} \{ \teta^t_h(R_{ni} \times [0,h] \times \cK'_{r_0} \times \cK'_{r_0} \times \{0,1\}) =0\}.
\end{equation}
We now show that the probability of the latter event can be made small by choosing a large $h$. By stationarity of the Poisson point process $\teta_h$, it suffices to show that for each $i \leq k_a$, 
\begin{equation}
\label{e:smallprobi}
\lim_{h \to \infty} \BP(\bigcup_{t \in [0,1]} \{ \teta^t_h(R_{ni} \times [0,h] \times \cK'_{r_0} \times \cK'_{r_0} \times \{0,1\}) = 0 \}) = 0.
\end{equation}

  Let us fix a $i \leq k_a$ to prove \eqref{e:smallprobi}. Set $N_h := \teta_h(R_{ni} \times [0,h] \times \cK'_{r_0} \times \cK'_{r_0} \times \{0,1\})$. Note that $N_h$ is Poisson distributed with mean $a^d\beta_1 \beta_2 h$ where $\beta_i = \BQ_i(\cK'_{r_0}), i =1,2$. By assumption \eqref{e:assumq}, $\beta_i > 0, i =1,2$. By the definition of the Markov process $\teta^t_h$, if  $\teta^t_h(R_{ni} \times [0,h] \times \cK'_{r_0} \times \cK'_{r_0} \times \{0,1\}) = 0$ for some $t \in [0,1]$, then either $\teta_h(R_{ni} \times [0,h] \times \cK'_{r_0} \times \cK'_{r_0} \times \{0,1\}) = 0$ or the life-time of every particle in $\teta_h \cap R_{ni} \times [0,h] \times \cK'_{r_0} \times \cK'_{r_0} \times \{0,1\}$ is at most $1$. Since the lifetimes are all i.i.d.\ exponential($1$) random variables and $N_h$ is a Poisson random variable, we can now compute the required probability. 
\begin{align}
& \BP(\bigcup_{t \in [0,1]} \{ \teta^t_h(R_{ni} \times [0,h] \times \cK'_{r_0} \times \cK'_{r_0} \times \{0,1\}) = 0 \}) \nonumber \\
& \leq \BP(\teta_h(R_{ni} \times [0,h] \times \cK'_{r_0} \times \cK'_{r_0} \times \{0,1\}) = 0)  \nonumber  \\
& \quad + \BP(\mbox{life-time of every particle in $\teta_h \cap R_{ni} \times [0,h] \times \cK'_{r_0} \times \cK'_{r_0} \times \{0,1\}$ is at most $1$})  \nonumber  \\
& = \exp\{-a^d\beta_1\beta_2h\} + \BE[(1-e^{-1})^{N_h}]  \nonumber  \\
\label{e:confhbound} & = \exp\{-a^d\beta_1\beta_2h\} + \exp\{-a^d\beta_1\beta_2 he^{-1}\}.
\end{align}
Thus letting $h \to \infty$, we obtain \eqref{e:smallprobi} which gives \eqref{e:fnth} via \eqref{e:covevent}. This completes the proof of jump-regularity of $f_n^t(p_c)$ and hence the proof of the theorem.
\end{proof}
Combining the above theorem with \cite[Theorem 8.3]{Ahlbergsharp18}, we get the following corollary in the case of planar confetti percolation model with grains supported on balls centred at origin. Recall that $\cK_r^b$ is the space of all balls of radius at most $r$ centred at origin (see above \eqref{e:assumq}).
\begin{corollary}
\label{c:nsplanarconf}
Consider the occupied region $\cO_p$ in the planar confetti percolation model (i.e., $d =2$) as defined in \eqref{d:occconfetti} with grain distributions $\BQ_1, \BQ_2$ satisfying assumption \eqref{e:assumq} and further $\BQ_1,\BQ_2$ are supported on $\cK_r^b$. Then for $f_n := f_n(p_c)$ as defined in Theorem \ref{t:nscritconf}, the conclusions of Theorem \ref{t:nscritconf} hold.
\end{corollary}
We now give a second application to noise sensitivity of confetti percolation using the arguments in Theorem \ref{t:critprobconfplanar}. This applies to more general grains distribution than balls but is less quantitative. 
\begin{corollary}
\label{c:nsplanarconfsymm}
Consider the occupied region $\cO_p$ in the planar confetti percolation model (i.e., $d =2$) as defined in \eqref{d:occconfetti} with grain distributions $\BQ_1 = \BQ_2 = \BQ$ satisfying assumption \eqref{e:assumq}. Further, assume that $\BQ$ is invariant under $\pi/2$-rotations and reflections at the coordinate axes. Then for $f_n := f_n(1/2) = \1\{Cross_{n,n}(1/2)\}$, the first two statements in the conclusions of the Theorem \ref{t:nscritconf} hold.
\end{corollary}
\begin{proof}
Observe that by planar duality, a left-right crossing of $[0,n]^2$ by $\cO_p$ exists iff there is no top-down crossing of $[0,n]^2$ by $\cV_p$. Further using $\pi/2$-rotation invariance of $\BQ$, we have that $\BP(Cross_{n,n}(p)) = 1 - \BP(Cross^*_{n,n}(p))$ for all $p \in [0,1]$. By self-duality at $p = 1/2$, we obtain that  $\BP(Cross_{n,n}(1/2)) = \BP(Cross^*_{n,n}(1/2)) = 1/2$ for all $n \geq 1$. Thus, $f_n$ is non-degenerate as in \eqref{e:deltadeg}. 

 As remarked below Theorem \ref{t:nscritconf}, we have that 
$$ \BP(Arm_{r,s}(1/2)) \leq b_r^{-1}\theta_{s-r}(1/2),$$
where $b_r$ is as defined below \eqref{e:revprobbd_conf} and $s > 2r$. From Zhang's argument as elaborated in the proof of Theorem \ref{t:critprobconfplanar}, we obtain that if $\theta(1/2) > 0$ then $\BP(Cross_{n,n}(1/2)) \to 1$ as $n \to \infty$. Since $\BP(Cross_{n,n}(1/2)) = 1/2$ for all $n \geq 1$, we have that $\theta(1/2) = 0$. Thus, we also obtain that $\lim_{s \to \infty} \theta_{s-r}(1/2) = 0$ and thereby verifying the assumptions of Theorem \ref{t:nscritconf}.
\end{proof}

\section{Noise sensitivity for the planar Boolean model with unbounded balls}
\label{s:pbmunbdd}

While the previous two sections studied sharp phase transition and noise sensitivity for some percolation models with bounded range of dependence, we will now show that our noise sensitivity results can be extended to a model with a unbounded range of dependence. We will focus on a very specific model - the planar Poisson Boolean model  - as many explicit estimates are known here and this will best illustrate the usefulness of our results and also the proof ideas.

We shall adopt the framework of \cite{Ahlbergsharp18}. Let $\teta := \{(X_i,R_i)\}_{i \geq 1} \subset \R^2 \times \R_+$ be a Poisson point process of intensity $\gamma \, \md x  \, \BQ(dr)$ where $\BQ$ is a probability distribution on $\R_+$ called {\em the radii distribution} such that for some $\alpha > 0$, 
\begin{equation}
\label{e:2mom}
\int_0^{\infty} r^{2+\alpha} \BQ(dr) \in (0,\infty).
\end{equation}
In other words, $\{X_i\}_{i \geq 1}$ is a stationary Poisson point process of intensity $\gamma$ and the points are equipped with i.i.d.\ marks $R_i, i \geq 1$. The {\em occupied region of the planar Poisson Boolean model} is defined as
\begin{equation}
\label{e:occpbmunbdd}
\cO(\gamma) := \cO(\teta) = \bigcup_{i \geq 1} B_{R_i}(X_i).
\end{equation}
Similarly, we will use other notations from Section \ref{s:kpercbm}. In particular, one can define a critical intensity $\gamma_c$ as in \eqref{e:critint1}.  It was shown in \cite{Hall85} that \eqref{e:2mom} with $\alpha = 0$ is necessary for $\gamma_c > 0$ and much later \cite{Gouere2008} showed that \eqref{e:2mom} with $\alpha = 0$ is sufficient for $\gamma_c > 0$. Thus, using a standard coupling with the bounded radius case, one can obtain that $\gamma_c \in (0,\infty)$ under condition \eqref{e:2mom} with $\alpha = 0$; see for example, \cite[Theorem 1.1]{Ahlbergsharp18}. We have assumed $\alpha > 0$ as the more quantitative estimates on arm probabilities in \cite{Ahlbergsharp18} require this assumption. 

Our main theorem in this section yields quantitative noise sensitivity, existence of exceptional times and bounds on critical window for crossing probabilities in the planar Poisson Boolean model as defined above. This was suggested as an open problem in \cite{Ahlbergsharp18}. 
\begin{theorem}[Noise sensitivity for planar Poisson Boolean model]
\label{t:nsplbmunbdd}
Let $\teta$ be the Poisson point process as defined above with $\gamma = \gamma_c$ and $\cO(\gamma)$ be the occupied region of the Poisson Boolean model as defined in \eqref{e:occpbmunbdd} with the radii distribution $\BQ$ satisfying \eqref{e:2mom}. Let $\kappa > 0$ be given. Let $f_n(\gamma_c) = f_n(\teta) := \1\{Cross_{\kappa n,n}(\teta)\}$ where $Cross_{\kappa n,n}(\teta)$ is the crossing event defined in \eqref{e:defpercprob1}. Then the following statements hold.

\begin{enumerate}
\item[(i)] $\{f_n(\gamma_c)\}_{n \geq 1}$ is noise sensitive. 

\item[(ii)] $\{f_n(\gamma_c)\}_{n \geq 1}$ has an infinite set of exceptional times in $[0,1]$ i.e., $|S_n \cap [0,1]| \overset{\BP}{\longrightarrow} \infty$  where the set of exceptional times $S_n$ is as defined below \eqref{e:deltadeg}.

\item[(iii)] There exists $b > 0$ such that $\BE[f_n((1+n^{-b})\gamma_c)] \to 1$ and $\BE[f_n((1- n^{-b})\gamma_c)] \to 0$ as $n \to \infty$. 
\end{enumerate}
\end{theorem}
The last statement can be used to provide an alternative proof of \cite[Theorem 5.1]{Ahlbergsharp18} which shows sharpness of phase transition of crossing probabilities via discretization, an inequality of Talagrand \cite{Talagrand94} and the Russo-Margulis formula. 
\begin{proof}
The proof of the theorem relies on our exact estimates for the case of bounded balls in Theorem \ref{t:nscritbm} and the following observation relating noise sensitivity of a sequence of functions to its approximation. Suppose that $f_n,f'_n : \bN \to [0,1], n \geq 1$ are two sequences of functions. Then, using the triangle inequality, we have that
\begin{equation}
\label{e:covbd}
|\BC[f_n(\teta^t),f_n(\teta)] - \BC[f'_n(\teta^t),f'_n(\teta)]| \leq 4 \BP(f_n(\teta) \neq f'_n(\teta)), 
\end{equation}
where $\teta^t, t \geq 0$ is the Markov process defined in Section \ref{ss:ou}. This is the Markov process driven by Ornstein-Uhlenbeck semigroup with stationary distribution same as that of $\teta$ and also $\teta^0 = \teta$. 

 Choose a sequence $r_n = n^{1-\epsilon}, n \geq 0$ for an $\epsilon < \frac{\alpha}{2 + \alpha}$. Define $\teta'_n := \{(X_i,R_i) \in \teta : R_i \leq r_n \}$ be the restricted Poisson point process and $f'_n(\teta) := \1\{Cross_{\kappa n,n}(\teta'_n)\}.$ Because of \eqref{e:covbd}, the proof of the theorem is complete if we show that $f'_n$ is noise sensitive and that 
\begin{equation}
\label{e:fnapprox}
\lim_{n \to \infty} \BP(f_n(\teta) \neq f'_n(\teta)) = 0
\end{equation}

 Firstly, we show that $f'_n(\teta)$ is noise sensitive. Following the proof of Theorem \ref{t:nscritbm}, we have that there exists a randomized stopping set $Z'_n$ determining $f'_n(\teta)$ and satisfying the assumptions of Proposition \ref{c:sspoisson} such that
$$ \delta'_n := \delta(Z'_n) \leq \frac{2}{\kappa n}\int_0^{\kappa n} \BP(Arm_{r_n,s}(\teta'_n)) \md s,$$
where $Arm_{r,s}(\cdot)$ is the one-arm crossing event as defined in \eqref{e:defpercprob1}. Since the one-arm crossing event is increasing and further using \cite[Corollary 4.4]{Ahlbergsharp18}, we obtain that for some $c_{arm} >0$,
$$\delta'_n \leq \frac{2}{\kappa n}\int_0^{\kappa n} \BP(Arm_{r_n,s}(\teta)) \md s \leq  \frac{C}{\kappa n} \left( 2r_n + \int_{2r_n}^{\kappa n} \Big( \frac{r_n}{s} \Big)^{c_{arm}}  \md s \right) \to 0,$$
where the convergence is due to the choice of $r_n$. Thus $f'_n$ is noise sensitive by Proposition \ref{c:sspoisson}. Now we will show \eqref{e:fnapprox}. 

 Observe that by definition,
\begin{equation}
\label{e:fnneqf'n}
\BP(f_n(\teta) \neq f'_n(\teta)) \leq \BP(\exists \, i \in \N, R_i \geq r_n \, \, \mbox{and} \, \, B_{R_i}(X_i) \cap [0,\kappa n]^2 \neq \emptyset).
\end{equation}
From arguments similar to that in \cite[(2.13)]{Ahlbergsharp18}, we have that for some constant $C$, 
\begin{align}
& \BP(\exists  \, i \in \N, R_i \geq r_n \, \, \mbox{and} \, \, B_{R_i}(X_i) \cap [0,\kappa n]^2 \neq \emptyset) \nonumber \\ 
& \leq \BP(\teta ([0,\kappa n + r_n]^2 \times [r_n,\infty)) \neq 0) \nonumber \\
& \quad +  \BP(\exists \, i \in \N, X_i \notin [0,\kappa n + r_n]^2 \, \, \mbox{and} \, \, B_{R_i}(X_i) \cap [0,\kappa n]^2 \neq \emptyset) \nonumber \\ 
& \leq  C\gamma_c \Big( 1 + \frac{\kappa n + 1}{r_n} \Big)^2 \left( r_n^2 \BQ((r_n,\infty)) + \int_{r_n}^{\infty}r^2\BQ(\md r) \right) \nonumber \\
\label{e:bdtrunbmunbdd} & \leq 2C\gamma_c \Big( 1 + \frac{\kappa n + 1}{r_n} \Big)^2r_n^{-\alpha} \int_{r_n}^{\infty}r^{2+\alpha}\BQ(\md r)  \leq 8C\gamma_c n^{-\alpha + \epsilon(2 + \alpha)} \int_{r_n}^{\infty}r^{2+\alpha}\BQ(\md r) \to 0,
\end{align}
where in the last inequality we have again used the choice of $r_n$ and in the convergence, we have used the choice of $\epsilon$. Thus, using \eqref{e:fnneqf'n}, we have proven \eqref{e:fnapprox} and hence the proof of the first statement is complete via \eqref{e:covbd}. Further, by using the above bounds and  \eqref{e:quantns1} for covariance of $f'_n$, we obtain that
\begin{equation}
\label{e:covbdfn}
\BC[f_n(\teta^t),f_n(\teta)] \leq 8C\gamma_c n^{-\alpha + \epsilon(2 + \alpha)} \int_{r_n}^{\infty}r^{2+\alpha}\BQ(\md r)  + C\delta'_n\frac{e^{-t}}{(1-e^{-t})^2}.
\end{equation}
The second and third statement follow from \eqref{e:covbdfn}, the explicit bounds for $\delta'_n$ and Theorems \ref{t:exctimdyn} and \ref{t:sharpphase} (see remark below the theorem) if we show jump-regularity of $f_n$. We will do so now by using a strategy similar to that in the proof of Theorem \ref{t:nscritconf}.  

 Fix $h > 0$. Let $\teta_h := \teta \cap ([0,\kappa n+h]^2 \times [0,h])$ be the truncated Poisson process and $g_{n,h}(\teta) = \1\{Cross_{\kappa n,n}(\teta_h)\}$ denote the indicator of the crossing event of the window  $[0,\kappa n ] \times [0,n]$ by $\cO(\teta_h)$. Since $\teta_h$ has finite intensity measure, $g_{n,h}$ is jump-regular.  As argued in Theorem \ref{t:nscritconf} (see \eqref{e:fnth} and below), the proof of jump-regularity of $f_n$ is complete if we show that
\begin{equation}
\label{e:cvgprobfngnh}
  \lim_{h \to \infty} \sup_{t \in [0,1]}|f_n(\teta^t) - g_{n,h}(\teta^t)| \overset{p}{=} 0.
\end{equation}
Let $\teta_+$ denote the new points born in time interval $[0,1]$ in the Markov process $\teta^t$. By the definition of the Markov process, $\teta_+$ is independent of $\teta$ and has intensity measure $\gamma_c \, \md x\, \BQ(\md r)$ i.e., $\teta_+ \overset{d}{=} \teta$. Using this and the definitions of $f_n,g_{n,h}$, we derive that
\begin{align*}
& \BP(\sup_{t \in [0,1]}|f_n(\teta^t) - g_{n,h}(\teta^t)| \geq 1) \\
& \leq  \BP(\exists t \in [0,1],  \teta^t(\{(x,r) : r > h, B_{r}(x) \cap [0,\kappa n]^2 \neq \emptyset\}) \neq 0) \\ 
& \leq  \BP(\teta(\{(x,r) : r > h, B_{r}(x) \cap [0,\kappa n]^2 \neq \emptyset\}) \neq 0) \\
& \quad +  \BP(\teta_+(\{(x,r) : r > h, B_{r}(x) \cap [0,\kappa n]^2 \neq \emptyset\}) \neq 0) \\
& \leq 2  \BP(\teta(\{(x,r) : r > h, B_{r}(x) \cap [0,\kappa n]^2 \neq \emptyset\}) \neq 0) \\
& \leq 4C\gamma_c \Big( 1 + \frac{\kappa n + 1}{h} \Big)^2 \int_{h}^{\infty}r^2\BQ(\md r)
\end{align*}
where in the last inequality we have used the bounds as derived in \eqref{e:bdtrunbmunbdd} but for $h$ instead of $r_n$. From our last bound the required convergence in \eqref{e:cvgprobfngnh} follows and hence the proof of jump-regularity of $f_n$ and also the proof of the Theorem is complete.
\end{proof}
%
%
The $2+\alpha$ moment condition \eqref{e:2mom} can be replaced by $d+\alpha$ moment condition in $d$-dimensions and while most of the above proof would go through, the arm probability estimates of \cite[Corollary 4.4]{Ahlbergsharp18} are not known and this is the main obstacle in extending the theorem to higher-dimensions.
\section{Further applications}
\label{s:furthapplns}

We now briefly indicate some further possible applications of our main theorems without going into details. 

\textbf{Vacant-set percolation:} Our proof methods can also be adapted to show sharp phase transition for  $\cO(\gamma)^c$ as well i.e., the region covered by at most $(k-1)$-balls. If $k = 0$, this is the vacancy region in the usual Boolean model and one can find sharp phase transition results for vacant region of the Boolean model with random radii (including the case of unbounded support in the planar case) in \cite{Ahlbergsharp18, Ahlbergvacant18,Penrose2018,DCRT18}. Further, phase transition results for the Boolean model with random radii with unbounded support have been proven in \cite{Gouere2008,Gouere2019,DCRT18}. A challenging question would be to extend the results of \cite{DCRT18} to Poisson Boolean model with general grains and possibly to $k$-percolation model.  \\

\textbf{Voronoi Percolation:} The results of the previous section on confetti percolation can also be shown to hold for the Voronoi percolation model on the Poisson process. These results were proven in \cite{DCRT19,Ahlberg18} by suitable discretization and application of OSSS and Schramm-Steif inequalities respectively. Our construction of stopping sets and proof methodology in the previous section can be extended to the case of Voronoi percolation as well. This will yield sharp phase transition as well as noise sensitivity and exceptional times at criticality more directly from the results of \cite{Tassion2016}. Further, it should be possible to obtain analogues of Corollary \ref{c:nsplanarbm} for more general Voronoi percolation models as defined in \cite[Section 8]{Ahlbergsharp18} and \cite[Section 1.2]{Ghosh18}. The necessary estimates for non-degeneracy of crossing events and decay of arm-probabilities were established in \cite[Theorem 8.1]{Ahlbergsharp18}.  \\

\remove{\paragraph{Confetti Percolation:} This model also known as "dead leaves model" was introduced in \cite{Jeulin97}. In this model, colored leaves or confetti fall on vertices of a Poisson process at random times and points in plane are assigned the colour of the first confetti to fall on them. Here again, non-degeneracy of crossing events and decay of arm-probabilities was established in \cite[Theorem 8.3]{Ahlbergsharp18}. One would again expect to obtain analogue of Corollary \ref{c:nsplanarbm} for this general confetti percolation model. }

\textbf{Level set percolation:} Percolation of level sets of Poissonian shot-noise random fields has been another model that has received some attention (see  \cite{Molchanov83,Alexander96level,Broman17}). More recently, in \cite{LRM2019} RSW-type estimates, non-degeneracy of crossing events, decay of arm-probabilities and sharp phase transition were proven. In particular, \cite[Proposition 4.2]{LRM2019} is proved using the discrete OSSS inequality and our continuum analogue of OSSS inequality can again be used to avoid discretization. Further, one would expect to deduce noise sensitivity and exceptional times for crossings in this model at criticality analogous to Corollary \ref{c:nsplanarbm}. \\

\textbf{Other continuum models:} One can apply our methods also to prove a sharp phase transition in random connection model with bounded edges and Miller-Abrahams random resistor network with lower-bounded conductances more straightforwardly (see \cite{Faggionato2019}). Another model where it would be interesting to apply our results is the Poisson stick model \cite{Roy91}. We had mentioned that $k$-percolation in the Boolean model corresponds to face percolation in the C\v{e}ch complex. One can also analogously define percolation of faces in the Vietoris-Rips complex built on the Poisson point process. This is same as clique percolation on the random graph induced by the Boolean model. Clique percolation was introduced in  \cite{Derenyi05}. See \cite{Iyer2020} for the above two models as well as two other models describing connectivity of faces.  Percolation in all these models can be studied using our results. 

\begin{appendix}

\renewcommand{\thesubsection}{\Alph{subsection}}
\counterwithin{theorem}{subsection}

\renewcommand{\theequation}{\thesubsection.\arabic{equation}}
\setcounter{equation}{0}

\section*{Appendices}

In the forthcoming Appendices \ref{appendix1} \&
  \ref{appendix2}, we consider the general setup of Section
  \ref{ss:intropoisson}. In particular, $(\BX,\cX)$ denotes a Borel
  space endowed with a localizing ring $\cX_0$ defined in
terms of an increasing sequence $\{B_n : n\geq 1\}$; see 
Subsection \ref{ss:intropoisson}.
Given a locally finite measure $\lambda$, we write
  $\Pi_\lambda$ to indicate the law of a Poisson process $\eta$ on
  $\BX$ such that $\lambda =\BE[\eta(\cdot)]$.

\subsection{Graph-measurable mappings and stopping sets} \label{appendix1}


A mapping $Z\colon \bN\to\cX$ is called {\em graph-measurable} (see \cite{Molchanov05})
if $(x,\mu)\mapsto \I\{x\in Z(\mu)\}$ is a measurable mapping on $\BX\times\bN$.
Note that we do not equip $\cX$ with a $\sigma$-field and that the mapping
$Z$ is not assumed to be a measurable mapping in any sense.
All that is required in our paper is the following measurability property.


\begin{lemma}\label{lapp} Suppose that
$Z\colon \bN\to\cX$ is graph-measurable. Then, writing $Z(\mu)^c = \BX\backslash Z(\mu)$, the mapping
$$\bN\times \bN \rightarrow \bN\times\bN\times\bN\times\bN :  (\nu,\mu)\mapsto (\nu_{Z(\mu)},\mu_{Z(\mu)}, \nu_{Z(\mu)^c}, \mu_{Z(\mu)^c} )$$ is measurable.
\end{lemma}
\begin{proof} It suffices to prove that the first and third components of the above mapping
are measurable. The proof proceeds as that of Lemma 115-(i) in
\cite{Kallenberg17}. Let $B\in\cX$. We need  to show
that $(\nu,\mu)\mapsto \nu_{Z(\mu)}(B)$ is measurable.
By definition of a localizing ring and monotone convergence
we may assume that $B\in\cX_0$.
By the monotone class theorem (using Dynkin systems) it can be shown that
the mapping
\begin{align*}
(\nu,\mu)\mapsto \int \I\{(x,\mu)\in A,\, x\in B\}\,\nu(dx)
\end{align*}
is measurable for each $A\in\cX\otimes\cN$. In particular, writing $H(\mu)$ for either $Z(\mu)$ or $Z(\mu)^c$ and applying the graph-measurability of $Z$ (and therefore of $Z^c$), one deduces that
$\int \I\{x\in H(\mu),\, x\in B\}\,\nu(dx)$
is a measurable function of $(\nu,\mu)$, as asserted.
\end{proof}

We say that a graph-measurable mapping $Z\colon \bN\to\cX$ is a {\em stopping set}, if
\begin{align}\label{estopset}
Z(\mu)=Z(\mu_{Z(\mu)}+\psi_{Z(\mu)^c}),\quad \mu,\psi\in\bN.
\end{align}
In particular a stopping set satisfies
\begin{align}\label{etr5}
Z(\mu)=Z(\mu_{Z(\mu)}),\quad \mu\in\bN,
\end{align}
so that $(x,\mu)\mapsto\I\{x\in Z(\mu)\}$ is
$\cX\otimes\cN_Z$-measurable, where $\cN_Z$ is the $\sigma$-field on
$\bN$ generated by $\mu\mapsto \mu_{Z(\mu)}$. Observe that, if $Z'$ is
another stopping set such that $Z(\mu) \subset Z'(\mu)$ for every
$\mu\in \bN$, then \eqref{etr5} implies that
$\mathcal{N}_Z\subset \mathcal{N}_{Z'}$.

We say that a measurable function $f\colon \bN\to\R$ is {\em determined} by a stopping set
$Z$ if $f(\mu)=f(\mu_{Z(\mu)})$ for each $\mu\in\bN$. This means that
$f$ is $\cN_Z$-measurable. In this case we even have that
\begin{align}\label{edetermined}
f(\mu)=f(\mu_{Z(\mu)}+\psi_{Z(\mu)^c}),\quad \mu,\psi\in\bN.
\end{align}
This follows by applying the determination property of $f$ with
$\mu$ replaced by $\mu_{Z(\mu)}+\psi_{Z(\mu)^c}$ and then \eqref{estopset}.

In contrast to the classical theory of stopping sets (see e.g.\
\cite{Zuyev99}) we do not assume that $Z$ is taking its values in the
space of closed (or compact) subsets of a locally compact Polish
space. But if this is the case and if $Z$ is measurable with respect
to the Borel $\sigma$-field generated by the Fell topology, then
\cite[Proposition A.1]{BaumstarkLast09} shows that \eqref{estopset} is
equivalent to the standard definition of a stopping set,
at least in the case of simple point measures. Moreover, in this
case the $\sigma$-field $\cN_Z$ coincides with the classical
stopped $\sigma$-field, given as the system of all
sets $A\in\cN$ such that $A\cap \{Z\subset K\}\in \cN_K$
for each compact $K\subset\BX$.
Our definition of a stopping set seems to be new in this generality. As
already observed, it is very convenient for the applications developed
in our work, since it allows one to avoid several restrictive (and
technical) measurability and topological assumptions that typically
emerge when using the classical theory.

But in the case $Z$ is measurable with respect
to the Borel $\sigma$-field generated by the Fell topology, then
\cite[Proposition A.1]{BaumstarkLast09} shows that \eqref{estopset} is
equivalent to the standard definition of a stopping set. 

We will need the following simple lemma.

\begin{lemma}\label{lA1} Suppose that $Z$ is a stopping set and let $\varphi,\psi\in\bN$.
Then $\mu_{Z(\mu)}=\psi$ if and only if $\mu_{Z(\psi)}=\psi$. In this case
$Z(\mu)=Z(\psi)$.
\end{lemma} 
\begin{proof} If $\mu_{Z(\mu)}=\psi$, then \eqref{estopset}
yields $Z(\mu)=Z(\psi)$ and in particular $\mu_{Z(\psi)}=\psi$.
If, conversely, the latter holds, we 
have that $\psi$ is supported by $Z(\psi)$ and
$\mu_{Z(\psi)}=\psi_{Z(\psi)}$. Applying \eqref{estopset} 
(with $\psi$ in place of $\mu$) we obtain that
\begin{align*}
Z(\psi)=Z(\psi_{Z(\psi)}+\mu_{Z(\psi)^c})=Z(\mu_{Z(\psi)}+\mu_{Z(\psi)^c})=Z(\mu).
\end{align*}
This finishes the proof.\end{proof}

\bigskip

We now fix a Poisson process $\eta$ on $\BX$ with locally finite intensity $\lambda$.
The following result is well-known for classical stopping sets
and can be attributed to \cite[Theorem 4]{Rozanov}; see also
\cite{Zuyev99,Zuyev06}. In our Poisson setting (and still adopting the classical definition of stopping sets from \cite{Zuyev99}) 
it can be derived from Lemma A.3 in \cite{BaumstarkLast09}
and the multivariate Mecke equation \eqref{e:mecke}. Our proof shows that such an approach
applies to the more general definiton stopping set adopted in this paper.

\begin{theorem}\label{tAstopping} Suppose that $Z$ is a stopping set and let $A\in\cN$. Then
\begin{align}\label{eMarkov}
\BP(\eta_{Z(\eta)^c}\in A\mid \eta_{Z(\eta)})
=\Pi_{\lambda_{Z(\eta)^c}}(A),\quad \BP\text{-a.s.\ on $\{\eta(Z(\eta))<\infty\}$}.
\end{align}
\end{theorem}
\begin{proof} Let $g\colon\bN\times\bN\to\R_+$ be measurable and $k\in\N$.
Given $x_1,\dots,x_k\in\BX$ we write $\mathbf{x}=(x_1,\dots,x_k)$ and 
$\delta_{\mathbf{x}}:=\delta_{x_1}+\cdots+\delta_{x_k}$.
We have that (see e.g.\ equation (4.19) in \cite{LastPenrose17})
\begin{align*}
I:=\BE[\I\{\eta(Z(\eta))=k\}g(\eta_{Z(\eta)},\eta_{Z(\eta)^c})]
=\frac{1}{k!}\BE\bigg[\int g(\delta_{\mathbf{x}},\eta_{Z(\eta)^c})\I\{\eta_{Z(\eta)}=\delta_{\mathbf{x}}\}
\,\eta^{(k)}(\mathrm{d}\mathbf{x})\bigg].
\end{align*}
By Lemma \ref{lA1} we have that $\eta_{Z(\eta)}=\delta_{\mathbf{x}}$
iff $\eta_{Z(\delta_{\mathbf{x}})}=\delta_{\mathbf{x}}$ in which case
$Z(\eta)=Z(\delta_{\mathbf{x}})$. Therefore,
\begin{align*}
I&=\frac{1}{k!}\BE\bigg[\int g(\delta_{\mathbf{x}},\eta_{Z(\delta_\mathbf{x})^c})
\I\{\eta_{Z(\delta_{\mathbf{x}})}=\delta_{\mathbf{x}}\}
\,\eta^{(k)}(\mathrm{d}\mathbf{x})\bigg]\\
&=\frac{1}{k!}\BE\bigg[\int g(\delta_{\mathbf{x}},(\eta+\delta_{\mathbf{x}})_{Z(\delta_\mathbf{x})^c})
\I\{(\eta+\delta_{\mathbf{x}})_{Z(\delta_{\mathbf{x}})}=\delta_{\mathbf{x}}\}
\,\lambda^{k}(\mathrm{d}\mathbf{x})\bigg],
\end{align*}
where we have used the the multivariate Mecke equation \eqref{e:mecke}
and the required measurability of the mappings
$(\mu,\mathbf{x})\mapsto \I\{\mu=\delta_{\mathbf{x}}\}$, $\mathbf{x}\mapsto \delta_\mathbf{x}$ and $(\mu, \mathbf{x})\mapsto \mu_{Z(\delta_\mathbf{x})^c}$ follows
from the Borel property of $\BX$ and Lemma \ref{lapp}. 
Since $(\eta+\delta_{\mathbf{x}})_{Z(\delta_{\mathbf{x}})}=\delta_{\mathbf{x}}$
iff $\mathbf{x}\in Z(\delta_{\mathbf{x}})$ and
$\eta(Z(\delta_{\mathbf{x}}))=0$ we can use the complete independence property
of $\eta$ (and Fubini's theorem) to obtain that
\begin{align}\label{eA12}
I=\frac{1}{k!}\BE\bigg[\iint g(\delta_{\mathbf{x}},\psi)
\I\{\mathbf{x}\in Z(\delta_{\mathbf{x}}),\eta(Z(\delta_{\mathbf{x}}))=0\}
\,\Pi_{\lambda_{Z(\delta_{\mathbf{x}})^c}}(\mathrm{d}\psi)\,\lambda^{k}(\mathrm{d}\mathbf{x})\bigg].
\end{align}
Reversing the preceding arguments yields
\begin{align*}
I=\BE\bigg[\int\I\{\eta(Z(\eta))=k\}g(\eta_{Z(\eta)},\psi)
\,\Pi_{\lambda_{Z(\eta)^c}}(\mathrm{d}\psi)\bigg].
\end{align*}
A simplified version of the preceding proof shows that
this remains true for $k=0$. Therefore
\begin{align*}
\BE[\I\{\eta(Z(\eta))<\infty\}g(\eta_{Z(\eta)},\eta_{Z(\eta)^c})]
=\BE\bigg[\int\I\{\eta(Z(\eta))<\infty\}g(\eta_{Z(\eta)},\psi)
\,\Pi_{\lambda_{Z(\eta)^c}}(\mathrm{d}\psi)\bigg].
\end{align*}
Taking $g$ of product form, this yields \eqref{eMarkov},
provided that $\Pi_{\lambda_{Z(\eta)^c}}(A)$ depends measurably on $\eta_{Z(\eta)}$.
Since $Z(\eta)=Z(\eta_{Z(\eta)})$ this follows from
the measurability of $\mu\mapsto \Pi_{\lambda_{Z(\mu)^c}}(A)$,
a fact that can be verified as in the the proof of \cite[Lemma 13.4]{LastPenrose17}.
Indeed, it suffices to take
$$
A:=\{\psi\in\bN:\psi(C_1)=m_1,\dots,\psi(C_n)=m_n\},
$$
where $C_1,\ldots,C_n\in\cX$ are pairwise disjoint and
$m_1,\ldots,m_n\in\N_0$. Then
$$
\Pi_{\lambda_{Z(\mu)^c}}(A)=\prod^n_{i=1}\frac{g_i(\mu)^{m_i}}{m_i!}\exp[-g_i(\mu)],
$$
where $g_i(\mu):=\int \I\{x\in C_i,\, x\in Z(\mu)^c\}\,\lambda(dx)$.
Since $(x,\mu)\mapsto \I\{x\notin Z(\mu)\}$ is measurable,
the functions $g_i$ are measurable.
\end{proof}
\begin{remark}\label{rassumptions}\rm The proof of Theorem \ref{tAstopping}
shows that the full stopping set property \eqref{etr5} is not required
to conclude \eqref{eMarkov}.
We only need that
$Z\colon \bN\to\cX$ is graph-measurable and that the following holds
for all $\mu,\psi\in\bN$ with $\psi(\BX)<\infty$.
We have $\mu_{Z(\mu)}=\psi$ if and only if
$\mu_{Z(\psi)}=\psi$. In this case $Z(\mu)=Z(\psi)$.
\end{remark}
\bigskip

It is worth mentioning that \eqref{eA12} yields
a formula for the distribution of $\eta_{Z(\eta)}$.
For a constant mapping $Z$ this reduces to an
elementary property of a Poisson process.

\begin{proposition}\label{p:ruse} Suppose that $Z$ is a stopping set. Then
\begin{align*}
\BP(\eta(Z(\eta))<\infty,\eta_{Z(\eta)}\in\cdot)=
\sum^\infty_{k=0}\frac{1}{k!}\int\I\{\delta_{\mathbf{x}}\in\cdot\}
\I\{\mathbf{x}\in Z(\delta_{\mathbf{x}})\}\exp[-\lambda(Z(\delta_{\mathbf{x}}))]
\,\lambda^{k}(\mathrm{d}\mathbf{x}),
\end{align*}
where the summand for $k=0$ is interpreted as
$\I\{{\bf 0} \in\cdot\}\exp[-\lambda(Z({\bf 0}))]$, where $\bf{0}$ stands for the zero measure.
\end{proposition}

Next we extend Theorem \ref{tAstopping} so as to allow for
$\eta(Z(\eta))=\infty$.

\begin{theorem}\label{tAstopping2} 
Suppose that $Z_1,Z_2,\ldots$ are stopping sets such that
$Z_n\uparrow Z$  for some $Z\colon\bN\to\cX$.
Then $Z$ is a stopping set. If, moreover,
$\BP(\eta(Z_n(\eta))<\infty)=1$ for each $n\in\N$ then
we have for each $A\in\mathcal{N}$ that
\begin{align}\label{eMarkov2}
\BP(\eta_{Z(\eta)^c}\in A\mid \eta_{Z(\eta)})
=\Pi_{\lambda_{Z(\eta)^c}}(A),\quad \BP\text{-a.s.}
\end{align}
\end{theorem}
\begin{proof} Graph-measurability of $Z$ is obvious.
To check \eqref{estopset} we take $\mu,\psi\in\bN$ such that
$\psi(Z(\mu))=0$. Then
\begin{align*}
Z(\mu_{Z(\mu)}+\psi)&=\bigcup^\infty_{n=1} Z_n(\mu_{Z(\mu)}+\psi)
=\bigcup^\infty_{n=1} Z_n(\mu_{Z_n(\mu)} + \mu_{Z(\mu)\setminus Z_n(\mu)}+\psi)\\
&=\bigcup^\infty_{n=1} Z_n(\mu_{Z(\mu)})
=Z(\mu_{Z(\mu)}),
\end{align*}
concluding the proof of \eqref{estopset}.

To prove \eqref{eMarkov2} we need to show 
for each bounded and measurable $f\colon\bN\times\bN\to\R_+$
that
\begin{align}\label{eA22}
\BE[f(\eta_{Z(\eta)},\eta_{\BX\setminus Z(\eta)})]
=\BE\bigg[ \int f(\eta_{Z(\eta)},\mu_{\BX\setminus Z(\eta)})\,\Pi_\lambda(d\mu)\bigg].
\end{align}
By a monotone class argument we can assume that there exists $B\in\cX_0$
such that $f(\mu,\nu)=f(\mu_B,\nu_B)$ for all $\mu,\nu\in\bN$.
It follows from \eqref{eMarkov} (and a monotone class argument) that
\begin{align}\label{eA23}
\BE[f(\eta_{Z_n(\eta)},\eta_{\BX\setminus Z_n(\eta)})]
=\BE\bigg[ \int f(\eta_{Z_n(\eta)},\mu_{\BX\setminus Z_n(\eta)})\,\Pi_\lambda(d\mu)\bigg].
\end{align}
Since $\eta(B)<\infty$ we have that
$f(\eta_{Z_n(\eta)},\eta_{\BX\setminus Z_n(\eta)})\to f(\eta_{Z(\eta)},\eta_{\BX\setminus Z(\eta)})$
with respect to the discrete topology.
By bounded convergence the left-hand side of 
\eqref{eA23} tends to the left-hand side of  \eqref{eA22}.
For the right-hand sides we note that bounded convergence
implies for each $\nu\in\bN$ that
\begin{align*}
\lim_{n\to\infty}\int f(\nu_{Z_n(\nu)},\mu_{\BX\setminus Z_n(\nu)})\,\Pi_\lambda(d\mu)
=\int f(\nu_{Z(\nu)},\mu_{\BX\setminus Z(\nu)})\,\Pi_\lambda(d\mu).
\end{align*}
Again by bounded convergence the right-hand side of 
\eqref{eA23} tends to the right-hand side of  \eqref{eA22}.
\end{proof}

Property \eqref{eMarkov2} is sometimes referred to as the {\em Markov property} of $\eta$
(see again \cite[Theorem 4]{Rozanov} and \cite{Zuyev06}).
To express it in a different way, we let $\eta'$ be an independent copy of
$\eta$. Then \eqref{eMarkov2} is equivalent to 
\begin{align}\label{eMarkov3}
\eta\overset{d}{=}\eta_{Z(\eta)} + \eta'_{\BX\backslash Z(\eta)},
\end{align}
where the equality in distribution is in the sense of point processes.

\begin{remark}\label{rA1}\rm
The preceding results can be generalized as follows.
Suppose that $(\BM,\mathcal{M})$ is another Borel space.
Consider the product $\BX\times\BM$ equipped with
the product $\sigma$-field and some localizing ring.
For each measure $\nu$ on $\BX\times\BM$ and each $B\in\mathcal{X}$ we denote
by $\nu_B:=\nu_{B\times \BM}$ the restriction of $\nu$ to 
$B\times \BM$. We say that a graph-measurable mapping $\tilde Z\colon \bN(\BX \times \BM) \to\cX$ 
is a {\em stopping set} on $\BX$ if \eqref{estopset} holds.
Then Theorems \ref{tAstopping} and \ref{tAstopping2} remain true in an obviously modified 
form and essentially unchanged proofs.
\end{remark}

\begin{remark}\label{rA2}\rm Let $\tilde{Z}$ be a stopping set
as in Remark \ref{rA1}. Then $Z:=\tilde{Z}\times\BM$
is graph-measurable and a stopping set on $(\BX\times\BM, \cX\otimes\mathcal{M})$,  in the sense
of \eqref{estopset}.
\end{remark}

A weaker version of the next result (corresponding to equation \eqref{e2.54} below) 
was proved in \cite[Theorem 3]{Zuyev99};
see also \cite[formula (3.3)]{Privault15}.

\begin{proposition}\label{pA22} Suppose that $Z,Z_1,Z_2,\ldots$ are stopping sets 
such that $\BP(\eta(Z_n(\eta))<\infty)=1$ for each $n\in\N$ and 
$Z_n\uparrow Z$. Then
\begin{align}\label{estableae}
\I\{x\in Z(\eta)\}=\I\{x\in Z(\eta+\delta_x)\},\quad 
\BP\text{-a.s.},\, \lambda\text{-a.e.\ $x$}.
\end{align}
\end{proposition}
\begin{proof} Let $B\in\cX_0$. Then
\begin{align}
\BE[\eta(B\cap Z(\eta) )] 
&=\lambda(B)-\BE[\BE[\eta(B\cap Z(\eta)^c)\mid \eta_{Z(\eta)}]] \nonumber \\
&=\lambda(B)-\BE\bigg[\int\mu(B)\,\Pi_{\lambda_{Z(\eta)^c}}(d\mu)\bigg] \nonumber \\
\label{e:number_stopset} &=\lambda(B)-\BE [\lambda_{Z(\eta)^c}(B)]=\BE [\lambda(B\cap Z(\eta))],
\end{align}
where we have used \eqref{eMarkov2} for the second identity.
On the other hand we obtain from the Mecke equation \eqref{e:mecke} that
\begin{align*}
\BE[\eta(B\cap Z(\eta))]
=\BE\bigg[\int\I\{x\in B\}\I\{x\in Z(\eta+\delta_x)\}\,\lambda(dx)\bigg].
\end{align*}
Therefore,
\begin{align*}
\int\I\{x\in B\}\BP(x\in Z(\eta))\,\lambda(dx)
=\int\I\{x\in B\}\BP(x\in Z(\eta+\delta_x))\,\lambda(dx),
\end{align*}
and since $B$ is arbitrary this shows that
\begin{align}\label{e2.54}
\BP(x\in Z(\eta))=\BP(x\in Z(\eta+\delta_x)),\quad \lambda\text{-a.e.\ $x$}.
\end{align}
Take $\mu\in\bN$ and $x\in\BX$ such that $x\notin Z(\mu+\delta_x)$.
By \eqref{estopset},
\begin{align*}
Z(\mu+\delta_x)=Z((\mu+\delta_x)_{Z(\mu+\delta_x)}+\mu_{Z(\mu+\delta_x)^c})
=Z(\mu_{Z(\mu+\delta_x)}+\mu_{Z(\mu+\delta_x)^c})=Z(\mu),
\end{align*}
so that $x\notin Z(\mu)$. Therefore $\{x\in Z(\eta)\}\subset \{x\in Z(\eta+\delta_x)\}$
and \eqref{e2.54} implies the asserted result \eqref{estableae}.
\end{proof}

{
\subsection{Non-attainable stopping sets}\label{appendix2}

For the rest of the section, we fix a locally finite measure $\lambda$ on $(\BX, \cX)$ and denote by $\eta$ a Poisson process on $\BX$ with intensity $\lambda$. Recall from Remark \ref{r:attainable} that a stopping set $Z$ is said to be $\lambda$-{\em attainable} if $Z = \cup_t Z_t$ for some $\lambda$-continuous (that is, verifying \eqref{ea1}) CTDT $\{Z_t\}$. The fact that the OSSS inequality \eqref{e3.42} has only been proved for $\lambda$-attainable stopping sets begs the question of whether there exist stopping sets $Z$ that are not $\lambda$-attainable (and therefore for which the validity of the OSSS inequality is not established). The principal aim of this section is to answer positively to this question by proving the next statement.

\begin{proposition}[Existence of non-attainable stopping sets]\label{p:nonatt} Let $Z$ be a stopping set such that $\lambda(Z(\mu))>0$ for every $\mu\in \bN$. Assume that there exists a measurable partition $\{C_i : i=1,...,m \}$ of $\BX$ such that, for every $i$, $\BP(\lambda(C_i\cap Z(\eta)) = 0) >0$. Then, there is no $\lambda$-continuous CTDT $\{Z_t\}$ such that $Z = \cup_t Z_t $.
\end{proposition}

\begin{example}
\label{ex:nonattan}
 {\rm The following example is inspired by the discussion contained in \cite[Example 1.28]{WernerPowell}, that was brought to our attention by Laurin Koehler-Schindler. Let $\lambda(\BX)\in (0,\infty)$  and assume that $\BX$ equals the disjoint union of three measurable sets $C_1,C_2,C_3$ such that $\lambda(C_i)>0$, $i=1,2,3$. For $\mu\in \bN$, set $X_i(\mu):=\I\{\mu(C_i)>0\}$ for $i = 1,2,3$. We define
\begin{align*}
Z(\mu):=
\begin{cases}
C_1\cup C_2\cup C_3 = \BX,& \text{if $X_1(\mu)=X_2(\mu)=X_3(\mu)$},\\
C_1\cup C_2,& \text{if $X_1(\mu)=1,X_2(\mu)=0$},\\
C_1\cup C_3,& \text{if $X_1(\mu)=0,X_3(\mu)=1$},\\
C_2\cup C_3,& \text{if $X_2(\mu)=1, X_3(\mu)=0$}.
\end{cases}
\end{align*}
Then, it is easily seen that $Z$ is a stopping set such that $\lambda(Z)>0$, and $\BP(\lambda(C_i\cap Z(\eta)) = 0)>0$ for every $i=1,2,3$.
}
\end{example}

The proof of Proposition \ref{p:nonatt} is based on a non-trivial zero-one law for CTDTs, stated in the forthcoming Theorem \ref{t:zero}. We first prove an ancillary result.

\begin{lemma}\label{l:tau} Suppose that $\{Z_t:t\in\R_+\}$ is a {$\lambda$-continuous} CTDT.
Define the mapping $\tau\colon\bN\to [0, \infty]$ by
\begin{align} \label{e:tau}
\tau(\mu):=\inf \{t\ge 0: \lambda(Z_t(\mu))>0\},\quad \mu\in \bN,
\end{align}
with the usual convention $\inf\emptyset := +\infty$. Then, writing $Z_\infty(\mu) := \cup_t Z_t(\mu)$, one has that $\mu\mapsto Z^u(\mu):=Z_{\tau(\mu)+u}(\mu)$ is a stopping set for every $u\geq 0$. Moreover, one has necessarily that $ \lambda( Z^0(\mu) )  = 0$, for every $\mu \in \bN$. 
\end{lemma}
\begin{proof} Given $\mu\in\bN$ and $t\ge 0$ we have $\tau(\mu)< t$ if and only if
$\lambda(Z_t(\mu))>0$. Hence, $\tau$ is measurable and $Z^u$ is graph-measurable since, for every $(\mu, x)\in \bN\times \BX$, the quantity ${\bf 1}\{ x\in Z^u(\mu)\}$ is the limit of
\begin{eqnarray*} 
{\bf 1}\{\tau(\mu) \!=\! +\infty, \, x\in Z_\infty\} \!+\! {\bf 1}\{\tau(\mu) \!=\! 0, \, x\in Z_u\} \!+\!\sum_{k=0}^\infty {\bf 1}\left\{\tau(\mu) \!\in\! \left(\frac{k}{n}, \frac{k+1}{n} \right],  \, x\in Z_{\frac{k+1}{n}+u} \right\},
\end{eqnarray*}
as $n\to\infty$, where we have used \eqref{etr2}. It is also immediately checked that $\lambda(Z^0(\mu)) = 0$, because of the $\lambda$-continuity property \eqref{ea1}.
Now take $u\geq 0$ and $\mu,\psi\in\bN$ such that $\psi(Z^u(\mu))=0$; we claim that
\begin{align}\label{432}
\tau(\mu) =\tau(\mu_{Z^u(\mu)}+\psi).
\end{align}
Reasoning as in the first part of the proof of Theorem \ref{tAstopping2}, one sees that \eqref{432} is true whenever $\tau(\mu) = +\infty$. Now assume $\tau(\mu)<\infty$; we will prove \eqref{432} for $u=0$, from which the general case follows at once. Abbreviate $s:=\tau(\mu)$, in such a way that $Z_s(\mu) = Z^0(\mu)$. 
 Since $\lambda(Z_s(\mu)) =0$ and $Z_s(\mu)=Z(\mu_{Z_s(\mu)}+\psi)$
we obtain that $\tau(\mu_{Z^0(\mu)}+\psi)\ge s$.
On the other hand, we have for each $t>s$ that
$\lambda(Z_t(\mu))>0$. Since the $Z_t$ are $\cX_0$-valued
and $\mu,\psi$ are locally finite, we can exploit the right-continuity \eqref{etr2}
of a CTDT to find a $t>s$ satisfying
\begin{align*}
\mu(Z_t(\mu)\setminus Z_s(\mu))=\psi(Z_t(\mu)\setminus Z_s(\mu))=0. 
\end{align*}
But then, $\psi(Z_t(\mu)) = 0$ and
\begin{align*}
\lambda(Z_t(\mu_{Z_s(\mu)})) = \lambda(Z_t(\mu_{Z_t(\mu)}+\psi)) = \lambda(Z_t(\mu) )>0,
\end{align*}
which shows that $\tau(\mu_{Z^0(\mu)}+\psi)\le t$ for every $t>s$ such that $t-s$ is sufficiently small, yielding \eqref{432}. To conclude, set $\mu':=\mu_{Z^u(\mu)}+\psi$, with $\psi(Z^u(\mu)) = 0$. By \eqref{432} we have $\tau(\mu')=\tau(\mu)$.
Therefore,
\begin{align*}
Z^u(\mu_{Z^u(\mu)}+\psi)=Z_{\tau(\mu') +u}(\mu')=Z_{\tau(\mu)+u}(\mu')
=Z_{\tau(\mu)+u}(\mu_{Z^u(\mu)}+\psi)=Z_{\tau(\mu)+u}(\mu)=Z^u(\mu).
\end{align*}
Hence, $Z^u$ is a stopping set.
\end{proof}

The next statement is the most important result of the section.

\begin{theorem}[Zero-one laws for CTDT]\label{t:zero} Let $\{Z_t:t\in\R_+\}$ be a {$\lambda$-continuous} CTDT, and define the mapping $\tau(\cdot)$ as in \eqref{e:tau}. Assume that $\tau(\mu)<\infty$ for every $\mu\in \bN$ and write
\begin{equation}\label{e:gfilt}
\mathcal{G}_u := \sigma(\eta_{Z^u(\eta)}), \quad u\geq 0,
\end{equation}
where the stopping sets $Z^u$ are defined in the statement of Lemma \ref{l:tau}. Then $\mathcal{G}_u\subset \mathcal{G}_v$ for every $u<v$. Moreover, $\eta_{Z^0(\eta)}$ is equal to the zero measure with probability one, and the $\sigma$-field
$$
\mathcal{G}_{0+} := \bigcap_{u>0} \mathcal{G}_u,
$$
is $\BP$-trivial, that is: for all $B\in \mathcal{G}_{0+}$, $\BP(B)\in \{0,1\}$. 
\end{theorem}
\begin{proof} The first part of the statement follows immediately from Lemma \ref{lapp} and the fact that, for every $u<v$, $Z^u$ and $Z^v$ are stopping sets such that $Z^u(\mu)\subset Z^v(\mu)$ for all $\mu \in \bN$. The fact that $\eta_{Z^0(\eta)}$ equals the zero measure {\bf 0}, $\BP${\rm -a.s.}, is a direct consequence of Proposition \ref{p:ruse} and of the fact that $\lambda(Z^0({\bf 0})) =0$, by virtue of Lemma \ref{l:tau}. To show the triviality of $\mathcal{G}_{0+}$, it is enough to show that, for every bounded $\sigma(\eta)$-measurable random variable $Y$, 
\begin{equation}\label{e:triv}
\BE[ Y \, |\, \mathcal{G}_{0+}] = \BE[ Y ], \quad \BP\mbox{\rm -a.s.}\;
\end{equation}
indeed, if \eqref{e:triv} is in order then, for every $B\in \mathcal{G}_{0+}$, one has that ${\bf 1}_B = \BP(B)$, $\BP${\rm -a.s.}, from which the triviality follows. By a monotone class argument, it is sufficient to prove \eqref{e:triv} for every $Y = \exp\{ i \int f d\eta\}$, where $f$ is a bounded deterministic kernel, whose support is contained in some $A\in \cX_0$. Now select a sequence $\epsilon_n\downarrow 0$. By the backwards martingale convergence theorem, one has that
$$
\BE[ Y \, |\, \mathcal{G}_{0+}] = \lim_{n\to \infty}  \BE[ Y \, |\, \mathcal{G}_{\epsilon_n}], \quad \BP\mbox{\rm -a.s.}.
$$
Exploiting the Markov property \eqref{eMarkov2} one has also that, for $\epsilon >0$,
$$
\BE[ Y \, |\, \mathcal{G}_{\epsilon}] = \exp\left\{ i \int_{Z^\epsilon(\eta)} f d\eta \right\} \times \int \exp\left\{ i \int_{\BX\backslash Z^\epsilon(\eta)} f d\mu\right\} \Pi_\lambda(d\mu), 
$$
and the right-hand side of the previous equality converges $\BP${\rm -a.s.} to 
$$
\int \exp\left\{ i \int_{\BX} f d\mu\right\} \Pi_\lambda(d\mu)  =\BE[Y],
$$
as $\epsilon \downarrow 0$, since $Z^\epsilon(\mu)\downarrow Z^0(\mu)$, for all $\mu \in \bN$, and $\int_{Z^0(\eta)} f d\eta =0 =\lambda(Z^0(\eta))$, $\BP${\rm -a.s.}. 
\end{proof}

\begin{proof}[Proof of Proposition \ref{p:nonatt}] The proof is by contradiction. Assume that $\{Z_t : t\in \R_+\}$ is a $\lambda$-continuous CTDT such that $Z = \cup_t Z_t$, and observe that the assumptions on $Z$ imply that $\tau(\mu)<\infty$ for every $\mu\in \bN$. For every $\mu\in \bN$, we define
$$
\mathcal{U}(\mu) = \{A\in \mathcal{X} : \lambda( A\cap Z^\epsilon(\mu) )>0, \,\, \forall \epsilon >0\}.
$$
It is readily checked that, for every $D\in \mathcal{X}$, the event $\{ D\in \mathcal{U}(\eta)\}$ is in $\mathcal{G}_{0+}$, and therefore $\BP[D\in \mathcal{U}(\eta)] \in \{0,1\}$, by virtue of Theorem \ref{t:zero}. By the definition of $\tau(\eta)$, there exists at least one index $i\in \{1,...,m\}$ such that $\BP[C_i \in \mathcal{U}(\eta)]>0$, and therefore $\BP[C_i \in \mathcal{U}(\eta)] = 1$. But this is absurd, since then
$$
0 = \BP[C_i \notin \mathcal{U}(\eta)]\geq \BP[\lambda(C_i\cap Z(\eta) )= 0] >0,
$$
where we have used the assumptions on $C_i$ in the statement.
\end{proof}
}

Example \ref{ex:nonattan} gives explicitly a stopping set that is not $\lambda$-attainable by any CTDT. However, one may wonder whether there are stopping sets that are not $\lambda$-attainable by any randomized CTDT as defined in Section \ref{s:variants_OSSS}. We will now show that this is true for the stopping set in Example \ref{ex:nonattan}. 

\begin{example}
\label{ex:nonattan_rand}{\rm
Let $Z$ be the stopping set defined in Example \ref{ex:nonattan}. Suppose that $\{Z_t = Z_t^Y : t\in \R_+\}$ is a $\lambda$-continuous randomized CTDT for an independent random variable $Y$ such that $Z = \cup_t Z_t$. Define $\mathcal{U}(\eta) = \mathcal{U}^Y(\eta)$ as in the proof of Proposition \ref{p:nonatt}. Without loss of generality assume that $\BP(C_1 \in \mathcal{U}(\eta)) > 0$. 

Observe that $\lambda(Z^0(\eta)) = 0$ for $Z^0$ as defined in Lemma \ref{l:tau} and so as in the proof of Theorem \ref{t:zero}, we have that $\eta(Z^0(\eta)) = {\bf 0}$, $\BP$-a.s.. Thus, we derive that, $\BP$-a.s.,
\begin{align}
\BP(\eta(C_2) > 0, \eta(C_3) = 0 \,  | \, \eta_{Z^0(\eta)} ) & = \BP(\eta(C_2 \setminus Z^0(\eta)) > 0, \eta(C_3 \setminus Z^0(\eta)) = 0 \,  | \,   \eta_{Z^0(\eta)}) \nonumber  \\
\label{e:counterexrandctdt} & = (1 - e^{-\lambda(C_2)})e^{-\lambda(C_3)} =: \kappa
\end{align}
where in deriving the second equality, we have used the strong Markov property \eqref{eMarkov2} and the complete independence property of the Poisson process as well as that $Z^0(\eta)$ is  a stopping set with $\lambda(Z^0(\eta)) = 0$. Now, for any $\epsilon > 0$, using \eqref{e:counterexrandctdt}, we derive that
\begin{align*}
 0 & <\kappa\BP(C_1 \in \mathcal{U}(\eta))  \leq \kappa \BP(C_1 \cap Z^{\epsilon}(\eta) \neq \emptyset) \\
 & \leq \BE[\BP(C_1 \cap Z^{\epsilon}(\eta) \neq \emptyset \,  | \,    \eta_{Z^0(\eta)})\BP(\eta(C_2 \setminus Z^0(\eta)) > 0, \eta(C_3 \setminus Z^0(\eta)) = 0\,  | \,   \eta_{Z^0(\eta)} ) ] \\
 & =  \BE[\BP(C_1 \cap Z^{\epsilon}(\eta) \neq \emptyset, \eta(C_2) > 0, \eta(C_3) = 0 \,  | \,   \eta_{Z^0(\eta)} ) ] \\
& = \BP(C_1 \cap Z^{\epsilon}(\eta) \neq \emptyset, \eta(C_2) > 0, \eta(C_3) = 0) \\
 & \leq \BP(C_1 \cap Z \neq \emptyset, Z = C_2 \cup C_3 ) = 0,
 \end{align*}
where in the first equality we have again used the strong Markov property \eqref{eMarkov2} and the complete independence property of the Poisson process, as well as the fact that $Z^0(\eta)$ is  a stopping set. Thus, we cannot realize $Z$ by a randomized CTDT as well. }
 
\end{example}

\end{appendix}

\begin{acks}[Acknowledgments]
The authors thank Sergei Zuyev for sharing an earlier version of his continuum noise sensitivity notes.  We are thankful to Raphael Lachi\'{e}ze-Rey for discussions regarding CTDT and suggesting some simplified CTDT constructions. We are thankful to Laurin Koehler-Schindler, Stephen Muirhead, Vincent Tassion and Hugo Vanneuville for some comments regarding the discrete OSSS and Schramm-Steif inequalities.  We thank Srikanth K. Iyer for discussing and pointing out percolation theory literature related to RSW-type estimates and comments on a preliminary draft.  
\end{acks}

\begin{funding} DY's research was supported by DST-INSPIRE faculty award and CPDA from the Indian Statistical Institute.  GP's research was supported by the FNR grant FoRGES (R-AGR3376-10) at Luxembourg University.  The work significantly benefitted from reciprocal visits to our respective institutions and we wish to thank all the three institutions for hosting us and supporting our visits. 
\end{funding}

\end{document}